\documentclass{article}
\usepackage{amsmath,color,amsfonts,amsthm}
\usepackage{amssymb,enumerate,verbatim,enumitem}
\usepackage{hyperref}
\usepackage[bottom,hang,flushmargin]{footmisc}
\usepackage{tikz}
\usetikzlibrary{calc}

\usepackage{a4wide}
\usepackage{graphicx}
\newtheorem{lemma}{Lemma}[section]
\newtheorem{corollary}[lemma]{Corollary}
\newtheorem{claim}[lemma]{Claim}
\newtheorem{theorem}[lemma]{Theorem}

\newtheorem{definition}[lemma]{Definition}
\newtheorem{conjecture}[lemma]{Conjecture}

\theoremstyle{definition}
\newtheorem{defn}[lemma]{Definition}

\global\long\def\E{\mathbb{E}}
\global\long\def\R{\mathbb{R}}

\global\long\def\eps{\epsilon}
\global\long\def\P{\mathbb{P}}
\global\long\def\subset{\subseteq}

\newcommand{\Remark}[1]{}

\newcommand{\LL}{{\ll}}

\newcommand\subsubsectionme[1]{\bigskip \noindent\textbf{#1}
\smallskip

}

\definecolor{darkgreen}{rgb}{0.1,0.7,0.1}

\date{}

\let\oldgg\gg
\renewcommand{\gg}{\stackrel{\scriptscriptstyle{\text{\sc poly}}}{\oldgg}}

\let\oldll\ll
\renewcommand{\ll}{\stackrel{\scriptscriptstyle{\text{\sc poly}}}{\oldll}}

\newcounter{propcounter}

\title{\vspace{-0.7cm} A proof of Ringel's Conjecture}
\author{R. Montgomery\thanks{School of Mathematics,
University of Birmingham,
Birmingham,
B15 2TT,
UK. r.h.montgomery@bham.ac.uk.}
, A. Pokrovskiy\thanks{Department of Economics, Mathematics, and Statistics, Birkbeck College, University of London. dr.alexey.pokrovskiy@gmail.com.}
, and B. Sudakov\thanks{Department of Mathematics, ETH, 8092 Zurich, Switzerland. benjamin.sudakov@math.ethz.ch. Research supported in part by SNSF grant 200021-175573.}
}

\begin{document}
\maketitle

\begin{abstract}
A typical decomposition question asks whether the edges of some graph $G$ can be partitioned into disjoint copies of another graph $H$.
One of the oldest and best known conjectures in this area, posed by Ringel in 1963, concerns the decomposition of complete graphs into edge-disjoint copies of a tree.
It says that any tree with $n$ edges packs $2n+1$ times into the complete graph $K_{2n+1}$.
In this paper, we prove this conjecture for large $n$.
\end{abstract}

\section{Introduction}\label{intro}
The study of decomposition problems for graphs and hypergraphs has a very long history, going back more than two hundred years to the work of Euler on Latin squares. Latin squares are $n \times n$  arrays filled with $n$ symbols such that each symbol appears once in every row and column.
In 1782, Euler asked for which values of $n$ there is a Latin square which can be decomposed into $n$ disjoint transversals, where a transversal is a collection of cells of the Latin square which do not share the same row, column or symbol.
This problem has many equivalent forms. In particular, it is equivalent to a \emph{graph decomposition problem}.

We say that a graph $G$ has a \emph{decomposition into copies of a graph $H$} if the edges of $G$ can be partitioned into edge-disjoint subgraphs isomorphic to $H$. Euler's problem is equivalent to asking for which values of $n$ does the balanced complete $4$-partite graph $K_{n,n,n,n}$ have a decomposition into copies of the complete graph on $4$ vertices, $K_4$.
In 1847, Kirkman studied decompositions of complete graphs $K_n$ and showed that they can be decomposed into copies of a triangle if, and only if, $n\equiv 1 \text{ or } 3 \pmod 6$. Wilson~\cite{wilson1975decompositions} generalized this result by proving necessary and sufficient conditions for a complete graph $K_n$ to be decomposed into copies of \emph{any} graph, for large $n$.
A very old problem in this area, posed in 1853 by Steiner, says that, for every $k$, modulo an obvious divisibility condition every sufficiently large complete $r$-uniform hypergraph can be decomposed into edge-disjoint copies of a complete $r$-uniform hypergraph on $k$ vertices. This problem was the so-called ``existence of designs'' question and has practical relevance to experimental designs. It was resolved only very recently in spectacular
work by Keevash \cite{keevash} (see the subsequent work of \cite{glock-all} for an alternative proof of this result). Over the years graph and hypergraph decomposition problems have been extensively studied and
by now this has become a vast topic with many exciting results and conjectures (see, for example, \cite{gallian2009dynamic,wozniak2004packing,yap1988packing}).

In this paper, we study decompositions of complete graphs into large trees, where a tree is a connected graph with no cycles. By large we mean that the size of the tree is comparable with the size of the complete graph (in contrast with the existence of designs mentioned above, where the decompositions are into small subgraphs).
The earliest such result  was obtained more than a century ago by Walecki. In 1882 he proved that a complete graph $K_n$ on an even number of vertices can be partitioned into edge-disjoint Hamilton paths.
A Hamilton path is a path which visits every vertex of the parent graph exactly once. Since paths are a very special kind of tree it is natural to ask which other large trees can be used to decompose a complete graph. This question was raised by Ringel \cite{ringel1963theory}, who in 1963 made the following
appealing conjecture on the decomposition of complete graphs into edge-disjoint copies of a tree with roughly half the size of the complete graph.

\begin{conjecture} \label{ringelconj}
The complete graph $K_{2n+1}$ can be decomposed into copies of any tree with $n$ edges.
\end{conjecture}

Ringel's conjecture is one of the oldest and best known open conjectures on graph decompositions. It has been established for many very special classes of trees such as caterpillars, trees with $\leq 4$ leaves, firecrackers,  diameter $\leq 5$ trees,   symmetrical trees, trees with $\leq 35$ vertices, and olive trees (see Chapter 2 of \cite{gallian2009dynamic} and the references therein).
There have also been some partial general results in the direction of Ringel's conjecture. Typically, for these results, an extensive technical method is developed which is capable of almost-packing any appropriately-sized collection of certain sparse graphs, see,
e.g., \cite{bottcher2016approximate, messuti2016packing, ferber2017packing, kim2016blow}.   In particular, Joos, Kim, K{\"u}hn and Osthus~\cite{joos2016optimal} have proved Ringel's conjecture for very large bounded-degree trees. Ferber and Samotij~\cite{ferber2016packing} obtained an almost-perfect packing of almost-spanning trees with maximum degree $O(n/\log n)$, thus giving an approximate version of Ringel's conjecture for trees with maximum degree $O(n/\log n)$. A different proof of this was obtained by
Adamaszek, Allen, Grosu, and Hladk{\'y}~\cite{adamaszek2016almost}, using graph labellings. Allen, B\"ottcher, Hladk{\'y} and Piguet~\cite{allen2017packing} almost-perfectly packed arbitrary spanning graphs with maximum degree $O(n/ \log n)$ and constant degeneracy\footnote{A graph is $d$-degenerate if each induced subgraph has a vertex of degree $\leq d$. Trees are exactly the $1$-degenerate, connected graphs.} into large complete graphs. 
Recently Allen, B\"ottcher, Clemens, and Taraz~\cite{allen2019perfectly} found perfect packings of complete graphs into specified graphs with maximum degree $o(n/\log n)$, constant degeneracy, and linearly many leaves.
To tackle Ringel's conjecture, the above mentioned papers developed many powerful techniques based on the application of probabilistic methods and
Szemer\'edi's regularity lemma. Yet, despite the variety of these techniques, they all have the same limitation, requiring that the maximum degree of the tree should be much smaller than $n$.

A lot of the work on Ringel's Conjecture has used the \emph{graceful labelling} approach.
This is an elegant approach proposed by R\'osa~\cite{rosa1966certain}. For an $(n+1)$-vertex tree $T$ a bijective labelling of its vertices $f: V(T) \rightarrow \{0, \dots, n\}$ is called graceful
if  the values $|f(x)-f(y)|$ are distinct over the edges $(x,y)$ of $T$.
In 1967 R\'osa conjectured that every tree has a graceful labelling. This conjecture has attracted a lot of attention in the last 50 years but has only been proved for some special classes of trees, see e.g., \cite{gallian2009dynamic}.
The most general result for this problem  was obtained by Adamaszek, Allen, Grosu, and Hladk{\'y}~\cite{adamaszek2016almost} who proved it asymptotically for trees with maximum degree $O(n/\log n)$.
The main motivation for studying graceful labellings is that one can use them to prove Ringel's conjecture. Indeed,
given a graceful labelling $f: V(T) \rightarrow \{0, \dots, n\}$, think of it as an embedding  of $T$ into $\{0, \dots, 2n\}$. Using addition modulo $2n+1$, consider $2n+1$ cyclic shifts $T_0, \ldots, T_{2n}$ of $T$, where the tree $T_i$ is an isomorphic copy of $T$ whose vertices are
$V(T_i)=\{f(v)+i~|~ v \in V(T)\}$ and whose edges are $E(T_i)=\{(f(x)+i,f(y)+i)~|~(x,y)\in E(T)\}$. It is easy to check that the fact that $f$ is graceful implies that the trees $T_i$ are edge disjoint and therefore
decompose $K_{2n+1}$.

R\'osa also introduced a related proof approach to Ringel's conjecture called ``$\rho$-valuations''. We describe it using the language of ``rainbow subgraphs'', since this is the language which we ultimately use in our proofs.
A \emph{rainbow} copy of a graph $H$ in an edge-coloured graph $G$ is a subgraph of $G$ isomorphic to $H$ whose edges have different colours. Rainbow subgraphs are important because many problems in combinatorics can be rephrased as problems asking for rainbow subgraphs  (for example the problem of Euler on Latin squares mentioned above).
Ringel's conjecture is implied by the existence of a rainbow copy of every $n$-edge tree $T$  in the following edge-colouring of the complete graph $K_{2n+1}$, which we call the \emph{near distance (ND-)colouring}.
Let $\{0,1,\dots,2n\}$ be the vertex set of $K_{2n+1}$. Colour the edge $ij$ by colour $k$, where $k\in [n]$, if either $i=j+k$ or $j=i+k$ with addition modulo $2n+1$.
Kotzig~\cite{rosa1966certain} noticed that if the ND-coloured $K_{2n+1}$ contains a rainbow copy of a tree $T$, then $K_{2n+1}$ can be decomposed into copies of $T$ by taking $2n+1$ cyclic shifts of the original rainbow copy, as explained above (see also Figure~\ref{FigureIntro}). Motivated by this and Ringel's Conjecture, Kotzig conjectured that the ND-coloured $K_{2n+1}$ contains a rainbow copy of every  tree on $n$ edges.
To see the connection with graceful labellings, observe that such a labelling of the tree $T$ is equivalent to a rainbow copy of this tree in the ND-colouring whose vertices are $\{0, \dots, n\}$. Clearly, specifying exactly the vertex set of the tree adds an additional restriction which makes it harder to find such a rainbow copy.

\begin{figure}[b]

\vspace{-0.4cm}

  \centering
\begin{tikzpicture}

\def\vxrad{0.1cm};
\def\circrad{1.5};
\def\shift{-6};
\def\medline{0.04cm};
\def\tline{0.04cm};
\def\tvxrad{0.075cm};
\def\incircrad{1};

\begin{scope}[shift={(\shift,0)}]

\foreach \x in {0,...,8}
{
\draw coordinate (B\x) at  (40*\x:\incircrad);
}

\draw ($0.9*(B4)+0.9*(B5)$) node {\large $T:$};

\draw [line width=\tline] (B0) -- (B5);
\draw [line width=\tline] (B4) -- (B5);
\draw [line width=\tline] (B3) -- (B5);
\draw [line width=\tline] (B0) -- (B6);

\foreach \x in {0,3,4,5,6}
{
\draw [fill=black] (B\x) circle [radius=\tvxrad];
}

\end{scope}

\foreach \x in {0,...,8}
{
\draw coordinate (A\x) at  (40*\x:\circrad);
}

\draw ($0.8*(A4)+0.8*(A5)$) node {\large $K_9:$};

\foreach \x in {0,...,7}
{
\pgfmathtruncatemacro\y{\x+1};
\draw [line width=\medline,darkgreen!50] (A\x) -- (A\y);
}
\foreach \x in {8}
{
\pgfmathtruncatemacro\y{\x-9+1};
\draw [line width=\medline,darkgreen!50] (A\x) -- (A\y);
}

\foreach \x in {0,...,6}
{
\pgfmathtruncatemacro\y{\x+2};
\draw [line width=\medline,blue!50] (A\x) -- (A\y);
}
\foreach \x in {0,...,5}
{
\pgfmathtruncatemacro\y{\x+3};
\draw [line width=\medline,black!50] (A\x) -- (A\y);
}
\foreach \x in {0,...,4}
{
\pgfmathtruncatemacro\y{\x+4};
\draw [line width=\medline,red!50] (A\x) -- (A\y);
}

\foreach \x in {7,...,8}
{
\pgfmathtruncatemacro\y{\x-9+2};
\draw [line width=\medline,blue!50] (A\x) -- (A\y);
}
\foreach \x in {6,...,8}
{
\pgfmathtruncatemacro\y{\x-9+3};
\draw [line width=\medline,black!50] (A\x) -- (A\y);
}
\foreach \x in {5,...,8}
{
\pgfmathtruncatemacro\y{\x-9+4};
\draw [line width=\medline,red!50] (A\x) -- (A\y);
}

\draw [line width=0.1cm,red] (A0) -- (A5);
\draw [line width=0.1cm,darkgreen] (A4) -- (A5);
\draw [line width=0.1cm,blue] (A3) -- (A5);
\draw [line width=0.1cm,black] (A0) -- (A6);

\foreach \x in {0,...,8}
{
\draw [fill=black] (A\x) circle [radius=\vxrad];
\draw ($1.2*(A\x)$) node {\x};
}

\end{tikzpicture}

\vspace{-0.3cm}
  \caption{The ND-colouring of $K_9$ and a rainbow copy of a tree $T$ with four edges. The colour of each edge corresponds to its Euclidean length.  By taking cyclic shifts of this tree around the centre of the picture we obtain $9$ disjoint copies of the tree decomposing $K_9$ (and thus a proof of Ringel's Conjecture for this particular tree). To see that this gives 9 disjoint trees, notice that edges must be shifted to other edges of the same colour (since shifts are isometries).}\label{FigureIntro}
  
  \vspace{-0.1cm}
  
\end{figure}
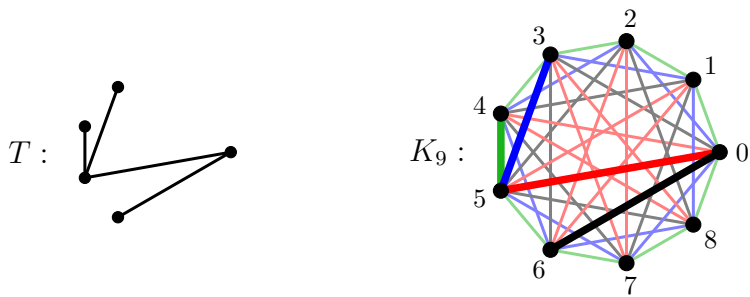

In \cite{MPS} we gave a new approach to embedding large trees (with no degree restrictions) into edge-colourings of complete graphs, and used this to prove Conjecture
\ref{ringelconj} asymptotically. Here, we further develop and refine this approach, combining it with several critical new ideas to prove Ringel's conjecture for large complete graphs.

\begin{theorem}\label{main}
For every sufficiently large $n$ the complete graph $K_{2n+1}$ can be decomposed into copies of any tree with $n$ edges.
\end{theorem}

The proof of Theorem \ref{main} uses the last of the three approaches mentioned above. Instead of  working directly with tree decompositions, or studying graceful labellings, we instead prove for large $n$ that every ND-coloured $K_{2n+1}$ contains a rainbow copy of every $n$-edge tree (see Theorem~\ref{Theorem_Ringel_proof}).
Then, we obtain a decomposition of the complete graph by considering cyclic shifts of one copy of a given tree (as in Figure~\ref{FigureIntro}). The existence of such a cyclic decomposition was separately conjectured by Kotzig~\cite{rosa1966certain}. Therefore, this also gives a proof of the conjecture by Kotzig for large $n$.

 Our proof approach builds on ideas from the previous research on both graph decompositions and graceful labellings.
From the work on graph decompositions, our approach is inspired by randomized decompositions and the absorption technique. The rough idea of absorption is as follows.  Before the embedding of $T$ we prepare a template which has some useful properties. Next we find a partial embedding of the tree $T$ with some vertices removed such that we did not use the edges of the template.
Finally we use the template to embed the remaining vertices. This idea was introduced as a general method by R\"odl, Ruci\'nski and Szemer\'edi \cite{RRS} and has been used extensively since then. For example, the proof of Ringel's Conjecture for bounded degree trees is based on this technique~\cite{joos2016optimal}.

We are also inspired by graceful labellings. When dealing with  trees with very high degree vertices, we use a  completely deterministic approach  for finding a rainbow copy of the tree. This approach heavily relies on features of the ND-colouring and produces something very close to a graceful labelling of the tree.

Our theorem is the first general result giving a perfect decomposition of a graph into subgraphs with arbitrary degrees. As we mentioned, all previous comparable results placed  a bound on the maximum degree of the subgraphs into which they decomposed the complete graph. Therefore, we hope that further development of our techniques can help overcome this ``bounded degree barrier'' in other problems as well.

\section{Proof outline}\label{Section_proof_outline}

From the discussion in the introduction, to prove Theorem~\ref{main} it is sufficient to prove the following result.

\begin{theorem}\label{Theorem_Ringel_proof}
For sufficiently large $n$, every ND-coloured $K_{2n+1}$ has a rainbow copy of every $n$-edge tree.
\end{theorem}

That is, for large $n$, and each $(n+1)$-vertex tree $T$, we seek a rainbow copy of $T$ in the ND-colouring of the complete graph with $2n+1$ vertices, $K_{2n+1}$. Our approach varies according to which of 3 cases the tree $T$ belongs to. For some small $\delta>0$, we show that, for every large $n$, every $(n+1)$-vertex tree falls in one of the following 3 cases (see Lemma~\ref{Lemma_case_division}), where a bare path is one whose internal vertices have degree 2 in the parent tree.

\begin{enumerate}[label = \Alph{enumi}]
\item $T$ has at least $\delta^6 n$ non-neighbouring leaves.
\item $T$ has at least $\delta n/800$ vertex-disjoint bare paths with length $\delta^{-1}$.
\item Removing leaves next to vertices adjacent to at least $\delta^{-4}$ leaves gives a tree with at most
$n/100$ vertices.
\end{enumerate}

As defined above, our cases are not mutually disjoint. In practice, we will only use our embeddings for trees in Case A and B which are not in Case C. In~\cite{MPS}, we developed methods to embed any $(1-\eps)n$-vertex tree in a rainbow fashion into any 2-factorized $K_{2n+1}$, where $n$ is sufficiently large depending on $\eps$. A colouring is a \emph{2-factorization} if every vertex is adjacent to exactly 2 edges of each colour. In this paper, we embed any $(n+1)$-vertex tree $T$ in a rainbow fashion into a specific 2-factorized colouring of $K_{2n+1}$, the ND-colouring, when $n$ is large. To do this, we introduce three key new methods, as follows.

\begin{enumerate}[label = {\bfseries M\arabic{enumi}}]

\item We use our results from \cite{montgomery2018decompositions} to suitably randomize the results of \cite{MPS}. This allows us to randomly embed a $(1-\eps)n$-vertex tree into any 2-factorized $K_{2n+1}$, so that the image is rainbow and has certain random properties. These properties allow us to apply a case-appropriate \emph{finishing lemma} with the uncovered colours and vertices.\label{T2}
\item We use a new implementation of absorption to embed a small part of $T$ while using some vertices in a random subset of $V(K_{2n+1})$ and exactly the colours in a random subset of $C(K_{2n+1})$. This uses different \emph{absorption structures} for trees in Case A and in Case B, and in each case gives the finishing lemma for that case.\label{T1}
\item We use an entirely new, deterministic, embedding for trees in Case C.\label{T3}
\end{enumerate}

 For trees in Cases A and B, we start by finding a random rainbow copy of most of the tree using~\ref{T2}, as outlined in Section~\ref{sec:T2}. Then, we embed the rest of the tree using uncovered vertices and exactly the unused colours using~\ref{T1}, which gives a finishing lemma for each case. These finishing lemmas are discussed in Section~\ref{sec:T1}. We use \ref{T3} to embed trees in Case C, which is essentially independent of our embeddings of trees in Cases A and B. This method is outlined in Section \ref{sec:T3}. In Section~\ref{sec:overhead}, we state our main lemmas and theorems, and prove Theorem~\ref{Theorem_Ringel_proof} subject to these results.

The rest of the paper is structured as follows. Following details of our notation, in Section~\ref{sec:prelim} we recall and prove various preliminary results.  We then prove the finishing lemma for Case A in Section~\ref{sec:finishA} and the finishing lemma for Case B in Section~\ref{sec:finishB} (together giving \ref{T1}). In Section~\ref{sec:almost}, we give our randomized rainbow embedding of most of the tree (\ref{T2}). In Section~\ref{sec:lastC}, we embed the trees in Case C with a deterministic embedding (\ref{T3}). Finally, in Section~\ref{sec:conc}, we make some concluding remarks.

\subsection{\ref{T2}: Embedding almost all of the tree randomly in Cases A and B}\label{sec:T2}
For a tree $T$ in Case A or B, we carefully choose a large subforest, $T'$ say, of $T$, which contains almost all the edges of $T$. We find a rainbow copy $\hat{T}'$ of $T'$ in the ND-colouring of $K_{2n+1}$ (which exists due to~\cite{MPS}), before applying a finishing lemma to extend $\hat{T}'$ to a rainbow copy of $T$. Extending to a rainbow copy of $T$ is a delicate business --- we must use exactly the $n-e(\hat{T})$ unused colours. Not every rainbow copy of $T'$ will be extendable to a rainbow copy of $T$. However, by combining our methods in \cite{montgomery2018decompositions} and \cite{MPS}, we can take a random rainbow copy $\hat{T}'$ of $T'$ and show that it is likely to be extendable to a rainbow copy of $T$. Therefore, some rainbow copy of $T$ must exist in the ND-colouring of $K_{2n+1}$.

As $\hat{T}'$ is random, the sets $\bar{V}:=V(K_{2n+1})\setminus V(\hat{T})$ and $\bar{C}:=C(K_{2n+1})\setminus C(\hat{T})$ will also be random.  The distributions of $\bar{V}$ and $\bar{C}$ will be complicated, but we will not need to know them. It will suffice that there will be large (random) subsets $V\subset \bar{V}$ and $C\subset \bar{C}$ which each do have a nice, known, distribution.
Here, for example, $V\subset V(K_{2n+1})$ has a nice distribution if there is some $q$ so that each element of $V(K_{2n+1})$ appears independently at random in $V$ with probability $q$ --- we say here that $V$ is \emph{$q$-random} if so, and analogously we define a $q$-random subset $C\subset C(K_{2n+1})$ (see Section~\ref{sec:prob}). A natural combination of the techniques in~\cite{montgomery2018decompositions,MPS} gives the following.
\begin{theorem}\label{Sketch_Near_Embedding} For each $\epsilon>0$, the following holds for sufficiently large $n$.
Let $K_{2n+1}$ be $2$-factorized and let $T'$ be a forest on $(1-\epsilon)n$ vertices. Then, there is a randomized subgraph $\hat{T}'$ of $K_{2n+1}$ and random subsets $V\subset V(K_{2n+1})\setminus V(\hat{T}')$ and $C\subset C(K_{2n+1})\setminus C(\hat{T}')$ such that the following hold for some $p:=p(T')$ (defined precisely in Theorem~\ref{nearembedagain}).
\stepcounter{propcounter}
\begin{enumerate}[label = {\bfseries \Alph{propcounter}\arabic{enumi}}]
\item $\hat{T}'$ is a rainbow copy of $T'$ with high probability.\label{woo1}
\item $V$ is $(p+\eps)/6$-random and $C$ is $(1-\epsilon)\eps$-random. ($V$ and $C$ may depend on each other.)\label{woo2}
\end{enumerate}
\end{theorem}

We will apply a variant of Theorem~\ref{Sketch_Near_Embedding} (see Theorem~\ref{nearembedagain}) to subforests of trees in Cases A and B, and there we will have that $p\gg \eps$. Note that, then, $C$ will likely be much smaller than $V$. This reflects that $\hat{T}'\subset K_{2n+1}$ will contain fewer than $n$ out of $2n+1$ vertices, while $C(K_{2n+1})\setminus C(\hat{T}')$ contains exactly $n-e(\hat{T}')$ out of $n$ colours.

As explained in~\cite{MPS}, in general the sets $V$ and $C$ cannot be independent, and this is in fact why we need to treat trees in Case C separately. In order to finish the embedding in Cases A and B, we need, essentially, to find \emph{some} independence between the sets $V$ and $C$ (as discussed below). The embedding is then as follows for some small $\delta$ governing the case division, with $\eps=\delta^6$ in Case A and $\bar{\eps}\oldgg \delta$ in Case B. Given an $(n+1)$-vertex tree $T$ in Case A or B we delete either $\epsilon n$ non-neighbouring leaves (Case A) or $\bar{\epsilon} n/k$ vertex-disjoint bare paths with length $k=\delta^{-1}$ (Case B) to obtain a forest $T'$. Using (a variant of) Theorem~\ref{Sketch_Near_Embedding}, we find a randomized rainbow copy $\hat{T}'$ of $T'$ along with some random vertex and colour sets and apply a finishing lemma to extend this to a rainbow copy of $T$.

After a quick note on the methods in~\cite{montgomery2018decompositions} and~\cite{MPS}, we will discuss the finishing lemmas, and explain why we need some independence, and how much independence is needed.

\subsubsection*{Randomly embedding nearly-spanning trees}
In~\cite{MPS}, we embedded a $(1-\epsilon)n$-vertex tree $T$ into a 2-factorization of $K_{2n+1}$ by breaking it down mostly into large stars and large matchings. For each of these, we embedded the star or matching using its own random set of vertices and random set of colours (which were not necessarily independent of each other). In doing so, we used almost all of the colours in the random colour set, but only slightly less than one half of the vertices in the random vertex set. (This worked as we had more than twice as many vertices in $K_{2n+1}$ than in $T$.)  For trees not in Case C, a substantial portion of a large subtree was broken down into matchings. By embedding these matchings more efficiently, using results from~\cite{montgomery2018decompositions}, we can use a smaller random vertex set. This reduction allows us to have, disjointly from the embedded tree, a large random vertex subset $V$.

More precisely, where $q,\epsilon\oldgg n^{-1}$, using a random set $V$ of $2qn$ vertices and a random set $C$ of $qn$ colours, in~\cite{MPS} we showed that, with high probability, from any set $X\subset V(K_{2n+1})\setminus V$ with $|X|\leq (1-\epsilon)qn$, there was a $C$-rainbow matching from $X$ into $V$. Dividing $V$ randomly into two sets $V_1$ and $V_2$, each with $qn$ vertices, and using the results in~\cite{montgomery2018decompositions}, we can use $V_1$ to find the $C$-rainbow matching (see Lemma~\ref{Lemma_MPS_nearly_perfect_matching}).
Thus, we gain the random set $V_2$ of $qn$ vertices which we do not use for the embedding of $T$, and we can instead use it to extend this to an $(n+1)$-vertex tree in $K_{2n+1}$.
Roughly speaking, if in total $pn$ vertices of $T$ are embedded using matchings, then we gain altogether a random set of around $pn$ vertices.

If a tree is not in Case C, then the subtree/subforest we embed using these techniques has plenty of vertices embedded using matchings, so that in this case we will be able to take $p\geq 10^{-3}$ when we apply the full version of Theorem~\ref{Sketch_Near_Embedding} (see Theorem~\ref{nearembedagain}). Therefore, we will have many spare vertices when adding the remaining $\eps n$ vertices to the copy of $T$. Our challenges are firstly that we need to use exactly all the colours not used on the copy of $T'$ and secondly that there can be a lot of dependence between the sets $V$ and $C$. We first discuss how we ensure that we use every colour.

\subsection{\ref{T1}: Finishing the embedding in Cases A and B}\label{sec:T1}

To find trees using every colour in an ND-coloured $K_{2n+1}$ we prove two \emph{finishing lemmas} (Lemma~\ref{lem:finishA} and~\ref{lem:finishB}).  These lemmas say that, for a given randomized set of vertices $V$ and a given randomized set of colours $C$, we can find a rainbow matching/path-forest which uses exactly the colours in $C$, while using some of the vertices from $V$.  These lemmas are used to finish the embedding of the trees in Cases A and B, where the last step is to (respectively) embed a matching or path-forest that we removed from the tree $T$ to get the forest $T'$. Applying (a version of) Theorem~\ref{Sketch_Near_Embedding} we get a random rainbow copy $\hat{T'}$ of $T'$ and random sets $\bar{V}=V(K_n)\setminus V(\hat{T}')$ and $\bar{C}=C(K_n)\setminus C(\hat{T}')$.

In order to apply the case-appropriate finishing lemma, we need some independence between $\bar{V}$ and $\bar{C}$, for reasons we now discuss for Case A, and then Case B. Next, we discuss the independence property we use and how we achieve this independence. (Essentially, this property is that $\bar{V}$ and $\bar{C}$ contain two small random subsets which are independent of each other.) Finally, we discuss the absorption ideas for Case A and Case B.

\subsubsection*{Finishing with matchings (for Case A)}
In Case A, we take the $(n+1)$-vertex tree $T$ and remove a large matching of leaves, $M$ say, to get a tree, $T'$ say, that can be embedded using Theorem~\ref{Sketch_Near_Embedding}. This gives a random copy, $\hat{T}'$ say, of $T'$ along with random sets $V\subset V(K_{2n+1})\setminus V(\hat{T}')$ and $C\subset C(K_{2n+1})\setminus C(\hat{T}')$
 which are $(p+\eps)/6$-random and $(1-\eps)\eps$-random respectively. For trees not in Case C we will have $p\oldgg \eps$.
 Let $X\subset V(\hat{T}')$ be the set of vertices we need to add neighbours to as leaves to make $\hat{T}'$ into a copy of $T$.

We would like to find a perfect matching from $X$ to $\bar{V}:=V(K_{2n+1})\setminus V(\hat{T}')$ with exactly the colours in $\bar{C}:=C(K_{2n+1})\setminus C(\hat{T}')$, using that $V\subset \bar{V}$ and $C\subset \bar{C}$. (A perfect matching from $X$ to $\bar{V}$ is such a matching covering every vertex in $X$.)
 Unfortunately, there may be some $x\in X$ with no edges with colour in $\bar{C}$ leading to $\bar{V}$. (If $C$ and $V$ were independent, then this would not happen with high probability.) If this happens, then the desired matching will not exist.

In Case B, a very similar situation to this may occur, as discussed below, but in Case A there is another potential problem. There may be some colour $c\in \bar{C}$ which does not appear between $X$ and $\bar{V}$, again preventing the desired matching existing. This we will avoid by carefully embedding a small part of $T'$ so that every colour appears between $X$ and $\bar{V}$ on plenty of edges.

\subsubsection*{Finishing with paths (for Case B)}
In Case B, we take the $(n+1)$-vertex tree $T$ and remove a set of vertex-disjoint bare paths to get a forest, $T'$ say, that can be embedded using Theorem~\ref{Sketch_Near_Embedding}. This gives a random copy, $\hat{T}'$ say, of $T'$ along with random sets $V\subset V(K_{2n+1})\setminus V(\hat{T}')$ and $C\subset C(K_{2n+1})\setminus C(\hat{T}')$
 which are $(p+\eps)/6$-random and $(1-\eps)\eps$-random respectively. For trees not in Case C we will have $p\oldgg \eps$.

 Let $\ell$ and $X=\{x_1,\ldots,x_\ell,y_1,\ldots,y_\ell\}\subset V(\hat{T}')$ be such that to get a copy of $T$ from $\hat{T}'$ we need to add vertex-disjointly a suitable path between $x_i$ and $y_i$, for each $i\in [\ell]$.
We would like to find these paths with interior vertices in $\bar{V}:= V(K_{2n+1})\setminus V(\hat{T}')$ so that their edges are collectively rainbow with exactly the colours in $\bar{C}:=C(K_{2n+1})\setminus C(\hat{T}')$, using that $V\subset \bar{V}$ $C\subset \bar{C}$. Unfortunately, there may be some $x\in X$ with no edges with colour in $\bar{C}$ leading to $V$. (If $C$ and $V$ were independent, then, again, this would not happen with high probability.) If this happens, then the desired paths will not exist.

Note that the analogous version of the second problem in Case A does not arise in Case B. Here,  it is likely that every colour appears on many edges within $V$, so that we can use any colour by putting an appropriate edge within $V$ in the middle of one of the missing paths.

\subsubsection*{Retaining some independence}
To avoid the problem common to Cases A and B, when proving our version of Theorem~\ref{Sketch_Near_Embedding} (that is, Theorem~\ref{nearembedagain}), we set aside small random sets $V_0$ and $C_0$ early in the embedding, before the dependence between colours and vertices arises. This gives us a version of Theorem~\ref{Sketch_Near_Embedding} with the additional property that, for some $\mu\oldll \eps$, there are additional random sets $V_0\subset V(K_{2n+1})\setminus (V(\hat{T}')\cup V)$ and $C_0\subset C(K_{2n+1})\setminus (C(\hat{T}')\cup C)$ such that the following holds in addition to \ref{woo1} and \ref{woo2}.

\begin{enumerate}[label = {\bfseries \Alph{propcounter}\arabic{enumi}}]\addtocounter{enumi}{2}
\item $V_0$ is a $\mu$-random subset of $V(K_{2n+1})$, $C_0$ is a $\mu$-random subset of $C(K_{2n+1})$, and they are independent of each other.\label{propq2}
\end{enumerate}

\noindent
Then, by this independence, with high probability, every vertex in $K_{2n+1}$ will have $\mu^2 n/2$ adjacent edges with colour in $C_0$ going into the set $V_0$ (see Lemma~\ref{Lemma_high_degree_into_random_set}).

To avoid the problem that only arises in Case A, consider the set $U\subset V(T')$ of vertices which need leaves added to them to reach $T$ from $T'$. By carefully embedding a small subtree of $T'$ containing plenty of vertices in $U$, we ensure that, with high probability, each colour appears plenty of times between the image of $U$ and $V_0$. That is, we have the following additional property for some $1/n\oldll \xi\oldll \mu$.

\begin{enumerate}[label = {\bfseries \Alph{propcounter}\arabic{enumi}}]\addtocounter{enumi}{3}
\item With high probability, if $Z$ is the copy of $U$ in $\hat{T}'$, then every colour in  $C(K_{2n+1})$, has at least $\xi n$ edges between $Z$ and $V_0$.\label{propq1}
\end{enumerate}

Of course, \ref{propq2} and \ref{propq1} do not show that our desired matching/path-collection exists, only that (with high probability) there is no single colour or vertex preventing its existence. To move from this to find the actual matching/path-collection we use \emph{distributive absorption}.

\subsubsection*{Distributive absorption}
To prove our finishing lemmas, we use an \emph{absorption strategy}. Absorption has its origins in work by Erd\H{o}s, Gy\'arf\'as and Pyber~\cite{EP} and Krivelevich~\cite{MKtri}, but was codified  by R\"odl, Ruci\'nski and Szemer\'edi~\cite{RRS} as a versatile technique
for extending approximate results into exact ones. For both Case A and Case B we use a new implementation of \emph{distributive absorption}, a method introduced by the first author in~\cite{montgomery2018spanning}.

To describe our absorption, let us concentrate on Case A. Our methods in Case B are closely related, and we comment on these afterwards. To recap, we have a random rainbow tree $\hat{T}'$ in the ND-colouring of  $K_{2n+1}$ and a set $X\subset V(\hat{T})$, so that we need to add a perfect matching from  $X$ into $\bar{V}=V(K_{2n+1})\setminus V(\hat{T}')$ to make $\hat{T}'$ into a copy of $T$. We wish to add this matching in a rainbow fashion using (exactly) the colours in $\bar{C}= C(K_{2n+1})\setminus C(\hat{T})$.

To use distributive absorption, we first show that for any set $\hat{C}\subset C(K_{2n+1})$ of at most 100 colours, we can find a set $D\subset \bar{C}\setminus C$ and sets $X'\subset X$ and $V'\subset \bar{V}$ with $|D|\leq 10^3$, $|V'|\leq 10^4$ and $|X'|=|D|+1$, so that the following holds.
\stepcounter{propcounter}
\begin{enumerate}[label = {\bfseries \Alph{propcounter}\arabic{enumi}}]
    \item Given any colour $c\in\hat{C}$, there is a perfect $(D\cup \{c\})$-rainbow matching from $X'$ to $V'$.\label{switchprop}
\end{enumerate} 

We call such a triple $(D,X',V')$ a \emph{switcher} for $\hat{C}$. AS $|\bar{C}|=|X|$, a perfect $(\bar{C}\setminus D)$-rainbow matching from $X\setminus X'$ into $V\setminus V'$ uses all but 1 colour in $\bar{C}\setminus D$. If we can find such a matching whose unused colour, $c$ say, lies in $\hat{C}$, then using \ref{switchprop}, we can find a perfect $(D\cup\{c\})$-rainbow matching from $X'$ to $V'$. Then, the two matchings combined form a perfect $\bar{C}$-rainbow matching from $X$ into $V$, as required.

The switcher outlined above only gives us a tiny local variability property, reducing finding a large perfect matching with exactly the right number of colours to finding a large perfect matching with one spare colour so that the unused colour lies in a small set (the set $\bar{C}$). However, by finding many switchers for carefully chosen sets $\hat{C}$, we can build this into a global variability property. These switchers can be found using different vertices and colours (see Section~\ref{sec:switcherspaths}), so that matchings found using the respective properties \ref{switchprop} can be combined in our embedding. 

We choose different sets $\hat{C}$ for which to find a switcher by using an auxillary graph as a template. This template is a \emph{robustly matchable bipartite graph} --- a bipartite graph, $K$ say, with vertex classes $U$ and $Y\cup Z$ (where $Y$ and $Z$ are disjoint), with the following property.

\begin{enumerate}[label = {\bfseries \Alph{propcounter}\arabic{enumi}}]\addtocounter{enumi}{1}
\item For any set $Z^\ast\subset Z$ with size $|U|-|Y|$, there is a perfect matching in $K$ between $U$ and $Y\cup Z^\ast$.
\end{enumerate}

Such bipartite graphs were shown to exist by the first author~\cite{montgomery2018spanning}, and, furthermore, for large $m$ and $\ell\leq m$, we can find such a graph with maximum degree at most 100, $|U|=3m$, $|Y|=2m$ and $|Z|=m+\ell$ (see Lemma~\ref{Lemma_H_graph}). 
To use the template, we take disjoint sets of colours, $C'=\{c_v:v\in Y\}$ and $C''=\{c_v:v\in Z\}$ in $\bar{C}$. For each $u\in U$, we find a switcher $(D_u,X_u,V_u)$ for the set of colours $\{c_v:v\in N_K(u)\}$. Furthermore, we do this so that the sets $D_u$ are disjoint and in $\bar{C}\setminus (C'\cup C'')$, and the sets $X_u$, and $V_u$, are disjoint and in $X$, and $\bar{V}$, respectively.  We can then show we have the following property.
\begin{enumerate}[label = {\bfseries \Alph{propcounter}\arabic{enumi}}]\addtocounter{enumi}{2}
\item For any set $C^*\subset C''$ of $m$ colours, there is a perfect $(C^*\cup C'\cup(\cup_{u\in U}D_u))$-rainbow matching from $\cup_{u\in U}X_u$ into $\cup_{u\in U}V_u$.\label{Kprop}
\end{enumerate}
Indeed, to see this, take any set of $C^*\subset C''$ of $m$ colours, let $Z^*=\{v:c_v\in C^*\}$ and note that $|Z^*|=m$. By \ref{Kprop}, there is a perfect matching in $K$ from $U$ into $Y\cup Z^\ast$, corresponding to the function $f:U\to Y\cup Z^*$ say. For each $u\in U$, using that $(D_u,X_u,V_u)$ is a switcher for $\{c_v:v\in N_K(u)\}$ and $uf(u)\in E(K)$, find a perfect $(D_u\cup \{c_{f(u)}\})$-rainbow matching $M_u$ from $X_u$ to $V_u$.
As the sets $D_u$, $X_u$, $V_u$, $u\in U$, are disjoint, $\cup_{u\in U}M_u$ is a perfect $(C^*\cup C'\cup(\cup_{u\in U}D_u))$-rainbow matching from $\cup_{u\in U}X_u$ into $\cup_{u\in U}V_u$, as required.

Thus, we have a set of colours $C''$ from which we are free to use any $\ell$ colours, and then use the remaining colours together with $C'\cup(\cup_{u\in U}D_u)$ to find a perfect rainbow matching from $\cup_{u\in U}X_u$ into $\cup_{u\in U}V_u$. By letting $m$ be as large as allowed by our construction methods, and $C''$ be a random set of colours, we have a useful \emph{reservoir} of colours, so that we can find a structure in $T$ using $\ell$ colours in $C''$, and then finish by attaching a matching to $\cup_{u\in U}X_u$.


We have two things to consider to fit this final step into our proof structure, which we discuss below. Firstly, we can only absorb colours in $C''$, so after we have covered most of the colours, we need to cover the unused colours outside of $C''$ (essentially achieved by \ref{propp2} below). Secondly, we find the switchers greedily in a random set. There are many more unused colours from this set than we can absorb, and the unused colours no longer have good random properties, so we also need to reduce the unused colours to a number that we can absorb (essentially achieved by \ref{propp1} below).

\subsubsection*{Creating our finishing lemmas using absorption}
To recap, we wish to find a perfect $\bar{C}$-rainbow matching from $X$ into $\bar{V}$. To do this, it is sufficient to find partitions $X=X_1\cup X_2\cup X_3$, $\bar{V}=V_1\cup V_2\cup V_3$ and $\bar{C}=C_1\cup C_2\cup C_3$ with the following properties.
\stepcounter{propcounter}
\begin{enumerate}[label = {\bfseries \Alph{propcounter}\arabic{enumi}}]
\item There is a perfect $C_1$-rainbow matching from $X_1$ into $V_1$. \label{propp1}
\item Given any set of colours $C'\subset C_1$ with $|C'|\leq |C_1|-|X_1|$, there is a perfect $(C_2\cup C')$-rainbow matching from $X_2$ into $V_2$ which uses each colour in $C'$.\label{propp2}
\item Given any set of colours $C''\subset C_2$ with size $|X_3|-|C_3|$, there is a perfect  $(C''\cup C_3)$-rainbow matching from $X_3$ into $V_3$.\label{propp3}
\end{enumerate}
Finding such a partition requires the combination of all our methods for Case A. In brief, however, we develop \ref{propp1} using a result from~\cite{montgomery2018decompositions} (see Lemma~\ref{Lemma_MPS_nearly_perfect_matching}), we develop \ref{propp2} using the condition \ref{propq1}, and we develop \ref{propp3} using the distributive absorption strategy outlined above.

If we can find such a partition, then we can easily show that the matching we want must exist. Indeed, given such a partition, then, using \ref{propp1}, let $M_1$ be a perfect $C_1$-rainbow matching from $X_1$ into $V_1$, and let $C'=C_1\setminus C(M_1)$. Using~\ref{propp2}, let $M_2$ be a perfect  $(C_2\cup C')$-rainbow matching from $X_2$ into $V_2$ which uses each colour in $C'$, and let $C''=(C_2\cup C')\setminus C(M)=C_2\setminus C(M)$. Finally, noting that $|C''|+|C_3|=|\bar{C}|-|X_1|-|X_2|=|X_3|$, using~\ref{propp3}, let $M_3$ be a perfect $(C''\cup C_3)$-rainbow matching from $X_3$ into $V_3$. Then, $M_1\cup M_2\cup M_3$ is a $\bar{C}$-rainbow matching from $X$ into $\bar{V}$.

The above outline also lies behind our embedding in Case B, where we finish instead by embedding $\ell$ paths vertex-disjointly between certain vertex pairs, for some $\ell$. Instead of the partition $X_1\cup X_2\cup X_3$ we have a partition $[\ell]=I_1\cup I_2\cup I_3$, and, instead of each matching from $X_i$ to $V_i$, $i\in [3]$, we find a set of vertex-disjoint $x_j,y_j$-paths, $j\in I_i$, with interior vertices in $V_i$ which are collectively $C_i$-rainbow. The main difference is how we construct switchers using paths instead of matchings (see Section~\ref{sec:switchers}).

\subsection{\ref{T3}: The embedding in Case C}\label{sec:T3}
After large clusters of adjacent leaves are removed from a tree in Case C, few vertices remain. We remove these large clusters, from the tree, $T$ say, to get the tree $T'$, and carefully embed $T'$ into the ND-colouring using a deterministic embedding. The image of this deterministic embedding occupies a small interval in the ordering used to create the ND-colouring. Furthermore, the embedded vertices of $T'$ which need leaves added to create a copy of $T$ are well-distributed within this interval. These properties will allow us to embed the missing leaves using the remaining colours. This is given more precisely in Section~\ref{sec:caseC}, but in order to illustrate this in the easiest case, we will give the embedding when there is exactly one vertex with high degree.

Our embedding in this case is rather simple. Removing the leaves incident to a very high degree vertex, we embed the rest of the tree into $[n]$ so that the high degree vertex is embedded to 1. The missing leaves are then embedded into $[2n+1]\setminus [n]$ using the unused colours.

\begin{theorem}[One large vertex]\label{Theorem_one_large_vertex}
Let $n\geq 10^6$. Let $K_{2n+1}$ be $ND$-coloured, and let $T$ be an $(n+1)$-vertex tree containing a vertex $v_1$ which is adjacent to $\geq 2n/3$ leaves.
Then, $K_{2n+1}$ contains a rainbow copy of $T$.
\end{theorem}
\begin{proof}
See Figure~\ref{Figure_1_vertex} for an illustration of this proof.
Let $T'$ be $T$ with the neighbours of $v_1$ removed and let $m=|T'|$. By assumption, $|T'|\leq n/3+1$. Order the vertices of $T'-v_1$ as $v_2, \dots, v_{m}$ so that $T[v_1, \dots, v_i]$ is a tree for each $i\in [m]$. Embed $v_1$ to $1$ in $K_{2n+1}$, and then greedily embed $v_2, \dots, v_m$ in turn to some vertex in $[n]$ so that the copy
of $T'$ which is formed is rainbow in $K_{2n+1}$. This is possible since at each step at most $|T'|\leq n/3$ of the vertices in $[n]$ are occupied, and at most $e(T')\leq n/3-1$ colours are used. Since the $ND$-colouring has 2 edges of each colour adjacent to each vertex,
this forbids at most $n/3+2(n/3-1)=n-2$ vertices in $[n]$. Thus, we can embed each $v_{i}$, $2\leq i\leq m$ using an unoccupied vertex in $[n]$ so that the edge from $v_i$ to $v_1, \ldots, v_{i-1}$ has a colour that we have not yet used. Let $S'$ be the resulting rainbow copy of $T'$, so that $V(S')\subset [n]$.

Let $S$ be $S'$ together with the edges between $1$ and $2n+2-c$ for every $c\in [n]\setminus C(S')$. Note that the neighbours added are all bigger than $n$, and so the resulting graph is a tree. There are exactly $n-e(T')$ edges added, so $S$ is a copy of $T$. Finally, for each $c\in [n]\setminus C(S')$, the edge from $1$ to $2n+2-c$ is colour $c$, so the resulting tree is rainbow.
\end{proof}

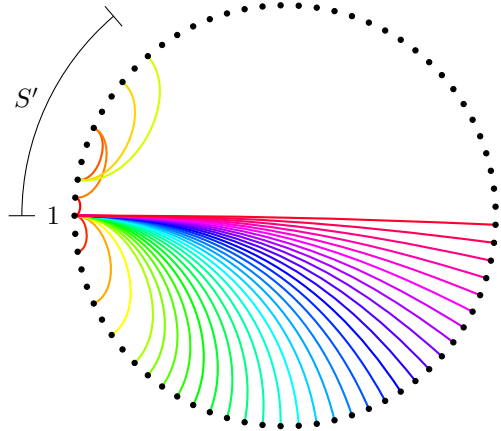
\begin{figure}[h]

\vspace{-0.1cm}

\begin{center}
{
\begin{tikzpicture}[scale=0.7,define rgb/.code={\definecolor{mycolor}{rgb}{#1}},
                    rgb color/.style={define rgb={#1},mycolor},reflect at xaxis/.style={xshift=1cm,xscale=-1,xshift=-1*1cm}]


\begin{scope}[reflect at xaxis,rotate=0]

\draw [thick,rgb color={1,0,0}, rotate = -4.9315068493*0] (0:4) to [in={4.9315068493*1+180},out=180] ({4.9315068493 *1}:4);
\draw [thick,rgb color={1,0.166666666666667,0}, rotate = -4.9315068493*2] (0:4) to [in={4.9315068493*2+180},out=180] ({4.9315068493 *2}:4);
\draw [thick,rgb color={1,0.333333333333333,0}, rotate = 4.9315068493*2] (0:4) to [in={4.9315068493*3+180},out=180] ({4.9315068493 *3}:4);
\draw [thick,rgb color={1,0.5,0}, rotate = 4.9315068493*1] (0:4) to [in={4.9315068493*4+180},out=180] ({4.9315068493 *4}:4);
\draw [thick,rgb color={1,0.666666666666667,0}, rotate = -4.9315068493*5] (0:4) to [in={4.9315068493*5+180},out=180] ({4.9315068493 *5}:4);
\draw [thick,rgb color={1,0.833333333333333,0}, rotate = 4.9315068493*2] (0:4) to [in={4.9315068493*6+180},out=180] ({4.9315068493 *6}:4);
\draw [thick,rgb color={1,1,0}, rotate = -4.9315068493*7] (0:4) to [in={4.9315068493*7+180},out=180] ({4.9315068493 *7}:4);
\draw [thick,rgb color={0.833333333333333,1,0}, rotate = 4.9315068493*2] (0:4) to [in={4.9315068493*8+180},out=180] ({4.9315068493 *8}:4);
\draw [thick,rgb color={0.666666666666667,1,0}, rotate = -4.9315068493*9] (0:4) to [in={4.9315068493*9+180},out=180] ({4.9315068493 *9}:4);
\draw [thick,rgb color={0.5,1,0}, rotate = -4.9315068493*10] (0:4) to [in={4.9315068493*10+180},out=180] ({4.9315068493 *10}:4);
\draw [thick,rgb color={0.333333333333333,1,0}, rotate = -4.9315068493*11] (0:4) to [in={4.9315068493*11+180},out=180] ({4.9315068493 *11}:4);
\draw [thick,rgb color={0.166666666666667,1,0}, rotate = -4.9315068493*12] (0:4) to [in={4.9315068493*12+180},out=180] ({4.9315068493 *12}:4);
\draw [thick,rgb color={0,1,0}, rotate = -4.9315068493*13] (0:4) to [in={4.9315068493*13+180},out=180] ({4.9315068493 *13}:4);
\draw [thick,rgb color={0,1,0.166666666666667}, rotate = -4.9315068493*14] (0:4) to [in={4.9315068493*14+180},out=180] ({4.9315068493 *14}:4);
\draw [thick,rgb color={0,1,0.333333333333333}, rotate = -4.9315068493*15] (0:4) to [in={4.9315068493*15+180},out=180] ({4.9315068493 *15}:4);
\draw [thick,rgb color={0,1,0.5}, rotate = -4.9315068493*16] (0:4) to [in={4.9315068493*16+180},out=180] ({4.9315068493 *16}:4);
\draw [thick,rgb color={0,1,0.666666666666667}, rotate = -4.9315068493*17] (0:4) to [in={4.9315068493*17+180},out=180] ({4.9315068493 *17}:4);
\draw [thick,rgb color={0,1,0.833333333333333}, rotate = -4.9315068493*18] (0:4) to [in={4.9315068493*18+180},out=180] ({4.9315068493 *18}:4);
\draw [thick,rgb color={0,1,1}, rotate = -4.9315068493*19] (0:4) to [in={4.9315068493*19+180},out=180] ({4.9315068493 *19}:4);
\draw [thick,rgb color={0,0.833333333333333,1}, rotate = -4.9315068493*20] (0:4) to [in={4.9315068493*20+180},out=180] ({4.9315068493 *20}:4);
\draw [thick,rgb color={0,0.666666666666667,1}, rotate = -4.9315068493*21] (0:4) to [in={4.9315068493*21+180},out=180] ({4.9315068493 *21}:4);
\draw [thick,rgb color={0,0.5,1}, rotate = -4.9315068493*22] (0:4) to [in={4.9315068493*22+180},out=180] ({4.9315068493 *22}:4);
\draw [thick,rgb color={0,0.333333333333333,1}, rotate = -4.9315068493*23] (0:4) to [in={4.9315068493*23+180},out=180] ({4.9315068493 *23}:4);
\draw [thick,rgb color={0,0.166666666666667,1}, rotate = -4.9315068493*24] (0:4) to [in={4.9315068493*24+180},out=180] ({4.9315068493 *24}:4);
\draw [thick,rgb color={0,0,1}, rotate = -4.9315068493*25] (0:4) to [in={4.9315068493*25+180},out=180] ({4.9315068493 *25}:4);
\draw [thick,rgb color={0.166666666666667,0,1}, rotate = -4.9315068493*26] (0:4) to [in={4.9315068493*26+180},out=180] ({4.9315068493 *26}:4);
\draw [thick,rgb color={0.333333333333333,0,1}, rotate = -4.9315068493*27] (0:4) to [in={4.9315068493*27+180},out=180] ({4.9315068493 *27}:4);
\draw [thick,rgb color={0.5,0,1}, rotate = -4.9315068493*28] (0:4) to [in={4.9315068493*28+180},out=180] ({4.9315068493 *28}:4);
\draw [thick,rgb color={0.666666666666667,0,1}, rotate = -4.9315068493*29] (0:4) to [in={4.9315068493*29+180},out=180] ({4.9315068493 *29}:4);
\draw [thick,rgb color={0.833333333333333,0,1}, rotate = -4.9315068493*30] (0:4) to [in={4.9315068493*30+180},out=180] ({4.9315068493 *30}:4);
\draw [thick,rgb color={1,0,1}, rotate = -4.9315068493*31] (0:4) to [in={4.9315068493*31+180},out=180] ({4.9315068493 *31}:4);
\draw [thick,rgb color={1,0,0.833333333333333}, rotate = -4.9315068493*32] (0:4) to [in={4.9315068493*32+180},out=180] ({4.9315068493 *32}:4);
\draw [thick,rgb color={1,0,0.666666666666667}, rotate = -4.9315068493*33] (0:4) to [in={4.9315068493*33+180},out=180] ({4.9315068493 *33}:4);
\draw [thick,rgb color={1,0,0.5}, rotate = -4.9315068493*34] (0:4) to [in={4.9315068493*34+180},out=180] ({4.9315068493 *34}:4);
\draw [thick,rgb color={1,0,0.333333333333333}, rotate = -4.9315068493*35] (0:4) to [in={4.9315068493*35+180},out=180] ({4.9315068493 *35}:4);
\draw [thick,rgb color={1,0,0.166666666666667}, rotate = -4.9315068493*36] (0:4) to [in={4.9315068493*36+180},out=180] ({4.9315068493 *36}:4);

\foreach \x in {0,...,72}
{
\draw [fill=black] (4.9315068493*\x:4) circle [radius=0.05cm];
}

\draw (0:4.4) node {$1$};
\draw (0:4.75) -- (0:5.25);
\draw [rotate=4.9315068493*10] (0:4.75) -- (0:5.25);
\draw (5,0) arc (0:4.9315068493*10:5);
\draw (4.9315068493*5:5.4) node {$S'$};

\end{scope}

\end{tikzpicture}\;\;\;\;\;\;
}
\end{center}

\vspace{-0.1cm}

\caption{Embedding the tree in Case C when there is  1  vertex with many leaves as neighbours.}\label{Figure_1_vertex}
\end{figure}

The above proof demonstrates the main ideas of our strategy for Case C. Notice that the above proof has two parts --- first we embed the small tree $T'$, and then we find the neighbours of the high degree vertex $v_1$. In order to ensure that the final tree is rainbow we choose the neighbours of $v_1$ in some interval $[n+1, 2n]$ which is disjoint from the copy of $T'$, and to which every colour appears from the image of $v_1$. This way, we were able to use every colour which was not present on the copy of $T'$.
When there are multiple high degree vertices $v_1, \dots, v_{\ell}$ the strategy is the same --- first we embed a small rainbow tree $T'$ containing $v_1, \dots, v_{\ell}$, then we embed the neighbours of $v_1, \dots, v_{\ell}$. This is done in Section~\ref{sec:caseC}.

\subsection{Proof of Theorem~\ref{Theorem_Ringel_proof}}\label{sec:overhead}
Here we will state our main theorems and lemmas, which are proved in later sections, and combine them to prove Theorem~\ref{Theorem_Ringel_proof}. First, we have our randomized embedding of a $(1-\epsilon)n$-vertex tree, which is proved in Section~\ref{sec:almost}. For convenience we use the following definition.

\begin{definition}
Given a vertex set $V\subset V(G)$ of a graph $G$, we say $V$ is \emph{$\ell$-replete in $G$} if $G[V]$ contains at least $\ell$ edges of every colour in $G$. Given, further, $W\subset V(G)\setminus V$, we say $(W,V)$ is \emph{$\ell$-replete in $G$} if at least $\ell$ edges of every colour in $G$ appear in $G$ between $W$ and $V$. When $G=K_{2n+1}$, we simply say that $V$ and $(W,V)$ are \emph{$\ell$-replete}.
\end{definition}

\begin{theorem}[Randomised tree embeddings]\label{nearembedagain} Let $1/n\ll  \xi\ll \mu\ll \eta\ll \eps\ll 1$ and $\xi\ll 1/k\ll \log^{-1}n$.
Let $K_{2n+1}$ be ND-coloured, let $T'$ be a $(1-\epsilon)n$-vertex forest and let $U\subset V(T')$ contain $\eps n$ vertices. Let $p$ be such that removing leaves around vertices next to $\geq k$ leaves from $T'$ gives a forest with $pn$ vertices.

Then, there is a random subgraph $\hat{T}'\subset K_{2n+1}$ and disjoint random subsets $V,V_0\subset V(K_{2n+1})\setminus V(\hat{T}')$ and $C,C_0\subset C(K_{2n+1})\setminus C(\hat{T}')$ such that the following hold.
\stepcounter{propcounter}
\begin{enumerate}[label = {\bfseries \Alph{propcounter}\arabic{enumi}}]
\item With high probability, $\hat{T}'$ is a rainbow copy of $T'$ in which, if $W$ is the copy of $U$, then $(W,V_0)$ is $(\xi n)$-replete,\label{propa1}
\item $V_0$ and $C_0$ are $\mu$-random and independent of each other, and\label{propa2}
\item $V$ is $(p+\epsilon)/6$-random and $C$ is $(1-\eta)\epsilon$-random.\label{propa3}
\end{enumerate}
\end{theorem}

Next, we have the two finishing lemmas, which are proved in Sections~\ref{sec:finishA} and~\ref{sec:finishB} respectively.

\begin{lemma}[The finishing lemma for Case A]\label{lem:finishA} Let $1/n\ll \xi\ll \mu \ll \eta \ll\eps\ll p\leq 1$. Let $K_{2n+1}$ be 2-factorized. Suppose that $V,V_0$ are disjoint subsets of $V(K_{2n+1})$ which are $p$- and $\mu$-random respectively. Suppose that $C,C_0$ are disjoint subsets in $C(K_{2n+1})$, so that $C$ is $(1-\eta)\eps$-random, and
$C_0$ is $\mu$-random and independent of $V_0$. Then, with high probability, the following holds.

Given any disjoint sets $X,Z\subset V(K_{2n+1})\setminus (V\cup V_0)$ with $|X|=\eps n$, so that $(X,Z)$ is $(\xi n)$-replete, and any set $D\subset C(K_{2n+1})$ with $|D|=\eps n$ and $C_0\cup C\subset D$, there is a perfect $D$-rainbow matching from $X$ into $V\cup V_0\cup Z$.
\end{lemma}
Note that in the following lemma we implicitly assume that $m$ is an integer. That is, we assume an extra condition on $n$, $k$ and $\eps$. We remark on this further in Section~\ref{sec:not}.
\begin{lemma}[The finishing lemma for Case B]\label{lem:finishB} Let $1/n\ll  1/k\ll \mu\ll \eta \ll\eps\ll p\leq 1$ be such that $k = 7\mod 12$ and $695|k$. Let $K_{2n+1}$ be 2-factorized. Suppose that $V,V_0$ are disjoint subsets of $V(K_{2n+1})$ which are $p$- and $\mu$-random respectively. Suppose that $C,C_0$ are disjoint subsets in $C(K_{2n+1})$, so that $C$ is $(1-\eta)\eps$-random, and $C_0$ is $\mu$-random and independent of $V_0$. Then, with high probability, the following holds with $m=\eps n/k$.

For any set $\{x_1,\ldots,x_m,y_1,\ldots,y_m\}\subset V(K_{2n+1})\setminus (V\cup V_0)$, and any set $D\subset C(K_{2n+1})$ with $|D|=mk$ and $C\cup C_0\subset D$, the following holds. There is a set of vertex-disjoint $x_i,y_i$-paths with length $k$, $i\in [m]$, which have interior vertices in $V\cup V_0$ and which are collectively $D$-rainbow.
\end{lemma}

Finally, the following theorem, proved in Section~\ref{sec:caseC}, will allow us to embed trees in Case C.
\begin{theorem}[Embedding trees in Case C]\label{Theorem_case_C}
Let $n\geq 10^6$. Let $K_{2n+1}$ be $ND$-coloured, and let $T$ be a tree on $n+1$ vertices with a subtree $T'$ with $\ell:=|T'|\leq n/100$ such that $T'$ has vertices $v_1, \dots, v_{\ell}$ so that adding $d_i\geq \log^4n$ leaves to each $v_i$ produces $T$.
Then, $K_{2n+1}$ contains a rainbow copy of $T$.
\end{theorem}

We can now combine these results to prove Theorem~\ref{Theorem_Ringel_proof}. We also use a simple lemma concerning replete random sets, Lemma~\ref{Lemma_inheritence_of_lower_boundedness_random}, which is proved in Section~\ref{sec:prelim}. As used below, it implies that if $V_0$ and $V_1$, with $V_1\subset V_0\subset V(K_{2n+1})$, are $\mu$- and $\mu/2$-random respectively, then, given any randomised set $X$ such that $(X,V_0)$ is with high probability replete (for some parameter), then $(X,V_1)$ is also with high probability replete (for some suitably reduced parameter).\begin{proof}[Proof of Theorem~\ref{Theorem_Ringel_proof}]
Choose $\xi, \mu,\eta,\delta,\bar{\mu},\bar{\eta}$ and $\bar{\eps}$ such that $1/n \ll\xi \ll\mu\ll \eta\ll \delta\ll\bar{\mu}\ll\bar{\eta}\ll\bar{\eps} \ll\log^{-1} n$ and $k=\delta^{-1}$ is an integer such that $k = 7\mod 12$ and $695|k$. Let $T$ be an $(n+1)$-vertex tree and let $K_{2n+1}$ be ND-coloured.
By Lemma~\ref{Lemma_case_division}, $T$ is in Case A, B or C for this $\delta$. If $T$ is in Case C, then Theorem~\ref{Theorem_case_C} implies that $K_{2n+1}$ has a rainbow copy of $T$. Let us assume then that $T$ is not in Case C. Let $k=\delta^{-4}$, and note that, as $T$ is not in Case $C$, removing from $T$ leaves around any vertex adjacent to at least $k$ leaves gives a tree with at least $n/100$ vertices.

If $T$ is in Case A, then let $\eps=\delta^6$, let $L$ be a set of $\eps n$ non-neighbouring leaves in $T$, let $U=N_T(L)$ and let $T'=T-L$. Let $p$ be such that removing from $T'$ leaves around any vertex adjacent to at least $k$ leaves gives a tree with $pn$ vertices. Note that each leaf of $T'$ which is not a leaf of $T$ must be adjacent to a vertex in $L$ in $T$. Note further that, if a vertex in $T$ is next to fewer than $k$ leaves in $T$, but at least $k$ leaves in $T'$, then all but at most $k-1$ of those leaves in $T'$ must have a neighbour in $L$ in $T$. Therefore, $p\geq 1/100-(k+1)\eps\geq 1/200$.

By Theorem~\ref{nearembedagain}, there is a random subgraph $\hat{T}'\subset K_{2n+1}$ and random subsets $V,V_0\subset V(K_{2n+1})\setminus V(\hat{T}')$ and $C,C_0\subset C(K_{2n+1})\setminus C(\hat{T}')$ such that \ref{propa1}--\ref{propa3} hold.
Using \ref{propa2}, let $V_1,V_2\subset V_0$ be disjoint $(\mu/2)$-random subsets of $V(K_{2n+1})$.

 By Lemma~\ref{lem:finishA} (applied with $\xi'=\xi/4, \mu'=\mu/2, \eta=\eta, \epsilon=\epsilon, p'=(p+\epsilon)/6$, $V=V, C=C, V_0'=V_1,$ and $C_0'$ a $(\mu/2)$-random subset of $C_0$) 
 and \ref{propa2}--\ref{propa3}, and by \ref{propa1} and Lemma~\ref{Lemma_inheritence_of_lower_boundedness_random}, with high probability we have the following properties.
\stepcounter{propcounter}
\begin{enumerate}[label = {\bfseries \Alph{propcounter}\arabic{enumi}}]
\item Given any disjoint sets $X,Z\subset V(K_{2n+1})\setminus (V\cup V_1)$ so that $|X|=\eps n$ and $(X,Z)$ is $(\xi n/4)$-replete, and any set $D\subset C(K_{2n+1})$ with $|D|=\eps n$ and $C_0\cup C\cup D$, there is a perfect $D$-rainbow matching from $X$ into $V\cup V_1\cup Z$.\label{fine1}
\item $\hat{T}'$ is a rainbow copy of $T'$ in which, letting $W$ be the copy of $U$, $(W,V_2)$ is $(\xi n/4)$-replete.\label{fine2}
\end{enumerate}
Let $D=C(K_{2n+1})\setminus C(\hat{T}')$, so that $C_0\cup C\subset D$, and, as $\hat{T}'$ is rainbow by \ref{fine2}, $|D|=\eps n$. Let $W$ be the copy of $U$ in $\hat{T}'$. Then, using \ref{fine1} with $Z=V_2$ and \ref{fine2}, let $M$ be a perfect $D$-rainbow matching from $W$ into $V\cup V_1\cup V_2\subset V\cup V_0$. As $V\cup V_0$ is disjoint from $V(\hat{T}')$, $\hat{T}'\cup M$ is a rainbow copy of $T$. Thus, a rainbow copy of $T$ exists with high probability in the ND-colouring of $K_{2n+1}$, and hence certainly some such rainbow copy of $T$ must exist.

If $T$ is in Case B, then recall that $k=\delta^{-1}$ and let $m=\bar{\eps} n/k$. Let $P_1,\ldots,P_m$ be vertex-disjoint bare paths with length $k$ in $T$. Let $T'$ be $T$ with the interior vertices of $P_i$, $i\in [m]$, removed.
 Let $p$ be such that removing from $T'$ leaves around any vertex adjacent to at least $k$ leaves gives a forest with $pn$ vertices. Note that (reasoning similarly to as in Case A) $p\geq 1/100-2(k+1)m\geq 1/200$.

 By Theorem~\ref{nearembedagain}, there is a random subgraph $\hat{T}'\subset K_{2n+1}$ and disjoint random subsets $V,V_0\subset V(K_{2n+1})\setminus V(\hat{T}')$ and $C,C_0\subset C(K_{2n+1})\setminus C(\hat{T}')$ such that \ref{propa1}--\ref{propa3} hold with $\xi=\xi$, $\mu=\bar{\mu}$, $\eta=\bar{\eta}$ and $\eps=\bar{\eps}$. By Lemma~\ref{lem:finishB} and \ref{propa1}--\ref{propa3}, and by \ref{propa1}, with high probability we have the following properties.
\stepcounter{propcounter}
\begin{enumerate}[label = {\bfseries \Alph{propcounter}\arabic{enumi}}]
\item For any set $\{x_1,\ldots,x_m,y_1,\ldots,y_m\}\subset V(K_{2n+1})\setminus (V\cup V_0)$ and any set $D\subset C(K_{2n+1})$ with $|D|=mk$ and $C_0\cup C\subset D$, there is a set of vertex-disjoint paths $x_i,y_i$-paths with length $k$, $i\in [m]$, which have interior vertices in $V\cup V_0$ and which are collectively $D$-rainbow.\label{fine11}
\item $\hat{T}'$ is a rainbow copy of $T'$.\label{fine12}
\end{enumerate}
Let $D=C(K_{2n+1})\setminus C(\hat{T}')$, so that $C_0\cup C\subset D$, and, as $\hat{T}'$ is rainbow by \ref{fine12}, $|D|=\eps n$.
For each path $P_i$, $i\in [m]$, let $x_i$ and $y_i$ be the copy of the endvertices of $P_i$ in $\hat{T}'$. Using \ref{fine11}, let $Q_i$, $i\in [m]$, be a set of vertex-disjoint $x_i,y_i$-paths with length $k$, $i\in [m]$, which have interior vertices in $V\cup V_0$ and which are collectively $D$-rainbow. As $V\cup V_0$ is disjoint from $V(\hat{T}')$,
 $\hat{T}'\cup (\cup_{i\in [m]}Q_i)$ is a rainbow copy of $T$. Thus, a rainbow copy of $T$ exists with high probability in the ND-colouring of $K_{2n+1}$, and hence certainly some such rainbow copy of $T$ must exist.
\end{proof}



\section{Preliminary results and observations}\label{sec:prelim}

\subsection{Notation}\label{sec:not}
For a coloured graph $G$, we denote the set of vertices of $G$ by $V(G)$, the set of edges of $G$ by $E(G)$, and the set of colours of $G$ by $C(G)$.
For a coloured graph $G$, disjoint sets of vertices $A,B\subseteq V(G)$ and a set of colours $C\subseteq C(G)$ we use $G[A,B,C]$ to denote the subgraph of $G$ consisting of colour $C$ edges from $A$ to $B$, and $G[A,C]$ to be the graph of the colour $C$ edges within $A$.
For a single colour $c$, we denote the set of colour $c$ edges from $A$ to $B$ by $E_c(A,B)$.

A coloured graph is \emph{globally $k$-bounded} if every colour is on at most $k$ edges.
For a set of colours $C$, we say that a graph $H$ is ``$C$-rainbow''  if $H$ is rainbow and $C(H)\subseteq C$. We say that a collection of graphs $H_1, \dots, H_k$ is \emph{collectively} rainbow if their union is rainbow.
A star is a tree consisting of a collection of leaves joined to a single vertex (which we call the \emph{centre}).
A \emph{star forest} is a graph consisting of vertex disjoint stars.

For any reals $a,b\in \R$, we say $x=a\pm b$ if $x\in [a-b,a+b]$.

\subsubsectionme{Asymptotic notation}
For any $C\geq 1$ and $x,y\in (0,1]$, we use ``$x\oldll_C y$'' to mean ``$x\leq \frac{y^C}{C}$''. We will write ``$x\LL y$'' to mean that there is some absolute constant $C$ for which the proof works with ``$x\LL{} y$'' replaced by ``$x\oldll_{C} y$''. 
In other words the proof works if $y$ is a small but fixed power of $x$.
This notation compares to the more common notation $x\oldll y$ which means ``there is a fixed positive continuous function $f$ on $(0,1]$ for which the remainder of the proof works with ``$x\oldll y$'' replaced by ``$x\leq f(y)$''. (Equivalently, ``$x\oldll y$'' can be interpreted as ``for all $x\in (0,1]$, there is some $y\in (0,1]$ such that the remainder of the proof works with $x$ and $y$''.) The two notations ``$x\LL{} y$'' and ``$x\oldll y$'' are largely interchangeable --- most of our proofs remain correct with all instances of ``$\LL$'' replaced by ``$\oldll$''. The advantage of using ``$\LL$'' is that it proves polynomial bounds on the parameters (rather than bounds of the form ``for all $\epsilon>0$ and  sufficiently large $n$''). This is important towards the end of this paper, where the proofs need polynomial parameter bounds.

While the constants $C$ will always be implicit in each instance  of ``$x\LL{} y$'',  it is possible to work  them out explicitly. To do this one should go through the lemmas in the paper and choose the constants $C$ for a lemma after the constants have been chosen for the lemmas on which it depends. This is because an inequality $x\oldll_C y$ in a lemma  may be needed to imply an inequality $x\oldll_{C'} y$ for a lemma it depends on. Within an individual lemma we will often have several inequalities of the form $x\LL y$. There the constants $C$ need to be chosen in the reverse order of their occurrence in the text. The reason for this is the same --- as we prove a lemma we may use an inequality  $x\oldll_C y$  to imply another inequality $x\oldll_{C'} y$ (and so we should choose $C'$ before choosing $C$).

Throughout the paper, there are four operations we perform with the ``$x\LL{} y$'' notation:
\begin{enumerate}[label = (\alph{enumi})]
\item We will use $x_1\LL x_2\LL\dots\LL x_k$ to deduce finitely many inequalities of the form ``$p(x_1, \dots, x_k)\leq q(x_1, \dots, x_k)$'' where $p$ and $q$ are {monomials} with non-negative coefficients and $\min\{i: p(0, \dots, 0, x_{i+1}, \dots, x_k)=0\}< \min\{j: q(0, \dots, 0, x_{j+1}, \dots, x_k)=0\}$  e.g.\ $1000x_1\leq x_2^5x_4^2x_5^3$ is of this form.
\item We will use $x\LL y$ to deduce finitely many inequalities of the form ``$x\oldll_C y$'' for a fixed constant $C$.
\item For $x\LL y$ and  fixed constants $C_1, C_2$, we can choose a variable $z$ with $x\oldll_{C_1} z\oldll_{C_2}y$.
\item For $n^{-1}\LL 1$ and any fixed constant $C$, we can deduce $n^{-1}\oldll_C \log^{-1} n\oldll_C 1$.
\end{enumerate}
See \cite{montgomery2018decompositions} for a detailed explanation of why the above operations are valid.

\subsubsectionme{Rounding}
In several places, we will have, for example, constants $\eps$ and integers $n,k$ such that $1/n\ll \eps,1/k$ and require that $m=\eps n/k$ is an integer, or even divisible by some other small integer. Note that we can arrange this easily with a very small alteration in the value of $\eps$. For example, to apply Lemma~\ref{lem:finishB} we assume that $m$ is an integer, and therefore in the proof of Theorem~\ref{Theorem_Ringel_proof}, when we choose $\bar{\eps}$ we make sure that when this lemma is applied with $\eps=\bar{\eps}$ the corresponding value for $m$ is an integer.

\subsection{Probabilistic tools}\label{sec:prob}
For a finite set $V$, a $p$-random subset of $V$ is a set formed by choosing every element of $V$ independently  at random with probability $p$. If $V$ is not specified, then we will implicitly assume that $V$ is $V(K_{2n+1})$ or $C(K_{2n+1})$, where this will be clear from context. 
If $A,B\subseteq V$ with $A$ $p$-random and $B$ $q$-random, we say that $A$ and $B$ are \emph{disjoint} if every $v\in V$ is in $A$ with probability $p$, in $B$ with probability $q$, and outside of $A\cup B$ with probability $1-p-q$ (and this happens independently for each $v\in V$). We say that a $p$-random set $A$ is independent from a $q$-random set $B$ if the choices for $A$ and $B$ are made independently, that is, if $\P(A=A'\land B=B')=\P(A=A')\P(B=B')$ for any outcomes $A'$ and $B'$ of $A$ and $B$.

Often, we will have a $p$-random subset $X$ of $V$ and divide  it into two disjoint $(p/2)$-random subsets of $V$. This is possible by choosing which subset each element of $X$ is in independently at random with probability $1/2$ using the following simple lemma.
\begin{lemma}[Random subsets of random sets]\label{Lemma_mixture_of_p_random_sets}
Suppose that $X,Y:\Omega\to 2^V$ where $X$ is a $p$-random subset of $V$ and $Y|X$ is a $q$-random subset of $X$ (i.e. the distribution of $Y$ conditional on the event ``$X=X'$'' is that of a $q$-random subset of $X'$). Then $Y$ is a $pq$-random subset of $V$.
\end{lemma}
\begin{proof}
First notice that, to show a set $X\subseteq V$ is $(pq)$-random, it is sufficient to show that  $\P(S\subseteq X)=(pq)^{|S|}$ for all $S\subseteq V$ (for example, by using inclusion-exclusion).
Now, we prove the lemma.
Since $X$ is $p$-random we have $\P(S\subseteq X)=p^{|S|}$. Since $Y|X$ is $q$-random, we have $\P(S\subseteq Y|S\subseteq X)=q^{|S|}$. This gives $\P(S\subseteq Y)= \P(S\subseteq Y|S\subseteq X)\P(S\subseteq X)=(pq)^{|S|}$ for every set $S\subseteq V$.
\end{proof}

We will use the following standard form of Azuma's inequality and a Chernoff Bound. For a probability space $\Omega=\prod_{i=1}^n \Omega_i$, a random variable $X:\Omega\to \R$ is \emph{$k$-Lipschitz} if changing $\omega\in \Omega$ in any one coordinate changes $X(\omega)$ by at most $k$.

\begin{lemma}[Azuma's Inequality]\label{Lemma_Azuma}
Suppose that $X$ is $k$-Lipschitz and influenced by $\leq m$ coordinates in $\{1, \dots, n\}$. Then{, for any $t>0$,}
$$\P\left(|X-\E(X)|>t \right)\leq 2e^{\frac{-t^2}{mk^2}}$$.
\end{lemma}

Notice that the bound in the above inequality can be rewritten as $\P\left(X\neq \E(X)\pm t \right)\leq 2e^{\frac{-t^2}{mk^2}}$.

\begin{lemma}[Chernoff Bound] \label{Lemma_Chernoff}
Let $X$ be a binomial random variable with parameters $(n,p)$.
Then, for each $\epsilon\in (0,1)$, we have
$$\mathbb P\big(|X-pn|> \epsilon pn\big)\leq
2e^{-\frac{pn\epsilon^2}{3}}.$$
\end{lemma}

For an event $X$ in a probability space depending on a parameter $n$, we say ``$X$ holds with high probability'' to mean ``$X$ holds with probability $1-o(1)$'' where $o(1)$ is some function $f(n)$ with $f(n)\to 0$ as $n\to \infty$.
We will use this definition for the following operations.
\begin{itemize}
\item \textbf{Chernoff variant:} For $\epsilon \gg n^{-1}$, if $X$ is a $p$-random subset of $[n]$, then, with high probability, $|X|=(1\pm \epsilon)pn$.
\item \textbf{Azuma variant:} For $\epsilon\gg n^{-1}$ and fixed $k$, if $Y$ is a  $k$-Lipschitz random variable influenced by at most $n$ coordinates, then, with high probability, $Y=\E(Y)\pm \epsilon n$.
\item \textbf{Union bound variant:} For fixed $k$, if $X_1, \dots, X_{k}$ are events which hold with high probability then they simultaneously occur with high probability.
\end{itemize}
The first two of these follow directly from Lemmas~\ref{Lemma_Azuma} and~\ref{Lemma_Chernoff}, the latter from the union bound.

\subsection{Structure of trees}\label{Section_structure_of_trees}
Here, we gather lemmas about the structure of trees. Most of these lemmas say something about the leaves and bare paths of  a tree. It is easy to see that a tree with few leaves must have many bare paths. The most common version of this is the following well known lemma.
\begin{lemma}[\cite{montgomery2018spanning}]\label{split} For any integers $n,k>2$, a tree with~$n$ vertices either has at least $n/4k$ leaves or a collection of at least $n/4k$ vertex disjoint bare paths, each with length $k$.
\end{lemma}
As a corollary of this lemma we show that every tree either has many bare paths, many non-neighbouring leaves, or many large stars. This lemma underpins the basic case division for this paper. The rest of the proofs focus on finding rainbow copies of the three types of trees.
\begin{lemma}[Case division]\label{Lemma_case_division}
Let $1\gg \delta \gg n^{-1}$. Every $n$-vertex tree satisfies one of the following:
\begin{enumerate}[label= \Alph{enumi}]
\item There are at least $\delta n/800$ vertex-disjoint bare paths with length at least $\delta^{-1}$.
\item There are at least $\delta^6 n$ non-neighbouring leaves.
\item Removing leaves next to vertices adjacent to at least $\delta^{-4}$ leaves gives a tree with at most
$n/100$ vertices.
\end{enumerate}
\end{lemma}
\begin{proof}
Take $T$ and remove leaves around any vertex adjacent to at least $\delta^{-4}$ leaves, and call the resulting tree $T'$. If $T'$ has at most $n/100$ vertices then we are in Case C. Assume then, that $T'$ has at least $n/100$ vertices.

By Lemma~\ref{split} applied with $n'=|T'|$ and $k=\lceil \delta^{-1}\rceil$, the tree $T'$  either has at least $\delta n/600$ vertex disjoint bare paths with length at least $\delta^{-1}$ or at least $\delta n/600$ leaves. In the first case, as vertices were deleted next to at most $\delta^4 n$ vertices to get $T'$ from $T$, $T$ has at least $\delta n/600-\delta^4n\geq  \delta n/800$ vertex disjoint bare paths with length at least $\delta^{-1}$, so we are in Case B.

In the second case, if there are at least $\delta n/1200$ leaves of $T'$ which are also leaves of $T$, then, as there are at most $\delta^{-4}$ of these leaves around each vertex in $T'$, $T$ has at least $(\delta n/1200)/\delta^{-4}\geq \delta^{6}n$ non-neighbouring leaves, so we are in Case A. On the other hand, if there are not such a number of leaves of $T'$ which are leaves of $T$, then $T'$ must have at least $\delta n/1200$ leaves which are not leaves of $T$, and which therefore are adjacent to a leaf in $T$. Thus, $T$ has at least $\delta n/1200\geq \delta^{6}n$ non-neighbouring leaves, and we are also in Case A.
\end{proof}
We say a set of subtrees $T_1,\ldots, T_\ell\subset T$ divides a tree $T$ if $E(T_1)\cup\ldots \cup E(T_\ell)$ is a partition of $E(T)$. We use the following lemma.

\begin{lemma}[\cite{montgomery2018spanning}]\label{littletree} Let $n,m\in \mathbb N$ satisfy $1\leq m\leq n/3$. Given any tree~$T$ with~$n$ vertices and a vertex $t\in V(T)$, we can find two trees $T_1$ and $T_2$ which divide~$T$ so that $t\in V(T_1)$ and $m\leq |T_2|\leq 3m$.
\end{lemma}
Iterating this, we can divide a tree into small subtrees.
\begin{lemma}\label{dividetree} Let $T$ be a tree with at least $m$ vertices, where $m\geq 2$. Then, for some $s$, there is a set of subtrees $T_1,\ldots, T_s$ which divide $T$ so that $m\leq |T_i|\leq 4m$ for each $i\in [s]$.
\end{lemma}
\begin{proof}
We prove this by induction on $|T|$, noting that it is trivially true if $|T|\leq 4m$.

Suppose then $|T|>4m$ and the statement is true for all trees with fewer than $|T|$ vertices and at least $m$ vertices. By Lemma~\ref{littletree}, we can find two trees $T_1$ and $S$ which divide $T$ so that $m\leq |T_1|\leq 3m$. As $|T|>4m$, we have $m<|S|<|T|$, so there must be a set of subtrees $T_2,\ldots,T_s$, for some $s$, which divide $S$ so that $m\leq |T_i|\leq 4m$ for each $2\leq i\leq s$.
The subtrees $T_1,\ldots,T_s$ then divide $T$, with $m\leq |T_i|\leq 4m$ for each $i\in [s]$.
\end{proof}

For embedding trees in Cases A and B, we will need a finer understanding of the structure of trees. In fact, every tree can be built up from a small tree by successively adding leaves, bare paths, and stars, as follows.
 \begin{lemma}[\cite{MPS}]\label{Lemma_decomp} Given integers $d$ and $n$, $\mu>0$ and a tree $T$ with at most $n$ vertices, there are integers $\ell\leq 10^4 d\mu^{-2}$ and $j\in\{2,\ldots,\ell\}$ and a sequence of subgraphs $T_0\subset T_1\subset \ldots \subset T_\ell=T$ such that
\stepcounter{propcounter}
\begin{enumerate}[label = {\bfseries \Alph{propcounter}\arabic{enumi}}]
\item for each $i\in [\ell]\setminus \{1,j\}$, $T_{i}$ is formed from $T_{i-1}$ by adding non-neighbouring leaves,\label{cond1}
\item $T_j$ is formed from $T_{j-1}$ by adding at most $\mu n$ vertex-disjoint bare paths with length $3$,\label{cond2}
\item $T_1$ is formed from $T_0$ by adding vertex-disjoint stars with at least $d$ leaves each, and\label{cond3}
\item $|T_0|\leq 2\mu n$.\label{cond4}
\end{enumerate}
\end{lemma}

The following variation will be more convenient to use here. It shows that an arbitrary tree $T$ can be built out of a preselected, small subtree $T_1$ by a sequence of operations. It is important to control the starting tree as it allows us to choose which part of the tree will form our absorbing structure.

For forests $T'\subseteq T$, we say that $T$ is obtained from $T'$ by \emph{adding a matching of leaves} if all the vertices in $V(T)\setminus V(T')$ are non-neighbouring leaves in $T$.
\begin{lemma}[Tree splitting]\label{Lemma_tree_splitting}
Let $1\geq  d^{-1} \gg n^{-1}$.
Let $T$ be a tree with $|T|=n$ and $U\subseteq V(T)$ with $|U|\geq n/d^3$.
Then, there are forests $T_1^{\mathrm{small}}\subseteq T_2^{\mathrm{stars}}\subseteq T_3^{\mathrm{match}}\subseteq T_4^{\mathrm{paths}}\subseteq T_5^{\mathrm{match}}=T$ satisfying the following.
\stepcounter{propcounter}
\begin{enumerate}[label = {\bfseries \Alph{propcounter}\arabic{enumi}}]
\item $|T_{1}^{\mathrm{small}}|\leq  n/d$ and $|U\cap T_1^{small}|\geq n/d^6$.\label{rush1}
\item $T_2^{\mathrm{stars}}$ is formed from $T_{1}^{\mathrm{small}}$ by adding vertex-disjoint stars of size at least $d$.\label{rush2}
\item $T_3^{\mathrm{match}}$ is formed from $T_2^{\mathrm{stars}}$ by adding a sequence of $d^8$ matchings of leaves.\label{rush3}
\item $T_4^{\mathrm{paths}}$ is formed from $T_3^{\mathrm{match}}$ by adding at most $n/d$ vertex-disjoint paths of length $3$.\label{rush4}
\item $T_5^{\mathrm{match}}$ is formed from $T_4^{\mathrm{paths}}$ by adding a sequence of $d^8$ matchings of leaves.\label{rush5}
\end{enumerate}
\end{lemma}
\begin{proof}
First, we claim that there is a subtree $T_0$ of order $\leq n/2d^2$ containing at least $n/d^6$ vertices of $U$. To find this, use Lemma~\ref{dividetree} to find $s\leq 32d^2$ subtrees $T_1,\ldots, T_{s}$ which divide $T$ so that $n/16d^2\leq |T_i|\leq n/2d^2$ for each $i\in [s]$. As each vertex in $U$ must appear in some tree $T_i$,  there must be some tree $T_k$ which contains at least $|U|/32d^2\geq n/d^6$ vertices in $U$, as required.

  Let  $T'$ be the $n$-vertex tree  formed from $T$ by contracting $T_k$ into a single vertex $v_0$ and adding $e(T_k)$ new leaves at $v_0$ (called ``dummy'' leaves). 
Notice that  Lemma~\ref{Lemma_decomp} applies to  $T'$ with    $d=d, \mu= d^{-3}, n=n$ which gives a sequence of forests $T_0', \dots, T_{\ell}'$ for $\ell\leq 10^4d^7\leq d^8$.
Notice that $v_0\in T_0'$. Indeed, by construction, every vertex which is not in $T_0'$ can have in $T_{\ell}'=T'$ at most $\ell$ leaves. Since $v_0$ has $\geq  e(T_k)>\ell$ leaves in $T'$, it must be in $T'_0$.
For each $i=0, \dots, {\ell}$, let $T_i''$ be $T_i'$ with $T_k$ uncontracted and any dummy leaves of $v_0$ deleted.
Let  $T_2^{\mathrm{stars}}=T_1''$, $T_3^{\mathrm{match}}=T_{j-1}''$, $T_4^{\mathrm{paths}}=T_j''$, $T_5^{\mathrm{match}}=T_{\ell}''$.
Let $T_1^{\mathrm{small}}$ be $T_0''$ together with the $T_1''$-leaves of any $v\in T_k$ for which $|N_{T_1''}(v)\setminus N_{T_0''}(v)|<d$. We do this because when we uncontract $T_k$ the leaves which were attached to $v_0$ in $T_0'$ are now attached to vertices of $T_k$. 
If they form a star of size less than $d$ we cannot add them when we form $T_2^{\mathrm{stars}}$ without violating \ref{rush2} so we add them already when we form $T_1^{\mathrm{small}}$.
Since we are adding at most $d$ leaves for every vertex of $T_k$, we have $|T_1^{\mathrm{small}}|\leq d|T_k|+|T_0'|\leq  n/d$ so \ref{rush1} holds. This ensures that at least $d$ leaves are added to vertices of $T_1^{\mathrm{small}}$ to form  $T_2^{\mathrm{stars}}$ so \ref{rush2} holds. The remaining conditions \ref{rush3} -- \ref{rush5} are immediate from the application of Lemma~\ref{Lemma_decomp}.
\end{proof}


\subsection{Pseudorandom properties of random sets of vertices and colours}\label{Section_pseudorandomness}
Suppose that $K_{2n+1}$ is $2$-factorized. Choose a $p$-random set of vertices $V\subseteq V(K_{2n+1})$  and a $q$-random set of colours $C\subseteq C(K_{2n+1})$. What can be said about the subgraph $K_{2n+1}[V,C]$ consisting of edges within the set $V$ with colour in $C$? What ``pseudorandomness'' properties is this subgraph likely to have? In this section, we gather lemmas giving various such properties.
The setting of the lemmas is quite varied, as are the properties they give. For example, sometimes the sets $V$ and $C$ are chosen independently, while sometimes they are allowed to depend on each other arbitrarily. We split these lemmas into three groups based on the three principal settings.


\subsubsectionme{Dependent vertex/colour sets}
In this setting, our colour set $C\subseteq C(K_{2n+1})$ is $p$-random, and the vertex set is $V(K_{2n+1})$ (i.e., it is 1-random). Our pseudorandomness condition is that the number of edges between any two sizeable disjoint vertex sets is close to the expected number. Though a lemma of this kind was first proved in~\cite{alon2016random}, the precise pseudorandomness condition we will use here is in the following version from~\cite{MPS}. A colouring is \emph{locally $k$-bounded} if every vertex is adjacent to at most $k$ edges of each colour. 

\begin{lemma}[\cite{MPS}]\label{MPScolour}\label{Lemma_MPS_boundrandcolour} Let $k\in\mathbb N$ be constant and let $\epsilon, p\geq n^{-1/100}$. Let $K_n$ have a {locally}  $k$-bounded colouring and suppose $G$ is a subgraph of $K_n$ chosen by including the edges of each colour independently at random with probability $p$. Then, with  probability $1-o(n^{-1})$, for any disjoint sets $A,B\subset V(G)$, with $|A|,|B|\geq n^{3/4}$,
\[
\big|e_G(A,B)-p|A||B|\big|\leq \epsilon p|A||B|.
\]
\end{lemma}

If $V\subseteq V(K_{2n+1})$ is a $p$-random set of vertices, then the edges going from $V$ to $V(K_{2n+1})\setminus V$ are typically pseudorandomly coloured. The lemma below is a version of this. Here ``pseudorandomly coloured'' means that most colours have at most a little more than the expected number of colours leaving $V$.
\begin{lemma}[\cite{MPS}]\label{MPSvertex}\label{Corollary_MPS_randsetcor}
Let $k$ be constant and  $\epsilon,p\geq n^{-1/10^3}$. Let $K_n$ have a {locally}  $k$-bounded colouring and let $V$ be a $p$-random subset of $V(K_n)$. Then, with  probability $1-o(n^{-1})$, for each $A\subset V(K_n)\setminus V$  with $|A|\geq n^{1/4}$, for all but at most $\epsilon n$ colours there are at most $(1+\epsilon)pk|A|$ edges of that colour between $V$ and $A$.
\end{lemma}

A random vertex set $V$ likely has the property from Lemma~\ref{MPSvertex}. A random colour set $C$ likely has the property from Lemma~\ref{MPScolour}. If we combine these two properties, we can get a property involving $C$ and $V$ that is likely to hold. Importantly, this will be true even if $C$ and $V$ are not independent of each other. Doing this, we get the following lemma.

\begin{lemma}[Nearly-regular subgraphs]\label{Lemma_nearly_regular_subgraph}
Let $p, \gamma\gg n^{-1}$ and let  $K_{2n+1}$ be $2$-factorized.
Let $V\subseteq V(K_{2n+1})$, and $C\subseteq C(K_{2n+1})$  with $V$ $p/2$-random and $C$ $p$-random (possibly depending on each other). The following holds with  probability $1-o(n^{-1})$.

For every $U\subseteq V(K_{2n+1})\setminus V$  with $|U|= pn$, there are subsets $U'\subseteq U, V'\subseteq V, C'\subseteq C$ with $|U'|=|V'|= (1\pm\gamma)|U|$
so that $G=K_{2n+1}[U',V',C']$ is globally $(1+\gamma)p^2n$-bounded, and every vertex $v\in V(G)$ has $d_{G}(v)=(1\pm\gamma)p^2n$.
\end{lemma}
\begin{proof}
Choose $\alpha$ and $\epsilon$ so that $p, \gamma\gg \alpha\gg \epsilon\gg n^{-1}$. With high probability, by Lemma~\ref{MPScolour}, Lemma~\ref{MPSvertex} (with $k=2$, $n'=2n+1$, and $p'=p/2$) and Chernoff's bound, we can assume the following occur simultaneously.
\begin{enumerate}[label = (\roman{enumi})]
\item \label{prop1} For any disjoint $A,B\subset V(K_{2n+1})$ with $|A|,|B|\geq (2n+1)^{3/4}$, $|E_C(A,B)|=(1\pm \eps)p|A||B|$.
\item \label{prop2} For any $A\subseteq V(K_{2n+1})\setminus V$ with $|A|\geq n^{1/4}$, for all but at most $\epsilon n$ colours there are at most $(1+\epsilon)p|A|$ edges of that colour between $A$ and $V$.
\item \label{prop3} $|V|=(1\pm \eps)pn$.
\end{enumerate}

Let $U\subseteq V(K_{2n+1})\setminus V$ with $|U|= pn$ be arbitrary. Let $\hat C\subseteq C$ be the subset of colours $c\in C$ with $|E_c(U,V)|> (1+ \epsilon)p|U|$. Note that, from \ref{prop2}, we have $|\hat C|\leq \eps n$.

Let $\hat U^+\subseteq U$ and $\hat V^+\subseteq V$ be subsets of vertices $v$ with $d_{K_{2n+1}[U,V,C]}(v)> (1+ \alpha)p^2 n$, and let $\hat U^-\subseteq U$ and $\hat V^-\subseteq V$ be subsets of vertices $v$ with $d_{K_{2n+1}[U,V,C]}(v)< (1- \alpha)p^2 n$.

Now, $|E_C(U^+,V)|> |U^+|(1+\alpha)p^2 n\geq (1+\eps)p|U^+||V|$ by the definition of $U^+$ and \ref{prop3}. Therefore, as, by \ref{prop3}, $|V|\geq (1-\eps)pn>(2n+1)^{3/4}$, from \ref{prop1} we must have $|U^+|\leq (2n+1)^{3/4}\leq \eps n$. Similarly, we have $|U^-|,|V^+|,|V^-|\leq \eps n$.

Let $\hat U=U^+\cup U^-$ and Let $\hat V=V^+\cup V^-$. We have $|U\setminus \hat U|, |V\setminus \hat V| \geq (1\pm \epsilon)pn\pm 2\epsilon n$. Therefore, we can choose subsets $U'\subseteq U\setminus \hat U$ and $V'\subseteq V\setminus \hat V$ with $|U'|=|V'|= pn-3\epsilon n=(1\pm \gamma)pn$. Note that, by \ref{prop3} and as $|U|=pn$, $|U\setminus U'|,|V\setminus V'|\leq 4\epsilon n$.
Let  $C'=C\setminus \hat C$ and set $G=K_{2n+1}[U',V',C']$.
For each $v\in U'\cup V'$ we have $d_G(v)= (1\pm \alpha)p^2 n \pm 2|\hat C|\pm |U\setminus U'|\pm |V\setminus V'|=(1\pm \gamma)p^2 n$. For each $c\in C'$ we have $|E_c(U,V)|\leq (1+ \epsilon)p|U|\leq (1+ \gamma)p^2n$.
\end{proof}


\subsubsectionme{Deterministic colour sets and random vertex sets}
The following lemma  bounds the number of edges each colour typically has within a random vertex set.
\begin{lemma}[Colours inside random sets]\label{Lemma_number_of_colours_inside_random_set}
Let $p,\gamma\gg n^{-1}$  and let  $K_{2n+1}$ be $2$-factorized.
Let $V\subseteq V(K_{2n+1})$  be $p$-random. With high probability, every colour has $(1\pm \gamma)2p^2n$ edges inside $V$.
\end{lemma}
\begin{proof}
For an edge $e\in E(K_{2n+1})$ we have $\P(e \in E(K_n[V]))=p^2$.
By linearity of expectation, for any colour $c\in C(K_{2n+1})$, we have $\E(|E_c(V)|)=p^2(2n+1)$. Note that $|E_c(V)|$ is $2$-Lipschitz and affected by $\leq 2n+1$ coordinates.
By Azuma's inequality, we have $\P(|E_c(V)|\neq (1\pm \gamma)2p^2n )\leq e^{-\gamma^2p^4n/100}= o(n^{-1})$.
The result follows by a union bound over all the colours.
\end{proof}

Recall that  for any two sets $U,V\subseteq V(G)$ inside a coloured graph $G$, we say that the pair $(U,V)$ is \emph{$k$-replete} if every colour of $G$ occurs at least $k$ times between $U$ and $V$.
We will use the following auxiliary lemma about how this property is inherited by random subsets.
\begin{lemma}[Repletion between random sets]\label{Lemma_inheritence_of_lower_boundedness_random}
Let $q,  p\gg n^{-1}$  and let  $K_{2n+1}$ be $2$-factorized.
Suppose that $A,B\subseteq V(K_{2n+1})$ are disjoint randomized sets with the pair $(A,B)$   $pn$-replete with high probability.
Let $V\subseteq V(K_{2n+1})$ be $q$-random and independent of $A,B$.
Then with high probability the pair $(A, B\cap V)$ is $(qpn/2)$-replete.
\end{lemma}
\begin{proof}
Fix some choice $A'$ of $A$ and $B'$ of $B$ for which the pair $(A',B')$ is $pn$-replete. As $V$ is independent of $A,B$, for each edge $e$ between $A$ and $B$ we have $\P(e\cap B\cap V\neq \emptyset|A=A', B=B')=q$.
Therefore, for any colour $c$, conditional on ``$A=A', B=B'$'', we have $\E(|E_c(A,B\cap  V)|)=q|E_c(A,B)|\geq qpn$. Note that $|E_c(A,B\cap  V)|$ is $2$-Lipschitz and affected by $\leq 2n+1$ coordinates.
By Azuma's inequality, we have $\P(|E_c(A,B\cup V)|< qpn/2| A=A', B=B')\leq e^{-q^2p^2n/100}= o(1)$. Thus, with probability $1-o(1)$, conditioned on  $A=A', B=B'$ we have that $(A,B\cap V)$ is $(qn/2)$-replete.

This was all under the assumption that $A=A'$ and $B=B'$. 
Therefore using that $(A,B)$ is $pn$-replete with high probability, we have
\begin{align*}
\P(\text{$(A,B\cap V)$ is $(qn/2)$-replete})&\geq \sum_{\substack{(A',B') \\ \text{ $pn$-replete}}}\P(\text{$(A,B\cap V)$ is  $(qn/2)$-replete}|A=A', B=B')\cdot\P(A=A', B=B')\\
&\geq  \sum_{\substack{(A',B') \\ \text{ $pn$-replete}}} (1-o(1))\cdot \P(A=A', B=B')=1-o(1).\qedhere
\end{align*}
\end{proof}


\subsubsectionme{Independent vertex/colour sets}
The setting of the next three lemmas is the same: we independently choose a $p$-random set of vertices $V$ and a $q$-random set of colours $C$. For such a pair $V,C$ we expect all vertices of the vertices $v$ in $K_{2n+1}$ to have many $C$-edges going into $V$. Each of the following lemmas is a variation on this theme.
\begin{lemma}[Degrees into independent vertex/colour sets]\label{Lemma_high_degree_into_random_set}
Let $p, q\gg n^{-1}$ and let  $K_{2n+1}$ be $2$-factorized.
Let $V\subseteq V(K_{2n+1})$ be $p$-random, and let $C\subseteq C(K_{2n+1})$ be $q$-random and independent of $V$. With  probability $1-o(n^{-1})$, every vertex $v\in V(K_{2n+1})$ has $|N_C(v)\cap V|\geq pq n$.
\end{lemma}
\begin{proof} Let $v\in V(K_{2n+1})$.
For any vertex $x\neq v$, we have $\P(x\in N_C(v)\cap V)= pq$ and so $\E(|N_C(v)\cap V|)= 2pq n$.
Also $|N_C(v)\cap V|$ is $2$-Lipschitz and affected by $3n$ coordinates.
By Azuma's Inequality, we have that $\P(|N_C(v)\cap V|\leq pq n)\leq 2e^{-p^2q^{2}n/1000}= o(n^{-2})$. The result follows by taking a union bound over all $v\in V(K_{2n+1})$.
\end{proof}

\begin{lemma}\label{lem:goodpairs} Let $1/n\ll \eta \ll \mu$. Let $K_{2n+1}$ be 2-factorized. Suppose that $V_0\subset V(K_{2n+1})$ and $D_0\subset C(K_{2n+1})$ are $\mu$-random subsets which are independent. With high probability, for each distinct $u,v\in V(K_{2n+1})$, there are at least $\eta n$ colours $c\in D_0$ for which there are colour-$c$ neighbours of both $u$ and $v$ in $V_0$.
\end{lemma}
\begin{proof} Let $u,v\in V(K_{2n+1})$ be distinct, and let $X_{u,v}$ be the number of colours $c\in D_0$ for which there are colour-$c$ neighbours of both $u$ and $v$ in $V_0$. Note that $\E X_{u,v} \geq \mu^3 n$, $X_{u,v}$ is
$2$-Lipschitz and affected by $3n-1$ coordinates.
By Azuma's Inequality, we have that $\P(X_{u,v}\leq \mu^3 n/2)\leq 2e^{-\mu^6 n/1000}= o(n^{-2})$. The result follows by taking a union bound over all distinct pairs $u,v\in V(K_{2n+1})$.
\end{proof}

Lemma~\ref{Lemma_high_degree_into_random_set} says that, with high probability, every vertex $v$ has many colours $c\in C$ for which there is a $c$-edge into $V$. The following lemma is a strengthening of this. It shows that, for any set $Y$ of $100$ vertices, there are many colours $c\in C$ for which \emph{each $v\in Y$} has a $c$-edge into $V$.
\begin{lemma}[Edges into independent vertex/colour sets]\label{Lemma_edges_into_independent_vertexcolour_sets} 
Let $p\gg q\gg n^{-1}$ and  let  $K_{2n+1}$ be $2$-factorized.
Let $V\subseteq V(K_{2n+1})$ and $C\subseteq C(K_{2n+1})$ be $p$-random and independent.
Then, with high probability, for any set $Y$ of $100$ vertices, there are $qn$ colours $c\in C$ for which each $y\in Y$ has a $c$-neighbour in $V$.
\end{lemma}
\begin{proof} Fix $Y\subset V(K_{2n+1})$ with $|Y|=100$.
Let $C_Y=\{c\in C: \mbox{each $y\in Y$ has a $c$-neighbour in $V$}\}$.
For any colour $c$ without edges inside $Y$, we have $\P(c\in C_Y)\geq p^{101}$ and so $\E(|C_Y|)\geq p^{101}(n-\binom{|Y|}2)\geq 2qn$. Notice that $|C_Y|$ is $100$-Lipschitz and affected by $3n+1$ coordinates.
By Azuma's inequality, we have that $\P(|C_Y|\leq qn)\leq e^{-q^2n/10^6}= o(n^{-100})$. The result follows by taking a union bound over all sets $Y\subset V(K_{2n+1})$ with $|Y|=100$.
\end{proof}


\subsection{Rainbow matchings}\label{Section_matchings}
We now gather lemmas for finding large rainbow matchings in random subsets of coloured graphs, despite dependencies between the colours and the vertices that we use. Simple greedy embedding strategies are insufficient for this, and instead we will use a variant of R\"odl's Nibble proved by the authors in~\cite{montgomery2018decompositions}.

\begin{lemma}[\cite{montgomery2018decompositions}]\label{Lemma_MPS_nearly_perfect_matching}
Suppose that we have $n, \delta, \gamma, p, \ell$ with $1\geq \delta \gg  p \gg \gamma \gg n^{-1}$ and $n\gg \ell$.

Let $G$ be a locally $\ell$-bounded,  globally $(1+\gamma) \delta n$-bounded,  coloured, balanced bipartite graph with $|G|=(1\pm \gamma)2n$ and $d_G(v)=(1\pm \gamma)\delta n$ for all $v\in V(G)$. Then $G$ has  a random rainbow matching $M$ which has size $\geq (1-2p)n$ where
\begin{align}
\P(e\in E(M))&\geq(1-  9p)\frac{1}{\delta n}\hspace{0.5cm}\text{ for each $e\in E(G)$.} \label{Eq_Near_Matching_Edge_Probability_Lower_Bound}
\end{align}
\end{lemma}
We remark that in the statement of this lemma \cite{montgomery2018decompositions}, the conditions
``$|G|=(1\pm \gamma)2n$ and $d_G(v)=(1\pm \gamma)\delta n$ for all $v\in V(G)$'' are referred to collectively as ``$G$ is $(\gamma, \delta, n)$-regular''.
The following lemma is at the heart of the proofs in this paper. It shows there is typically a nearly-perfect rainbow matching using random vertex/colour sets. Moreover, it allows \emph{arbitrary} dependencies between the sets of vertices and colours. As mentioned before, when embedding high degree vertices such dependencies are unavoidable. Because of this, after we have embedded the high degree vertices, the remainder of the tree will be embedded using variants of this lemma.

\begin{lemma}[Nearly-perfect matchings]\label{Lemma_nearly_perfect_matching}
Let $p\in [0,1]$, $\beta \gg n^{-1}$, and let  $K_{2n+1}$ be $2$-factorized.
Let $V\subseteq V(K_{2n+1})$ be $p/2$-random and let $C\subseteq C(K_{2n+1})$ be $p$-random (possibly depending on each other). Then, with  probability $1-o(n^{-1})$, for every $U\subseteq V(K_{2n+1})\setminus V$  with $|U|\leq pn$, $K_{2n+1}$ has a $C$-rainbow matching of size $|U|-\beta n$ from $U$ to $V$.
\end{lemma}
\begin{proof}
The lemma is vacuous when $p<\beta$, so suppose $p\geq \beta$. We will first prove the lemma in the special case when $p\leq 1-\beta$.
 Choose $p\geq \beta\gg \alpha \gg \gamma\gg n^{-1}$.
With probability $1-o(n^{-1})$, $V$ and $C$ satisfy the conclusion of Lemma~\ref{Lemma_nearly_regular_subgraph}  with $p, \gamma, n$.
Using Chernoff's bound and $p\leq 1-\beta$, with probability $1-o(n^{-1})$ we have $|V|\leq n$. By the union bound, both of these simultaneously occur.
Notice that it is sufficient to prove the lemma for sets $U$ with $|U|=pn$ (since any smaller set $U$  is contained in a set of this size which is disjoint from $V$ as $|V| \leq n$).
From Lemma~\ref{Lemma_nearly_regular_subgraph}, we have that, for $U$ of order $pn$, there are subsets $U'\subseteq U, V'\subseteq V, C'\subseteq C$ with $|U'|=|V'|= (1\pm\gamma)pn$
so that $G=K_{2n+1}[U',V',C']$ is globally $(1+\gamma)p^2n$-bounded, and every vertex $v\in V(G)$ has $d_{G}(v)=(1\pm\gamma)p^2n$. Now $G$ satisfies the assumptions of Lemma~\ref{Lemma_MPS_nearly_perfect_matching} (with $n'=pn$, $\delta =p$, $p'= \alpha$, $\gamma'=2\gamma$ and $\ell=2$), so it has a rainbow matching of size  $(1-2\alpha)pn\geq pn-\beta n$.

Now suppose that $p\geq 1-\beta$. Choose a $(1-\beta/2)p/2$-random subset $V'\subseteq V$ and a $(1-\beta/2)p$-random subset $C'\subseteq C$. Fix $p'=(1-\beta/2)p$ and $\beta'=\beta/2$ and note that $p' \leq 1-\beta'$. By the above argument again, with high probability the conclusion of the ``$p'\leq (1-\beta')$'' version of the lemma applies to $V',C',p',\beta'$. Let $U$ be a set with $|U| \leq pn$. Choose $U'\subseteq U$ with $|U'|=|U|-\beta n/2$. Then $|U'| \leq pn - \beta n/2 \leq p'n$.
From the ``$p'\leq (1-\beta')$'' version of the lemma we get a $C'$-rainbow matching $M$ from $U'$ to $V'$ of size $|U'|-\beta n/2= |U|-\beta n$.
\end{proof}

The following variant of Lemma~\ref{Lemma_nearly_perfect_matching} finds a rainbow matching which completely covers the deterministic set $U$. To achieve this we introduce a small amount of independence between the vertices/colours which are used in the matching. 
\begin{lemma}[Perfect matchings]\label{Lemma_sat_matching_random_embedding}
Let $1\geq \gamma \gg n^{-1}$, let $p\in[0,1]$,  and let  $K_{2n+1}$ be $2$-factorized.
Suppose that we have disjoint sets $V_{dep}, V_{ind}\subseteq V(K_{2n+1})$, and $C_{dep}, C_{ind}\subseteq C(K_{2n+1})$  with $V_{dep}$ $p/2$-random, $C_{dep}$ $p$-random, and $V_{ind}, C_{ind}$ $\gamma$-random. Suppose that $V_{ind}$ and $C_{ind}$ are independent of each other. Then, the following holds with  probability $1-o(n^{-1})$.

For every $U\subseteq V(K_{2n+1})\setminus (V_{dep}\cup V_{ind})$ of order $\leq p n$, there is a perfect $(C_{dep}\cup C_{ind})$-rainbow matching from $U$ into $(V_{dep}\cup V_{ind})$.
\end{lemma}
\begin{proof}
Choose $\beta$ such that $1\geq \gamma \gg \beta\gg n^{-1}$.
With  probability $1-o(n^{-1})$, we can assume the conclusion of
Lemma~\ref{Lemma_nearly_perfect_matching} holds for $V=V_{dep}, C=C_{dep}$ with $p=p, \beta=\beta, n=n$, and, by Lemma~\ref{Lemma_high_degree_into_random_set} applied to $V=V_{ind}, C=C_{ind}$ with $p=q=\gamma, n=n$ that the following holds. For each $v\in V(K_{2n+1})$, we have $|N_{C_{ind}}(v) \cap V_{ind}| \geq \gamma^2n>\beta n$. We will show that the property in the lemma holds. 

Let then $U\subset V(K_{2n+1})\setminus (V_{dep}\cup V_{ind})$ have order $\leq p n$.
From the conclusion of Lemma~\ref{Lemma_nearly_perfect_matching}, there is  a $C_{dep}$-rainbow matching $M_1$ of size $|U|-\beta n$ from $U$ to $V_{dep}$. 
Since $|N_{C_{ind}}(v) \cap V_{ind}| >\beta n$ for each $v\in U\setminus V(M_1)$ we can construct a $C_{ind}$-rainbow matching $M_2$ into $V_{ind}$ covering $U\setminus V(M_1)$ (by greedily choosing this matching one edge at a time). The matching $M_1\cup M_2$ then satisfies the lemma.
\end{proof}

We will also use a lemma about matchings using an exact set of colours.
\begin{lemma}[Matchings into random sets using specified colours]\label{Lemma_matching_into_random_set_using_specified_colours}
Let $p\gg q\gg\beta \gg n^{-1}$ and let  $K_{2n+1}$ be $2$-factorized.
Let  $V\subseteq V(K_{2n+1})$  be $(p/2)$-random. With high probability, for any  $U\subseteq V(K_{2n+1})\setminus V$ with $|U|\geq pn$, and any $C\subseteq C(K_{2n+1})$ with $|C|\leq q n$, there is a $C$-rainbow matching of size $|C|-\beta n$ from $U$ to $V$.
\end{lemma}
\begin{proof}
Choose $\gamma$ such that $\beta\gg\gamma \gg n^{-1}$.
By Lemma~\ref{MPSvertex} (applied with $n'=2n+1$), with high probability, we have that, for any set $A\subset V(K_{2n+1})\setminus V$ with $|A|\geq pn\geq (2n+1)^{1/4}$, for all but at most $\gamma n$ colours there are at most $(1+\gamma)p|A|$ edges of that colour between $A$ and $V$. By Chernoff's bound, with high probability $|V|=(1\pm \gamma)pn$. We will show that the property in the lemma holds.

Fix then an arbitrary pair $U$, $C$ as in the lemma. Without loss of generality, $|U|=pn$. Let $M$ be a maximal $C$-rainbow matching from $U$ to $V$. We will show that $|M|\geq |C|-\beta n$, so that the matching required in the lemma must exist (by removing edges if necessary). Now, let $C'=C\setminus C(M)$. For each $c\in C'$, any edge between $U$ and $V$ with colour $c$ must have a vertex in $V(M)$, by maximality. Therefore, there are at most $2|V(M)|\leq 4qn$ edges of colour $c$ between $U$ and $V$. From the property from Lemma~\ref{MPScolour}, there  are at most $\gamma n$ colours with more than $(1+\gamma)p|U|$ edges between $U$ and $V$. Therefore, using $|V|=(1\pm \gamma)pn$, 
\[
|U||V|\leq 4qn|C'|+(2n+1)\gamma n+(1+\gamma)p|U|(n-|C'|)\leq (4qn-p^2n)|C'|+3\gamma n^2+ (1+\gamma)(1+2\gamma)|U||V|,
\]
and hence
\[
|C'|p^2n/2\leq |C'|(p^2-4q)n\leq ((1+\gamma)(1+2\gamma)-1)|U||V|+3\gamma n^2\leq 7\gamma n^2.
\]
It follows that $|C'|\leq 14\gamma n/p^2\leq \beta n$. Thus, $|M|= |C|-|C'|\geq |C|-\beta n$, as required.
\end{proof}

\subsection{Rainbow star forests}\label{Section_stars}
Here we develop techniques for embedding the high degree vertices of trees, based on our previous methods in~\cite{MPS}. We will do this by proving lemmas about large star forests in coloured graphs. In later sections, when we find rainbow trees, we isolate a star forest of edges going through high degree vertices,  and embed them using the techniques from this section.
We start from the following lemma.

\begin{lemma}[\cite{MPS}]\label{Corollary_MPS_kdisjstars} Let $0<\epsilon<1/100$ and  $\ell\leq \epsilon^{2}n/2$. Let $G$ be an $n$-vertex graph with minimum degree at least $(1-\epsilon)n$  which contains an independent set on the distinct vertices $v_1,\ldots,v_\ell$.  Let $d_1,\ldots,d_\ell\geq 1$ be integers satisfying $\sum_{i\in [\ell]}d_i\leq (1-3\epsilon)n/k$,
and suppose $G$ has a locally $k$-bounded edge-colouring.

 Then, $G$ contains disjoint stars $S_1,\ldots,S_\ell$ so that, for each $i\in [\ell]$, $S_i$ is a star centered at $v_i$ with $d_i$ leaves, and $\cup_{i\in [\ell]}S_i$ is rainbow.
\end{lemma}

The following version of the above lemma will be more convenient to apply.
\begin{lemma}[Star forest]\label{Lemma_star_forest}
Let $1\gg \eta \gg\gamma\gg n^{-1}$ and let  $K_{2n+1}$ be $2$-factorized.
Let $F$ be a star forest with degrees $\geq 1$ whose set of centers is $I=\{i_1, \dots, i_{\ell}\}$ with $e(F)\leq (1-\eta)n$. Suppose we have disjoint sets $J, V\subseteq V(K_{2n+1})$ and $C\subseteq C(K_{2n+1})$ with $|V|\geq (1-\gamma)2n$, $|C|\geq (1-\gamma)n$ and  $J=\{j_1, \dots, j_{\ell}\}$.

Then, there is a $C$-rainbow copy of $F$ with $i_t$ copied to $j_t$, for each $t\in [\ell]$, whose vertices outside of $I$ are copied to vertices in $V$.
\end{lemma}
\begin{proof}
Choose $\epsilon$ such that $\eta \gg \epsilon\gg\gamma$.
Let $G$ be the subgraph of $K_n$ consisting of edges touching $V$ with colours in $C$. Notice that $\delta(G)\geq (1-4\gamma)2n\geq (1-\epsilon)(2n+1)$. Since $V$ and $J$ are disjoint, $J$ contains no edges in $G$. Notice that,  as $J$ and $V$ are disjoint, $\epsilon\gg \gamma$, and $|V|\geq (1-\gamma)2n$, we have $\ell=|J|\leq 2\gamma n+1\leq \epsilon^2 (2n+1)/2$. Let $k=2$ and let $d_1, \dots, d_{\ell}\geq 1$ be the degrees of $i_1,\ldots,i_\ell$ in $F$. Notice that $\eta \gg\epsilon$ implies $\sum_{i=1}^{\ell}d_i=e(F)\leq  (1-\eta)n\leq (1-3\epsilon)(2n+1)/k$.
Applying Lemma~\ref{Corollary_MPS_kdisjstars} to $G$ with $\{v_1, \dots, v_{\ell}\}=J$, $n'=2n+1$, $k=2$,  $\ell=\ell$, $\epsilon=\epsilon$  we find the required rainbow star forest.
\end{proof}

The above lemma can be used to find a rainbow copy of any star forest $F$ in a $2$-factorization as long as there are more than enough colours for a rainbow copy of $F$. However, we  also want this rainbow copy to be suitably randomized. This is achieved by finding a star forest larger than $F$ and then randomly deleting each edge independently.
The following lemma is how we embed rainbow star forests in this paper. It shows that we can find a rainbow copy of any star forest so that the unused vertices and colours are $p$-random sets. Crucially, and unavoidably, the sets of unused vertices/colours depend on each other. This is where the need to consider dependent sets arises.
\begin{lemma}[Randomized star forest]\label{Lemma_randomized_star_forest}
Let $1\geq p, \alpha\gg \gamma\gg d^{-1}, n^{-1}$ and $\log^{-1} n\gg d^{-1}$. Let  $K_{2n+1}$ be $2$-factorized.
Let $F$ be a star forest with degrees $\geq d$   with $e(F)= (1-p)n$ whose set of centers is $I=\{i_1, \dots, i_{\ell}\}$. Suppose we have disjoint sets $V,J\subseteq V(K_{2n+1})$ and $C\subseteq C(K_{2n+1})$ with $|V|\geq (1-\gamma)2n$, $|C|\geq (1-\gamma)n$ and  $J=\{j_1, \dots, j_{\ell}\}$

Then, there is a randomized subgraph $F'$  which is with high probability a $C$-rainbow copy of $F$,  with $i_t$ copied to $j_t$ for each $t$ and whose  vertices outside $I$ are copied to vertices in $V$. Additionally there are randomized sets $U\subseteq V\setminus V(F'), D\subseteq C\setminus C(F')$ such that $U$ is a $(1-\alpha)p$-random subset of $V$ and $D$ is a $(1-\alpha)p$-random subset of $C$ (with $U$ and $D$ allowed to depend on each other).
\end{lemma}

\begin{proof}
Choose $\eta$ such that $p, \alpha \gg \eta\gg  \gamma$.
Let $\hat F$ be a star forest obtained from $F$ by replacing every star $S$ with a star $\hat S$ of size $e(\hat S)=e(S)(1-\eta)/(1-p)$. Notice that $e(\hat F)=(1-\eta)n$, and, for each vertex $v$ at the centre of a star, $S$ say, in $F$, we have $d_{\hat F}(v)=e(\hat S)=e(S)(1-\eta)/(1-p)=d_F(v)(1-\eta)/(1-p)$. By Lemma~\ref{Lemma_star_forest}, there  is a $C$-rainbow embedding $\hat F'$ of $\hat F$   with $i_t$ copied to $j_t$ for each $t$  and   whose  vertices outside $I$ are contained in $V$.

Let $U$ be a $(1-\alpha)p$-random subset of $V$. Let $F'=\hat F' \setminus U$.  By Chernoff's Bound and $\log^{-1}n, p, \gamma \gg d^{-1}$, we have $\P(d_{F'}(v)< (1-\gamma)(1-p+\alpha p)d_{\hat F'}(v))\leq e^{-(1-p+\alpha p)\gamma^2d/3}\leq e^{-\log^5 n}$ for each center $v$ in $F$. Note that, for each centre $v$, $(1-\gamma)(1-p+\alpha p)d_{\hat F'}(v)=(1-\gamma)(1-p+\alpha p) d_F(v)(1-\eta)/(1-p) \geq d_F(v)$, where this holds as $p,\alpha\gg \eta\gg \gamma$.  Taking a union bound over all the centers shows that, with high probability, $F'$ contains a copy of $F$.
Since $\hat F'$ was rainbow, we have that $C(\hat F')\setminus C(F')$ is a $(1-\alpha)p$-random subset of $C(\hat F')$.
Let $\hat{C}$ be a $(1-\alpha)p$-random subset of $C\setminus C(\hat F')$ and set $D=\hat{C}\cup (C(\hat F')\setminus C(F'))$. Now $D$ and $U$ are both $(1-\alpha)p$-random subsets of $C$ and $V$ respectively.
\end{proof}

\subsection{Rainbow paths}\label{Section_paths}
Here we collect lemmas for finding rainbow paths and cycles in random subgraphs of $K_{2n+1}$. First we prove two lemmas about short paths between prescribed vertices. These lemmas are later used to incorporate larger paths into a tree.

\begin{lemma}[Short paths between two vertices]\label{Lemma_short_paths_between_two_vertices}  Let $p\gg \mu \gg n^{-1}>0$ and suppose $K_{2n+1}$ is $2$-factorized. Let $V\subset V(K_{2n+1})$ and $C\subset C(K_{2n+1})$ be $p$-random and independent. Then, with high probability, for each pair of distinct vertices $u,v\in V(K_{2n+1})$ there are at least
$\mu n$ internally vertex-disjoint  $u,v$-paths with length $3$ and internal vertices in $V$ whose union is $C$-rainbow.
\end{lemma}
\begin{proof}
Choose $p\gg  \mu \gg n^{-1}>0$.
Randomly partition $C=C_1\cup C_2\cup C_{3}$ into three $p/3$-random sets and $V=V_1\cup V_2$ into two $p/2$-random sets. With high probability the following simultaneously hold.

$\bullet$ By Lemma~\ref{Lemma_MPS_boundrandcolour} applied to $C_3$ with $p=p/3, \epsilon=1/2, n=2n+1$,  for any disjoint vertex sets $U,V$ of order at least $p^2n/10\geq (2n+1)^{3/4}$, we have $e_{C_3}(U,V)\geq p|U||V|/6\geq 10^{-3}p^5n^2$.

$\bullet$ By Lemma~\ref{Lemma_high_degree_into_random_set} applied to $C_i, V_j$ with $p=p/2, q=p/3, n=n$ we have $|N_{C_{i}}(v)\cap V_j| \geq p^2n/6$ for every $v\in V(K_{2n+1})$, $i\in [3]$, and $j\in [2]$.

We claim the property holds. Indeed, pick an arbitrary distinct pair of vertices $u,v\in V(K_{2n+1})$. Let $M$ be a maximum $C_3$-rainbow matching between $N_{C_{1}}(v)\cap (V_1\setminus \{u\})$ and $N_{C_{2}}(u)\cap (V_2\setminus \{v\})$ (these sets have size at least $p^2n/24$). Each of the $2e(M)$ vertices in $M$ has $2n$ neighbours in $K_{2n+1}$, and each colour in $M$ is on $2n+1$ edges in $K_{2n+1}$. The number of edges of $K_{2n+1}$ sharing a vertex or colour with $M$ is thus $\leq 7ne(M)$. By maximality, and the property from Lemma~\ref{Lemma_MPS_boundrandcolour}, we have $7ne(M)\geq e_{C_3}(U,V) \geq 10^{-3}p^5n^2$, which implies that $e(M)\geq 10^{-4}p^5n \geq  4\mu n$.
For any edge $v_1v_2$ in the matching $M$, the path $uv_1v_2v$ is a rainbow path. In the union of these paths, the only colour repetitions can happen at $u$ or $v$. Since  $K_{2n+1}$ is $2$-factorized, there is a subfamily of $\mu n$ paths which are collectively rainbow.
\end{proof}

We can use this lemma to find many disjoint length $3$ connecting paths.
\begin{lemma}[Short connecting paths]\label{Lemma_few_connecting_paths}  Let  $p\gg q \gg n^{-1}>0$ and let $K_{2n+1}$ be $2$-factorized. Let $V$ be a $p$-random set of vertices, and $C$ a $p$-random set of colours independent from $V$.
Then, with high probability, for any set of $\{x_1, y_1, \dots, x_{q n}, y_{q n}\}$ of vertices,
 there is a collection $P_1, \dots, P_{q n}$ of vertex-disjoint paths with length $3$, having internal vertices in $V$, where $P_i$ is an $x_i,y_i$-path, for each $i\in [qn]$, and $P_1\cup \dots\cup P_{q n}$ is $C$-rainbow.
\end{lemma}
\begin{proof}
By Lemma~\ref{Lemma_short_paths_between_two_vertices} applied to $C,V$ with $p=p, \mu=10q,  n=n$, with high probability, between any $x_i$ and $y_i$, there is a collection of $10q n$ internally vertex disjoint $x_i,y_i$-paths, which are collectively $C$-rainbow, and internally contained in $V$. By choosing such paths greedily one by one, making sure never to repeat a colour or vertex,  we can find the required collection of paths.
\end{proof}


\section{The finishing lemma in Case A}\label{sec:finishA}

The proof of our finishing lemmas uses \emph{distributive} absorption, a technique introduced by the first author in \cite{montgomery2018spanning}. For the finishing lemma in Case A, we start by constructing colour switchers for sets of $\leq 100$ colours of $C$. These are $|C|$ perfect rainbow matchings, each from the same small set $X$ into a larger set $V$ which use the same colour except for one different colour from $C$ per matching (see Lemma~\ref{absorbA}). This gives us a small amount of local variability, but we can build this into a global variability property (see Lemma~\ref{absorbAmacro}). This will allow us to choose colours to use from a large set of colours, but not all the colours, so we will need to find matchings which ensure any colours outside of this are used (see Lemma~\ref{colourcover}). To find the switchers we use a small proportion of colours in a random set. We will have to cover the colours not used in this, with no random properties for the colours remaining (see Lemma~\ref{Lemma_saturating_matching_lemma}). We put this all together to prove Lemma~\ref{lem:finishA} in Section~\ref{sec:finishAfinal}.

\subsection{Colour switching with matchings}\label{sec:switcherspaths}
We start by constructing colour switchers using matchings.
\begin{lemma}\label{absorbA} Let $1/n\ll \beta \ll\xi, \mu$. Let $K_{2n+1}$ be 2-factorized. Suppose that $V_0\subset V(K_{2n+1})$ and $D_0\subset C(K_{2n+1})$ are $\mu$-random subsets which are independent. With high probability, the following holds.

Let $C,\bar{C}\subset C(K_{2n+1})$ and $X,\bar{V},Z\subset V(K_{2n+1})$ satisfy $|\bar{{C}}|,|\bar{V}|\leq \beta n$, $|C|\leq 100$ and suppose that $(X,Z)$ is $(\xi n)$-replete.
Then, there are sets $X'\subset X\setminus \bar{V}$, $C'\subset D_0\setminus \bar{C}$ and $V'\subset (V_0\cup Z)\setminus \bar{V}$ with sizes $|C|$, $|C|-1$ and $\leq 3|C|$ respectively such that, for every $c\in C$, there is a perfect $(C'\cup \{c\})$-matching from $X'$ to $V'$.
\end{lemma}
\begin{proof}
With high probability, by Lemma~\ref{lem:goodpairs} we have the following property.

\stepcounter{propcounter}
\begin{enumerate}[label = {\bfseries \Alph{propcounter}}]
\item For each distinct $u,v\in V(K_{2n+1})$, there are at least $100\beta n$ colours $c\in D_0$ for which there are colour-$c$ neighbours of both $u$ and $v$ in $V_0$.\label{prop:goodpairs}
\end{enumerate}

We will show the property in the lemma holds. Let $C,\bar{C}\subset C(K_{2n+1})$ and $X,\bar{V},Z\subset V(K_{2n+1})$ be sets with $|\bar{{C}}|,|\bar{V}|\leq \beta n$, $|C|\leq 100$ and suppose that $(X,Z)$ is $(\xi n)$-replete.
Let $\ell=|C|\leq 100$ and label $C=\{c_1,\ldots,c_\ell\}$.
For each $i\in [\ell]$, using that there are at least $\xi n$ edges with colour $C$ between $X$ and $Z$, and $\xi\gg \beta$, $\ell\leq 100$ and $|\bar{V}|\leq \beta n$, pick a vertex $x_i\in X\setminus (\bar{V}\cup \{x_1,\ldots,x_{i-1}\})$ which has a $c_i$-neighbour $y_i\in Z\setminus (\bar{V}\cup \{y_1,\ldots,y_{i-1}\})$.
For each $1\leq i\leq \ell-1$, using~\ref{prop:goodpairs}, pick a colour $d_i\in D_0\setminus (\bar{C}\cup C\cup  \{d_1,\ldots,d_{i-1}\})$
and (not necessarily distinct) vertices $z_i,z'_i\in V_0\setminus (\bar{V}\cup \{y_1,\ldots,y_{\ell},z_1,\ldots,z_{i-1},z_1',\ldots,z_{i-1}'\})$ such that $x_iz_i$ and $x_{i+1}z'_{i}$ are both colour $d_i$. Let $C'=\{d_1,\ldots,d_{\ell-1}\}$, $X'=\{x_1,\ldots,x_{\ell}\}$ and $V'= \{y_1,\ldots,y_{\ell},z_1,\ldots,z_{\ell-1},z_1',\ldots,z_{\ell-1}'\}$. See Figure~\ref{Figure_Matching_Switching} for an example of the edges that we find. 
\begin{figure}
  \centering
    \includegraphics[width=0.8\textwidth]{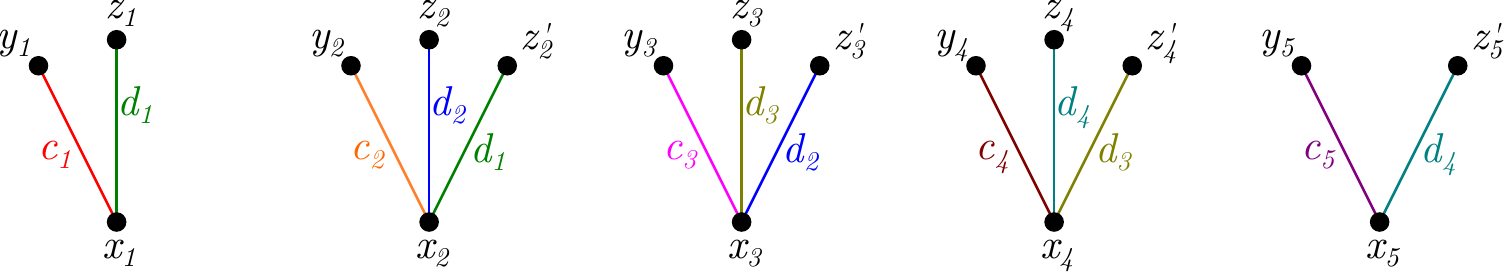}
  \caption{An example of the graph we find in Lemma~\ref{absorbA} when $\ell=5$. The key property it has is that for any colour $c_i$, there is a matching covering $x_1, \dots, x_{\ell}$ using $c_i$ and all the colours $d_{1}, \dots, d_{\ell-1}$.}
\label{Figure_Matching_Switching}
\end{figure}

Now, for each $j\in [\ell]$, let $M_j=\{x_jy_j\}\cup \{x_iz_i:i<j\}\cup \{x_{i+1}z_i':j\leq i\leq \ell-1\}$. Note that $M_j$ is a perfect $(C\cup \{c_j\})$-matching from $X'$ to $V'$, and thus $X'$, $C'$ and $V'$ have the property required.
\end{proof}

\subsection{Distributive absorption with matchings}
We now put our colour switchers together to create a global flexibility property. To do this, we find certain disjoint colour switchers, governed by a suitable \emph{robustly matchable bipartite graph}.
This is a bipartite graph which has a lot of flexibility in how one of its parts can be covered by matchings. It exists by the following lemma.
\begin{lemma}[Robustly matchable bipartite graphs, \cite{montgomery2018spanning}]\label{Lemma_H_graph}
There is a constant $h_0 \in \mathbb N$ such that, for every $h \geq h_0$, there exists
a bipartite graph $H$ with maximum degree at most $100$ and vertex classes $X$ and $Y \cup Y'$, with
$|X| = 3h$, and $|Y| = |Y'| = 2h$, so that the following is true. If $Y_0 \subseteq Y'$ and $|Y_0 | = h$, then
there is a matching between $X$ and $Y \cup Y_0$.
\end{lemma}

\begin{lemma}\label{absorbAmacro} Let $1/n\ll \beta \ll \eta\ll \xi,\mu$. Let $K_{2n+1}$ be 2-factorized. Suppose that $V_0\subset V(K_{2n+1})$ and $D_0\subset C(K_{2n+1})$ are $\mu$- and $\eta$-random respectively, and suppose that they are independent. With high probability, the following holds.

Suppose $X,Z\subset V(K_{2n+1})$ are disjoint subsets such that $(X,Z)$ is $(\xi n)$-replete, $\alpha\leq \beta$ and $C\subset C(K_{2n+1})\setminus D_0$ is a set of at most $2\beta n$ colours. Then, there is a set $X_0$ of $|D_0|+\alpha n$ vertices in $X$ such that, for every set $C'\subset C$ of $\alpha n$ colours, there is a perfect $(D_0\cup C')$-rainbow matching from $X_0$ into $V_0\cup Z$.
\end{lemma}
\begin{proof}
With high probability, by Lemmas~\ref{Lemma_Chernoff} and~\ref{absorbA} (applied with $\beta'=10^3\beta, \mu'=\eta$, $\xi'=\xi/2$, and an $\eta$-random subset of $V(K_{2n+1})$ contained in $V_0$) we have the following properties.
\stepcounter{propcounter}
\begin{enumerate}[label = {\bfseries \Alph{propcounter}\arabic{enumi}}]
\item $\eta n/2\leq |D_0|\leq 2\eta n$.\label{aa2}
\item Let $\bar{C},\hat{C}\subset C(K_{2n+1})$ and $X,\bar{V},\bar{Z}\subset V(K_{2n+1})$ satisfy $|\bar{{C}}|,|\bar{V}|\leq 10^3\beta n$, $|\hat{C}|\leq 100$ and suppose that $(X,\bar{Z})$ is $(\xi n/2)$-replete.
Then, there are sets $X'\subset Y\setminus \bar{V}$, $C'\subset D_0\setminus \bar{C}$ and $V'\subset V\cup \bar Z \setminus \bar{V}$ with sizes $|\hat{C}|$, $|\hat{C}|-1$ and $\leq 3|\hat{C}|$ such that, for every $c\in \hat{C}$, there is a perfect $(C'\cup \{c\})$-matching from $Y'$ to $V'$.\label{aa3}
\end{enumerate}

We will show that the property in the lemma holds. For this, let $X,Z\subset V(K_{2n+1})$ be disjoint subsets such that $(X,Z)$ is $(\xi n)$-replete, let $\alpha\leq \beta$ and let $C\subset C(K_{2n+1})\setminus D_0$ be a set of at most $2\beta n$ colours.
Let $h=2\beta n$, and, using \ref{aa2}, pick a set $D_1\subset D_0$ of $(3h-\alpha n)$ colours. 
Noting that $|D_1|\geq 2h$ and $|D_1\cup C|\leq 4h$, we can define sets $Y,Y'$ of order $2h$ with $Y\subseteq D_1$ and $D_1\cup C\subseteq Y\cup Y'$ (where any extra elements of $Y'$ are arbitrary dummy colours which will not be used in the arguments).
Using Lemma~\ref{Lemma_H_graph}, let $H$ be a bipartite graph with vertex classes $[3h]$ and $Y\cup Y'$, which has maximum degree at most 100, and is such that for any $\bar{Y}\subseteq Y'$ with $|\bar{Y}|=h$, there is a perfect matching between $[3h]$ and $Y\cup \bar{Y}$. In particular, since $|D_1|=3h-\alpha n$ and $Y\subseteq D_1$, we have the following.

\begin{enumerate}[label = {\bfseries \Alph{propcounter}\arabic{enumi}}]
\addtocounter{enumi}{2}
\item For each set $C'\subset C$ of $\alpha n$ colours, there is a perfect matching between $[3h]$ and $D_1\cup C'$ in $H$.  \label{Hmatch}
\end{enumerate}

For each $i\in [3h]$, let $D_i=N_H(i)\cap C(K_{2n+1})$ (i.e.\ $D_i$ is the set of non-dummy colours in $N_H(i)$). Using \ref{aa3} repeatedly, find sets $X_i\subset X\setminus (X_1\cup\ldots\cup X_{i-1})$,  $C_i\subset D_0\setminus (D_1\cup C_1\cup\ldots\cup C_{i-1})$ and $V_i\subset (V_0\cup Z)\setminus (V_1\cup \ldots\cup V_{i-1})$,
with sizes $|D_i|$, $|D_i|-1$ and at most $300$, such that the following holds. 
(To see that we can repeat this application of \ref{aa3} this many times, note that $h =2\beta n$ and $\xi \gg \beta$ and so at each application we delete $O(h)$ vertices from $X$ or $Z$,  leaving a pair which is still $(\xi n/2)$-replete.)

\begin{enumerate}[label = {\bfseries \Alph{propcounter}\arabic{enumi}}]
\addtocounter{enumi}{3}
\item For every $c\in D_i$, there is a perfect $(C_i\cup\{c\})$-rainbow matching from $X_i$ into $V_i$. 
\label{propabsorb}
\end{enumerate}

Greedily, using \ref{aa2}, that $h=2\beta n$, $\beta,\eta\ll \xi$ and $(X,Z)$ is $(\xi n)$-replete, find a $(D_0\setminus (D_1\cup C_1\cup\ldots\cup C_{3h}))$-rainbow matching $M$ with $|D_0\setminus (D_1\cup C_1\cup \ldots\cup C_{3h})|$ edges from $X\setminus (X_1\cup\ldots\cup X_{3h})$ into $Z\setminus (X_1\cup\ldots\cup X_h)$.
Let $X_0=(X\cap V(M))\cup X_1\cup \ldots\cup X_{3h}$. Note that $|X_0|=|D_0|-|D_1|-\sum_{i=1}^{3h}(|X_i|-|C_i|)=|D_0|-|D_1|+3h=|D_0|+\alpha n$.
We claim this has the property required. Indeed, suppose $C'\subset C$ is a set of $\alpha n$
colours. Using~\ref{Hmatch}, let $M'$ be a perfect matching between $[3h]$ and $D_1\cup C'$ in $H$, and label $D_1\cup C'=\{c_1,\ldots,c_{3h}\}$ so that, for each $i\in [3h]$, $c_i$ is matched to $i$ in $M'$. Note that the colours $c_i$ are not dummy colours.
By the definition of each $D_i$, $c_i\in D_i$ for each $i\in [3h]$.

By \ref{propabsorb}, for each $i\in [3h]$, there is a perfect $(C_i\cup \{c_i\})$-matching, $M_i$ say, from $X_i$ to $V_i$. Then, $M\cup M_1\cup \ldots\cup M_{3h}$ is a perfect matching from $X_0$ into $V_0\cup Z$ which is $(C(M)\cup D_1\cup C'\cup C_1\cup\ldots\cup C_{3h})$-rainbow, and hence $(D_0\cup C')$-rainbow, as required.
\end{proof}

\subsection{Covering small colour sets with matchings}
We find perfect rainbow matchings covering a small set of colours using the following lemma.
\begin{lemma}\label{colourcover} Let $1/n\ll \nu\ll \lambda\ll 1$. Let $K_{2n+1}$ be 2-factorized. Suppose that $X,V\subset V(K_{2n+1})$ are disjoint,  $(X,V)$ is $(10\nu n)$-replete and $|X|\leq \lambda n$. Suppose that $C\subset C(K_{2n+1})$ is such that every vertex $v\in V(K_{2n+1})$ has at least $3\lambda n$ colour-$C$ neighbours in $V$.
Then, given any set $C'\subset C(K_{2n+1})$ with at most $\nu n$ colours, there is a perfect $(C'\cup C)$-rainbow matching from $X$ to $V$ which uses every colour in $C'$.
\end{lemma}
\begin{proof} Using that $(X,V)$ is $(10\nu n)$-replete, greedily find a matching $M_1$ with $|C'|$ edges from $X$ to $V$ which is $C'$-rainbow. Then, using that every vertex $v\in V(K_{2n+1})$ has at least $3\lambda n$ colour-$C$ neighbours in $V$, greedily find a perfect $C$-rainbow matching $M_2$ from $X\setminus V(M_1)$ to $V\setminus V(M_1)$. This is possible as, when building $M_2$ greedily, each vertex in $X\setminus V(M_1)$ has at most $2|X|\leq 2\lambda n$ neighbouring edges with colour used in $C\cap C(M_1\cup M_2)$ (as the colouring is $2$-bounded) and at most $|X|\leq \lambda n$ colour-$C$ neighbours in $V(M_1)\cup V(M_2)$. 
The matching $M_1\cup M_2$ then has the required property.
\end{proof}

\subsection{Almost-covering colours with matchings}
We find rainbow matchings using almost all of a set of colours using the following lemma.
\begin{lemma}\label{Lemma_saturating_matching_lemma}
Let $1\geq p\gg q\gg \gamma,\eta \gg \nu\gg 1/n$  and let  $K_{2n+1}$ be $2$-factorized.
Suppose we have disjoint sets $V, V_0\subseteq V(K_{2n+1})$, and $D,D_0\subseteq C(K_{2n+1})$  with $V$ $p$-random, $D$ $q$-random, and $V_0, D_0$ $\gamma$-random. Furthermore, suppose that $V_0$ and $D_0$ are independent of each other. Then, the following holds with  probability $1-o(n^{-1})$.

For every $U\subseteq V(K_{2n+1})\setminus (V\cup V_0)$ with at most $(q+\gamma+\eta-\nu)n$ vertices, and any $C\subset C(K_{2n+1})$ with $D\cup D_0\subset C$ with $|C|\geq |U|+\nu n$, there is a perfect $C$-rainbow matching from $U$ into $V\cup V_0$.
\end{lemma}
\begin{proof}
Choose $\beta$ such that $\nu\gg \beta\gg n^{-1}$.
Partition $V=V_1\cup V_2$ so that $V_1$ and $V_2$ are $(p/2)$-random. Partition $D_0=D_1\cup D_2$  with $D_1$ $(\gamma-\nu/4)$-random and  $D_2$ $(\nu/4)$-random. We now find properties \ref{grape1} -- \ref{grape4}, which all hold with probability $1-o(n^{-1})$, as follows.

By Lemma~\ref{Lemma_matching_into_random_set_using_specified_colours} with $V=V_1$,   $p'=q/2$, $q=\gamma+2\eta$, $\beta=\beta$, and $n=n$, we have the following. 

\addtocounter{propcounter}{1}
\begin{enumerate}[label = {\bfseries \Alph{propcounter}\arabic{enumi}}]
\item For any $U\subset V(K_{2n+1})\setminus V_1$ with $|U|\geq qn/2$ and any $C'\subset C(K_{2n+1})$ with $|C'|\leq (\gamma +2\eta)n$, there is a $C'$-rainbow matching of size $|C'|-\beta n$ from $U$ into $V_1$. \label{grape1}
\end{enumerate}
By Lemma~\ref{Lemma_nearly_perfect_matching} with $V=V_2$, $C=D$, $p=q$, $\beta=\beta$, and $n=n$, we have the following. 
\begin{enumerate}[label = {\bfseries \Alph{propcounter}\arabic{enumi}}]\stepcounter{enumi}
\item For any $U\subset V(K_{2k+1})\setminus V_2$ with $|U|\leq q n$, there is a $D$-rainbow matching from $U$ into $V_2$ with size $|U|-\beta n$.\label{grape2b}
\end{enumerate}
By Lemma~\ref{Lemma_high_degree_into_random_set} with $V=V_0$, $C=D_2$, $p=\gamma$, $q=\nu/4$, and $n=n$ we have the following.
\begin{enumerate}[label = {\bfseries \Alph{propcounter}\arabic{enumi}}]\addtocounter{enumi}{2}
\item Every vertex $v\in V(K_{2n+1})$ has $|N_{D_2}(v)\cap V_0|\geq \gamma\nu n/4$.\label{grape3}
\end{enumerate}
By Lemma~\ref{Lemma_Chernoff}, we have the following.
\begin{enumerate}[label = {\bfseries \Alph{propcounter}\arabic{enumi}}]\addtocounter{enumi}{3}
\item $|D_1|=(\gamma-\nu/4\pm \beta)n\leq \gamma n$, $|D|=(q\pm \beta)n$ and $|D_2| \leq \nu n/3$.\label{grape4}
\end{enumerate}

We now show that the property in the lemma holds. For this, let $U\subseteq V(K_{2n+1})\setminus (V\cup V_0)$ have at most $(q+\gamma+\eta-\nu)n$ vertices, and let $C\subset C(K_{2n+1})$ satisfy   $D\cup D_0\subset C$ and $|C|\geq |U|+\nu n$.
Note that we can assume that $|C|=|U|+\nu n$, so that, using \ref{grape4}, $|C\setminus (D \cup D_2)|\leq |U|+\nu n -|D| \leq (q+\gamma+\eta-\nu)n+\nu n-(q-\beta)n= (\gamma+\eta +\beta)n\leq (\gamma+2\eta)n$. Notice that, by \ref{grape4} and as $D\subset C$, we have $|U|\geq |C|-\nu n\geq q n/2$.
Therefore, by \ref{grape1}, there is a $(C\setminus (D \cup D_2))$-rainbow matching $M_1$ of size $|C\setminus (D \cup D_2)|-\beta n$ from $U$ into $V_1$. Now, using \ref{grape4}, $|U\setminus V(M_1)|=|U|-|C\setminus (D \cup D_2)|+\beta n\leq |D|+|D_2|-\nu n+\beta n\leq (q+\beta)n+\nu n/3-\nu n+\beta n\leq q n$. By \ref{grape2b}, there is a $D$-rainbow matching $M_2$ of size $|U\setminus V(M_1)| -\beta n$ from $U\setminus V(M_1)$ into $V_2$.
Notice that  $|U\setminus (V(M_1)\cup V(M_2))|\leq \beta n$.
From \ref{grape3}, we have that $|N_{D_2}(v)\cap V_0|\geq \gamma \nu n/4> 3 \beta n \geq 3|U\setminus (V(M_1)\cup V(M_2))|$ for every vertex $v\in V(K_{2n+1})$. By greedily choosing neighbours of $u\in U\setminus (V(M_1)\cup V(M_2))$ one at a time (making sure to never repeat colours or vertices) we can find a $D_2$-rainbow matching $M_3$ into $V_0$ covering $U\setminus (V(M_1)\cup V(M_2))$. Indeed, when we seek a new colour-$D_2$ neighbour of a vertex $u\in U\setminus (V(M_1)\cup V(M_2))$ there are at most 2 neighbours of $u$ of each colour used from $D_2$ in $V_0$, ruling out at most $2|U\setminus (V(M_1)\cup V(M_2))|$ colour-$D_2$ neighbours of $u$ in $V_0$, while at most $|U\setminus (V(M_1)\cup V(M_2))|$ vertices have been used in $V_0$ in the matching.
Now the matching $M_1\cup M_2\cup M_3$ satisfies the lemma.
\end{proof}

\subsection{Proof of the finishing lemma in Case A}\label{sec:finishAfinal}
Finally, we can put all this together to prove Lemma~\ref{lem:finishA}.
\begin{proof}[Proof of Lemma~\ref{lem:finishA}]
Pick $\nu,\lambda,\beta$ and $\alpha$ so that $1/n\ll \nu \ll\lambda\ll\beta\ll \alpha\ll \xi$ and recall that $\xi \ll \mu\ll \eta\ll \eps\ll p\leq 1$. Let $V_1',V_2',V_3'$ be disjoint $(p/3)$-random subsets in $V$ and let $W_1,W_2,W_3$ be disjoint $(\mu/3)$-random subsets in $V_0$.
Let $D_1$, $D_2$, $D_3$ be disjoint $(\mu/3)$-, $\beta$- and $\alpha$-random disjoint subsets in $C_0$ respectively. 

We now find properties \ref{mouse11}--\ref{mouse44}, which collectively hold with high probability as follows.
First note that $(1-\eta)\eps+(\mu/3)+\eta-\nu \geq \eps$. Thus, by Lemma~\ref{Lemma_saturating_matching_lemma} applied with $V=V_1'$, $V_0=W_1$, $D=C$, $D_0=D_1$, $p'=p/3$, $q=(1-\eta)\epsilon$, $\gamma=\mu/3$, $\eta=\eta$ and $\nu=\nu$ we get the following.
\stepcounter{propcounter}
\begin{enumerate}[label = {\bfseries \Alph{propcounter}\arabic{enumi}}]
\item For any $U\subset V(K_{2n+1})\setminus (V'_1\cup W_1)$ with $|U|\leq  \eps n$ and any set of colours $C'\subset C(K_{2n+1})$ with $C\cup D_1\subset C'$, and $|C'|  \geq |U|+\nu n$, there is a perfect $C'$-rainbow matching from $U$ into $V_1'\cup W_1$.\label{mouse11}
\end{enumerate}
By Lemma~\ref{Lemma_high_degree_into_random_set}, we get the following.
\begin{enumerate}[label = {\bfseries \Alph{propcounter}\arabic{enumi}}]\addtocounter{enumi}{1}
\item Each $v\in V(K_{2n+1})$ has at least $\beta^2 n$ colour-$D_2$ neighbours in $W_2$.\label{mouse2}
\end{enumerate}
By Lemma~\ref{absorbAmacro} applied with $V_0=W_3$,
$D_0=D_3$, $\mu'=\mu/3$, $\xi'=\xi/3$, $\eta'=\alpha, \alpha'=\gamma$ and $\beta'=2\beta$, we have the following.
\begin{enumerate}[label = {\bfseries \Alph{propcounter}\arabic{enumi}}]\addtocounter{enumi}{2}
\item For any disjoint vertex sets $Y,Z\subset V(K_{2n+1})\setminus W_3$ for which $(Y,Z)$ is $(\xi n/3)$-replete, any $\gamma\leq 2\beta$ and any set $\hat{C}\subset C(K_{2n+1})\setminus D_3$ of at most $4\beta n$ colours, the following holds. There is a subset $Y'\subset Y$ of $|D_3|+\gamma n$ vertices so that, for any $C'\subset \hat{C}$ with $|C'|=\gamma n$, there is a perfect $(D_3\cup C')$-rainbow matching from $Y'$ into $Z\cup W_3$.\label{mouse3}
\end{enumerate}
Finally, by Lemma~\ref{Lemma_Chernoff}, the following holds.
\begin{enumerate}[label = {\bfseries \Alph{propcounter}\arabic{enumi}}]\addtocounter{enumi}{3}
\item $2\beta n\geq |D_2|\geq \beta n/2$ and $2\alpha n\geq |D_3|\geq \alpha n/2$.\label{mouse44}
\end{enumerate}

We will now show that the property in the lemma holds.
Let then $X,Z\subset V(K_{2n+1})\setminus (V\cup V_0)$ be disjoint and let $D\subset C(K_{2n+1})$, so that $|X|=|D|=\eps n$, $C_0\cup C\subset D$ and $(X,Z)$ is $(\xi n)$-replete. Let $Z_1$, $Z_2$ be disjoint $(1/2)$-random subsets of $Z$. Note that, by Lemma~\ref{Lemma_Chernoff}, $(X,Z_1)$ and $(X,Z_2)$ are with high probability $(\xi/3)$-replete. Thus, we can pick an instance of $Z_1$ and $Z_2$ for which this holds. Let $V_1=V_1'\cup W_1$, $V_2=V_2'\cup W_2\cup Z_1$ and $V_3=V_3'\cup W_3\cup Z_2$.
Set $\gamma=(|D_2|/n)-\lambda$, so that, by \ref{mouse44}, $2\beta\geq \gamma>0$. Then, by \ref{mouse44} and \ref{mouse3} applied with $\hat{C}=D_2$, $Y=X_1, Z=Z_2$ there is a set $X_3\subset X$ with size $|D_3|+\gamma n=|D_3|+|D_2|-\lambda n$ and the following property.

\begin{enumerate}[label = {\bfseries \Alph{propcounter}\arabic{enumi}}]\addtocounter{enumi}{4}
\item For any $C'\subset D_2$ with $|C'|=\gamma n=|D_2|-\lambda n$, there is a perfect $(D_3\cup C')$-rainbow matching from $X_3$ into $V_3$.\label{mouse4}
\end{enumerate}

Let $C_1=D\setminus (D_2\cup D_3)$, $C_2=D_2$ and $C_3=D_3$, and note that $C\cup D_1\subset C_1$. 
Let $X'=X\setminus X_3$, so that $|X'|=(\eps +\lambda)n-|D_3|-|D_2|=|C_1|+\lambda n$. Using \ref{mouse44}, we have $|X_3|\leq |D_2|+|D_3|\leq \xi n/10$, and hence $(X',Z_1)$ is $(\xi n/6)$-replete.
Let $X_2'\subset X_2$ be a $(\lambda n/|X'|)$-random subset of $X'$. By Lemma~\ref{Lemma_Chernoff}, $|X_2'|\leq (\lambda+\nu)n$ and $(X_2',Z_1)$ is $(10\nu n)$-replete. Choose disjoint sets $X_1$ and $X_2$ of $X'$ with size $|X'|-(\lambda+\nu)n=|C_1|-\nu n$ and $(\lambda+\nu)n$ respectively, and so that $X_2'\subset X_2$. Note that $(X_2,Z_1)$ is $(10\nu n)$-replete. Therefore, by Lemma~\ref{colourcover} (with $\nu=\nu$, $\lambda'=\lambda+\nu \ll \beta$, $X=X_2, V=V_2$, and $C=C_2=D_2$) and \ref{mouse2}, the following holds.

\begin{enumerate}[label = {\bfseries \Alph{propcounter}\arabic{enumi}}]\addtocounter{enumi}{5}
\item For any set $C'\subset C(K_n)$ of at most $\nu n$ colours, there is a perfect $(C'\cup C_2)$-rainbow matching from $X_2$ to $V_2$ which uses every colour in $C'$.\label{mouse5}
\end{enumerate}

We thus have partitions $X=X_1\cup X_2\cup X_3$, $C=C_1\cup C_2\cup C_3$ and $V=V_1\cup V_2\cup V_3$, for which, by \ref{mouse11}, \ref{mouse5} and \ref{mouse4}, we have that
\ref{propp1}--\ref{propp3} hold. Therefore, by the discussion after \ref{propp1}--\ref{propp3}, we have a perfect $C$-rainbow matching from $X$ into $V$.
\end{proof}


\section{The finishing lemma in Case B}\label{sec:finishB}

Our proof of the finishing lemma in Case B has the same structure as the finishing lemma in Case A, except we first construct the colour switchers in the ND-colouring in two steps, before showing this can be done with random vertices and colours. In overview, in this section we do the following, noting the lemmas in which the relevant result is given and the comparable lemmas in Section~\ref{sec:finishA}.
\begin{itemize}
\item Lemma~\ref{lem-switchpath}: We construct colour switchers in the ND-colouring for any pair of colours $(c_1,c_2)$ -- that is, two short rainbow paths with the same length between the same pair of vertices, whose colours are the same except that one uses $c_1$ and one uses $c_2$.
\item Lemma~\ref{lem-absorbpath}: We use Lemma~\ref{lem-switchpath} to construct a similar colour switcher in the ND-colouring that can use 1 of a set of 100 colours.
\item Lemma~\ref{lem-randabsorbpath}: We show that, given a pair of independent random vertex and colour sets, many of these switchers use only vertices and (non-switching) colours in the subsets (cf.\ Lemma~\ref{absorbA}).
\item Lemma~\ref{absorbBmacro}: We use distributive absorption to convert this into a larger scale absorption property (cf.\ Lemma~\ref{absorbAmacro}).
\item Lemma~\ref{colourcoverB}: We embed paths while ensuring that an (arbitrary) small set of colours is used (cf.\ Lemma~\ref{colourcover}).
\item Lemma~\ref{finishingB}: We embed paths using almost all of a set of colours, in such a way that this can reduce the number of `non-random' colours (cf.\ Lemma~\ref{Lemma_saturating_matching_lemma}).
\item Finally, we put this all together to prove Lemma~\ref{lem:finishB}, the finishing lemma in Case B (cf.\ the proof of Lemma~\ref{lem:finishA}).
\end{itemize}

\subsection{Colour switching with paths}\label{sec:switchers}
We start by constructing colour switchers in the ND-colouring capable of switching between 2 colours.
We use switchers consisting of two paths with length 7 between the same two vertices. By using one path or the other we can choose which of two colours $c_1$ and $c_2$ are used, as the other colours used appear on both paths.

\begin{defn}
In a complete graph $K_{2n+1}$ with vertex set $[2n+1]$, let $(x_0,a_1,\ldots,a_\ell)$ denote the path with length $\ell$ with vertices $x_0x_1\ldots x_\ell$ where, for each $i\in [\ell]$, $x_i=x_{i-1}+a_i\pmod 2n+1$.
\end{defn}

\begin{figure}[h]
\begin{center}
{
\begin{tikzpicture}[scale=0.7]
\draw [red,thick] (-1,0) -- (0,0);
\draw (-0.5,-0.3) node {$i$};
\draw [blue,thick] (4,1) -- (0,0);
\draw (2,0.75) node {$c_1$};
\draw [green,thick] (4,1) -- (6,1);
\draw (5,1.25) node {$d_1$};
\draw [magenta,thick] (13,1) -- (6,1);
\draw (9.5,1.25) node {$d_2$};
\draw [purple,thick] (13,1) -- (10,0.5);
\draw (11.5,0.45) node {$d_3$};
\draw [orange,thick] (5,0) -- (10,0.5);
\draw (7.5,0.55) node {$d_4$};
\draw [brown,thick] (5,0) -- (18,0);
\draw (15,-0.3) node {$y-i-c_1-k-x$};
\draw [green,thick] (5,0) -- (7,-0.5);
\draw (6,-0.55) node {$d_1$};
\draw [magenta,thick] (14,-1) -- (7,-0.5);
\draw (10.5,-0.45) node {$d_2$};
\draw [purple,thick] (14,-1) -- (11,-1);
\draw (12.5,-1.25) node {$d_3$};
\draw [orange,thick] (6,-1) -- (11,-1);
\draw (8.5,-1.25) node {$d_4$};
\draw [cyan,thick] (6,-1) -- (0,0);
\draw (3,-0.75) node {$c_2$};

\draw [fill=black] (-1,0) circle [radius = 0.1cm];
\draw [fill=black] (18,0) circle [radius = 0.1cm];
\draw [fill=black] (0,0) circle [radius = 0.1cm];
\draw [fill=black] (4,1) circle [radius = 0.1cm];
\draw [fill=black] (6,1) circle [radius = 0.1cm];
\draw [fill=black] (13,1) circle [radius = 0.1cm];
\draw [fill=black] (10,0.5) circle [radius = 0.1cm];
\draw [fill=black] (5,0) circle [radius = 0.1cm];
\draw [fill=black] (7,-0.5) circle [radius = 0.1cm];
\draw [fill=black] (6,-1) circle [radius = 0.1cm];
\draw [fill=black] (11,-1) circle [radius = 0.1cm];
\draw [fill=black] (14,-1) circle [radius = 0.1cm];

\draw (18.5,0) node {$y$};
\draw (-1.5,0) node {$x$};
\end{tikzpicture}
}
\end{center}
\caption{Two $x,y$-paths used to switch colours between $c_1$ and $c_2$.}\label{colourswitcher}
\end{figure}
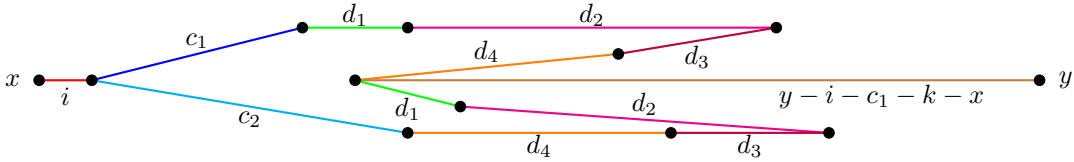

\begin{lemma}[1 in 2 colour switchers]\label{lem-switchpath}
Let $1\gg n^{-1}$ and suppose that $K_{2n+1}$ is ND-coloured. Suppose we have a pair of distinct vertices $x,y\in V(K_{2n+1})$, a pair of distinct colours $c_1,c_2\in C(K_{2n+1})$, and sets $X\subset V(K_{2n+1})$ and $C\subset C(K_{2n+1})$ with $|X|\leq n/25$ and $|C|\leq n/25$.

Then, there is some set $C'\subset C(K_{2n+1})\setminus (C\cup\{c_1,c_2\})$ with size $6$ so that, for each $i\in \{1,2\}$, there is a $(C'\cup\{c_i\})$-rainbow $x,y$-path with length 7 and interior vertices in $V(K_{2n+1})\setminus X$.
\end{lemma}
\begin{proof}
Since $2n+1$ is odd we can relabel $c_1$ and $c_2$ such that $c_1+2k=c_2$ for some $k\in [n]$ (here and later in this proof all the additions are$\pmod{2n+1}$). We will construct a switcher as depicted in Figure~\ref{colourswitcher}.
Find distinct $d_1,d_2,d_3,d_4\in [n]\setminus (C\cup \{c_1,c_2\})$ such that $d_1+d_2=d_3+d_4+k$ and  $(1,c_1,d_1,d_2,-d_3,-d_4)$ and $(1,c_1+k,d_1,d_2,-d_3,-d_4)$ are both valid paths (that is, they have distinct vertices) starting at $1$. This is possible, as follows.

Pick $d_1,d_2,d_3\in [n]\setminus (C\cup \{c_1,c_2\})$ distinctly in turn so that $(1,c_1,d_1,d_2,-d_3)$ and $(1,c_1+k,d_1,d_2,-d_3)$ are both valid paths with the first one avoiding $1+c_1+k$ and the second one avoiding $1+c_2$. Note that, as we choose each $d_i$, we add one more vertex to each of these paths, which have at most 4 vertices already, so at most 12 colours are ruled out by the requirement these paths are valid and avoid the mentioned vertices. Furthermore, there are at most $|C|+4\leq n/20$ colours in $C\cup \{c_1,c_2\}$ or already chosen as some $d_j$, $j<i$. Thus, there are at least $n/2$ choices for each $d_i$, and therefore at least $(n/2)^3$ choices in total.

Now, let $d_4:=d_1+d_2-d_3-k$. Note that there are at most $n^2(|C|+5)\leq n^3/20$ choices for $\{d_1,d_2,d_3\}$ so that $d_4\in C\cup\{c_1,c_2,d_1,d_2,d_3\}$. Therefore, we can choose distinct $d_1,d_2,d_3\in [n]\setminus (C\cup \{c_1\cup c_2\})$ so that $d_4\notin C\cup\{c_1,c_2,d_1,d_2,d_3\}$, in addition to  $(1,c_1,d_1,d_2,-d_3)$ and $(1,c_1+k,d_1,d_2,-d_3)$ both being valid paths which avoid $1+c_1+k$ and $1+c_2$, respectively. Noting that $1+c_1+d_1+d_2-d_3-d_4=1+c_1+k$, $(1,c_1,d_1,d_2,-d_3,-d_4)$ is therefore a valid path. Noting that $1+c_1+k+d_1+d_2-d_3-d_4=1+c_1+2k=1+c_2$, $(1,c_1+k,d_1,d_2,-d_3,-d_4)$  is therefore a valid path, with endvertex $1+c_2$. Therefore, its reverse,  $(1+c_2,d_4,d_3,-d_2,-d_1,-c_1-k)$, is a valid path that ends with $1$. Moving the vertex $1$ to the start of the path, we get the valid path $(1,c_2,d_4,d_3,-d_2,-d_1)$.

Let $I$ be the set of $i\in [n]\setminus (C\cup\{c_1,c_2,d_1,d_2,d_3,d_4\})$ such that $x$ and $y$ are not on $(x+i,c_1,d_1,d_2,-d_3,-d_4)$, or $(x+i,c_2,d_4,d_3,-d_2,-d_1)$, and that $y-i-c_1-k-x\notin C\cup\{c_1,c_2,d_1,d_2,d_3,d_4\}$. Note that these conditions rule out at most 12, 12, and $|C|+6$ values of $i$ respectively, so that $|I|\geq n-2|C|-12-24\geq n/2$.

For each $i\in I$, add $x$ and $y$ as the start and end of
 $(x+i,c_1,d_1,d_2,-d_3,-d_4)$ respectively, giving, as $k=d_1+d_2-d_3-d_4$ the path
 \[
 P_i:=(x,i,c_1,d_1,d_2,-d_3,-d_4,y-i-c_1-k-x).
 \]
Then, add $x$ and $y$ as the start and end of
 $(x+i,c_2,d_4,d_3,-d_2,-d_1)$ respectively, giving, as $d_4+d_3-d_1-d_2=-k=c_1-c_2+k$ the path
 \[
 Q_i:=(x,i,c_2,d_4,d_3,-d_2,-d_1,y-i-c_1-k-x).
 \]
Let $C_i=\{i,d_1,d_2,d_3,d_4,y-i-c_1-k\}$. Note that, $P_i$ and $Q_i$ are both rainbow $x,y$-paths with length 7, with colour sets $(C_i\cup\{c_1\})$ and $(C_i\cup\{c_1\})$ respectively (indeed, when we add $-d_i, 1 \leq i \leq 4$ to get the next vertex of the path the colour of this edge is $d_i$ and ``$-$" just indicates in which direction we are moving). 

Each vertex in $X$ can appear as the interior vertex of at most 6 different paths $P_i$ and 6 different paths $Q_i$. As $|I|\geq n/2 > 12|X|$, there must be some $j\in I$ for which the interior vertices of  $P_j$ and $Q_j$ avoid $X$. Then, $C'=C_j$ is a colour set as required by the lemma, as demonstrated by the paths $P_j$ and $Q_j$.
\end{proof}

The following lemma uses this to find colour switchers for an arbitary set of 100 colours $\{c_1,\ldots,c_{100}\}$ in the ND-colouring between an arbitrary vertex pair $\{x,y\}$. A sketch of its proof is as follows. First, we select a vertex $x_1$ and colours $d_1,\ldots,d_{100}$ so that $c_i+d_i=x_1-x$ for each $i\in [100]$. By choosing the vertex between $x$ and $x_1$ appropriately, we can find a $\{c_i,d_i\}$-rainbow $x,x_1$-path with length 2 for each $i\in [100]$. This allows us to use any pairs of colours $\{c_i,d_i\}$, so we need only construct a path which can switch between using any set of 99 colours from $\{d_1,\ldots,d_{100}\}$. This we do by constructing a sequence of $(d_i,d_{i+1})$-switchers for each $i\in [99]$ and putting them together between $x_1$ and $y$.

\begin{lemma}[1 in 100 colour absorbers]\label{lem-absorbpath}
Let $1\gg n^{-1}$. Suppose $K_{2n+1}$ is ND-coloured.
Suppose we have a pair of distinct vertices $x,y\in V(K_{2n+1})$, a set $C\subset C(K_{2n+1})$ of $100$ colours, and sets $X\subset V(K_{2n+1})$ and $C'\subset C(K_{2n+1})$ with $|X|\leq n/10^3$ and $|C'|\leq n/10^3$.

Then, there is a set $\bar{C}\subset C(K_{2n+1})\setminus (C\cup C')$ of 694 colours and a set $\bar{X}\subset V(K_{2n+1})\setminus X$ of at most $1500$ vertices so that, for each $c\in C$,
there is a $(\bar{C}\cup \{c\})$-rainbow $x,y$-path with length 695 and internal vertices in $\bar{X}$.
\end{lemma}
\begin{proof}
Let $\ell=100$, $x_0=x$, and $x_\ell=y$, and label $C=\{c_1,\ldots,c_\ell\}$. Pick $x_1\in [2n+1]\setminus (X\cup\{x_0,x_\ell\})$ so that $x_1-x_0+c_i\in [n]\setminus (C\cup C')$ and $x_1-c_i\in [2n+1]\setminus X$
for each $i\in [\ell]$. Each such condition forbids at most $n/10^3+100$ points and we have at most $2\ell=200$ conditions so we can indeed find $x_1$ satisfying all of them.
For each $i\in [\ell]$, let $d_i=x_1-x_0+c_i$ and $y_i= x_1-c_i=x_0+d_i$. Then,
\begin{itemize}
\item $c_1,\ldots,c_\ell,d_1,\ldots,d_\ell$ are distinct colours in $C\cup ([n]\setminus C')$,
\item $\{x_0,x_1,x_\ell,y_1,\ldots,y_\ell\}$ are distinct vertices in $\{x,y\}\cup ([2n+1]\setminus X)$, and
\item for each $i\in [\ell]$, $x_0y_ix_1$ is a $\{c_i,d_i\}$-rainbow path.
\end{itemize}
Let $C''=\{d_1,\ldots,d_\ell\}$.
Pick distinct vertices $x_2,\ldots,x_{\ell-1}\in [2n+1]\setminus (X\cup\{x_0,x_1,x_\ell,y_1,\ldots,y_\ell\})$ and let $X'=\{x_0,x_1,\ldots,x_\ell,y_1,\ldots,y_\ell\}$.

Next, iteratively, for each $1\leq i\leq \ell-1$, using Lemma~\ref{lem-switchpath} find a set $X_i$ of at most 12 vertices in $[2n+1]\setminus (X\cup X'\cup(\cup_{j<i}X_j))$ and a set $C_i$ of 6 colours in $[n]\setminus (C\cup C'\cup C''\cup(\cup_{j<i}C_j))$ so that there is a $(C_i\cup\{d_{i}\})$-rainbow $x_i,x_{i+1}$-path with length 7 and internal vertices in $X_i$ and a $(C_i\cup\{d_{i+1}\})$-rainbow
 $x_i,x_{i+1}$-path with length 7 and internal vertices in $X_i$.

Let $\bar{C}=C''\cup(\cup_{i\in [\ell-1]}C_i)$ and $\bar{X}=X'\cup (\cup_{i\in [\ell-1]}X_i)$, and note that $|\bar{C}|=\ell+6(\ell-1)=694$ and 
$|\bar{X}|\leq 2\ell+1+12(\ell-1) \leq 1500$. We will show that $\bar{C}$ and $\bar{X}$ satisfy the condition in the lemma.

Let then $j\in [\ell]$. For each $1\leq i< j$, let $P_i$ be a $(C_i\cup\{d_{i}\})$-rainbow $x_i,x_{i+1}$-path with length 7 and internal vertices in $X_i$. For each $j\leq i\leq \ell-1$, let $P_i$ be a $(C_i\cup\{d_{i+1}\})$-rainbow $x_i,x_{i+1}$-path with length 7 and internal vertices in $X_i$. Thus, the paths $P_i$, $i\in [\ell-1]$, cover all the colours in $C''$ except for $d_j$, as well as the colours in each set $C_i$.
Therefore, as $x_0=x$ and $x_\ell=y$,
\[
x_0y_jx_1P_1x_2P_2x_3P_3\ldots P_{\ell-1}x_{\ell}
\]
is a $(\bar{C}\cup \{c_j\})$-rainbow $x,y$-path with length $2+7\times 99=695$ whose interior vertices are all in $\bar{X}$, as required.
\end{proof}

The following corollary of this will be convenient to apply.
\begin{corollary}\label{cor-absorbpath}
Let $1\gg n^{-1}$ and suppose that $K_{2n+1}$ is ND-coloured. For each pair of distinct vertices $x,y\in V(K_{2n+1})$ and each set $C\subset C(K_{2n+1})$ of 100 colours, there are $\ell=n/10^7$ disjoint vertex sets $X_1,\ldots,X_\ell\subset V(K_{2n+1})\setminus \{x,y\}$ with size at most 1500 and disjoint colour sets $C_1,\ldots,C_\ell\subset [n]\setminus C$ with size 694
such that the following holds.

For each $i\in [\ell]$ and $c\in C$, there is a $C_i\cup\{c\}$-rainbow $x,y$-path with length 695 and interior vertices in $X_i$.
\end{corollary}
\begin{proof}
Iteratively, for each $i=1, \dots, \ell$, choose sets $X_i$ and $C_i$ using Lemma~\ref{lem-absorbpath} (at the $i$th iteration letting $X=X_1\cup \dots \cup X_{i-1}$, $C'=C_1\cup \dots \cup C_{i-1}$).
\end{proof}

The following lemma finds $1$-in-$100$ colour switchers in a random set of colours and vertices.
\begin{lemma}[Colour switchers using random vertices and colours]\label{lem-randabsorbpath}
Let $p,q\gg \mu \gg n^{-1}$ and suppose that $K_{2n+1}$ is ND-coloured.  Let $X\subset V_{2k+1}$ be $p$-random and  $C\subset C(K_{2n+1})$ $q$-random, and such that $X$ and $C$ are independent. With high probability the following holds.

For every distinct $x,y\in V(K_{2n+1})$, $C'\subset C(K_{2n+1})$ with $|C'|=100$, and $X'\subset X$, $C''\subset C$ with $|X'|,|C''|\leq\mu n$,  there is a set $\bar{C}\subset C\setminus (C'\cup C'')$ of 694 colours and a set $X''\subset X\setminus X'$ of at most $1500$ vertices with the following property. For each $c\in C'$,
there is a $(\bar{C}\cup \{c\})$-rainbow $x,y$-path with length 695 and internal vertices in $X''$.
\end{lemma}
\begin{proof}
We will show that for any distinct pair $x,y\in V(K_{2n+1})$ of vertices and set $C'\subset C(K_{2n+1})$ of 100 colours the property holds with probability $1-o(n^{-102})$ so that the lemma holds by a union bound.

Fix then distinct $x,y\in V(K_{2n+1})$ and a set $C'\subset C(K_{2n+1})$ with size 100. Fix $\ell= n/10^7$ and use Corollary~\ref{cor-absorbpath} to find disjoint vertex sets $X_i\subset V(K_{2n+1})$, $i\in [\ell]$, with size at most 1500 and disjoint colours sets $C_i\subset C(K_{2n+1})$, $i\in [\ell]$, with size $694$ so that for each $c\in C'$ and $i\in [\ell]$ there is a $C_i\cup \{c\}$-rainbow $x,y$-path with length 695 and internal vertices in $X_i$. 

Let $I\subset [\ell]$ be the set of $i\in [\ell]$ for which $X_i\subset X$ and $C_i\subset C$. Note that $|I|$ is $1$-Lipschitz, and, for each $i\in [\ell]$, we have $\P(X_i\subseteq X, C_i\subseteq C)\geq p^{1500}q^{695}\gg \mu$.
By Azuma's inequality, with probability $1-o(n^{-102})$ we have $|I|\geq 10^4\mu n$.

Take then any $X'\subset X$ and $C''\subset C$ with size at most $\mu n$ each. There must be some $j\in I$ for which $X'\cap X_j=\emptyset$ and $C''\cap C_j=\emptyset$. Let $\bar{C}=C_j\subset C\setminus (C'\cup C'')$ and $X''=X_j\subset X\setminus X'$. Then, as required, for each $c\in C'$ there is a $(\bar{C}\cup\{c\})$-rainbow $x,y$-path with interior vertices in $X''$.
\end{proof}

\subsection{Distributive absorption with paths}

We now use Lemma~\ref{lem-randabsorbpath} and distributive absorption to get a larger scale absorption property, as follows.

\begin{lemma}\label{absorbBmacro} Let $1/n\ll \eta \ll \mu\ll \eps$. Let $K_{2n+1}$ be 2-factorized. Suppose that $V_0$ is an $\eps$-random subset of $V(K_{2n+1})$ and $C_0$ is an $\eps$-random subset of $C(K_{2n+1})$, which is independent of $V_0$. Suppose $\ell=\mu n/695$ and that $\ell/3\in \mathbb{N}$. With high probability, the following holds.

Suppose that $\{x_1,\ldots,x_{\ell},y_1,\ldots,y_\ell\}\subset V(K_{2n+1})$, $\alpha\leq \eta$ and $C\subset C(K_{2n+1})\setminus C_0$ is a set of at most $\eta n$ colours. Then, there is a set $\hat{C}\subset C_0$ of $695\ell-\alpha n$ colours such that, for every set $C'\subset C$ of $\alpha n$ colours, there is a set of vertex disjoint $x_i,y_i$-paths which are collectively $(\hat{C}\cup C')$-rainbow, have length 695, and internal vertices in $V_0$.
\end{lemma}
\begin{proof} By Lemma~\ref{lem-randabsorbpath} (applied to $X=V_0, C=C_0$ with $p=q=\eps$ and $\mu'=3\mu$), we have the following property with high probability.
\stepcounter{propcounter}
\stepcounter{propcounter}
\begin{enumerate}[label = {\bfseries \Alph{propcounter}\arabic{enumi}}]
\item For each distinct $x,y\in V(K_{2n+1})$, and every $C'\subset C(K_{2n+1})$ with $|C'|=100$, and $X'\subset V_0$, $C''\subset C_0$ with $|X'|,|C''|\leq 3\mu n$,  there is a set $\bar{C}\subset C_0\setminus (C'\cup C'')$ of 694 colours and a set $X''\subset V_0\setminus X'$ of at most $1500$ vertices with the following property.

For each $c\in C'$,
there is a $(\bar{C}\cup \{c\})$-rainbow $x,y$-path with length 695 and internal vertices in $X''$.\label{bbb2}
\end{enumerate}

We will show that the property in the lemma holds. Suppose then that  $\{x_1,\ldots,x_{\ell},y_1,\ldots,y_\ell\}$, $\alpha\leq \eta$ and $C\subset C(K_{2n+1})$ is a set of at most $\eta n$ colors.
Let $h=\ell/3$. By Lemma~\ref{Lemma_Chernoff}, with high probability 
$|C_0|\geq \epsilon n/2> \mu n$. Therefore we can pick a set $\hat C_0 \subset C_0$ of $(3h-\alpha n)$ colours. Noting that $|\hat C_0|\geq 2h$ and $|\hat C_0\cup C|\leq 4h$.
By the same reasoning as just before \ref{Hmatch}, by Lemma~\ref{Lemma_H_graph} there is a bipartite graph, $H$ with maximum degree 100 say, with  vertex classes $[3h]$ and $\hat C_0\cup C$ such that the following holds.

\begin{enumerate}[label = {\bfseries \Alph{propcounter}\arabic{enumi}}]\addtocounter{enumi}{1}
\item For each set $C'\subset C$ of $\alpha n$ colours, there is a perfect matching between $[3h]$ and $\hat C_0\cup C'$ in $H$.\label{Hmatch2}
\end{enumerate}

Iteratively, for each $1\leq i\leq 3h$, let $D_i=N_H(i)$ and, using \ref{bbb2}, find sets $C_i\subset C_0\setminus (C_1\cup\ldots\cup C_{i-1})$ and $V_i\subset V_0\setminus (V_1\cup \ldots\cup V_{i-1})$,
with sizes $694$ and at most $1500$, such that the following holds.

\begin{enumerate}[label = {\bfseries \Alph{propcounter}\arabic{enumi}}]\addtocounter{enumi}{2}
\item For each $c\in D_i$, there is a $(C_i\cup\{c\})$-rainbow $x_i,y_i$-path with length 695 with internal vertices in $V_i$.\label{propabsorb2}
\end{enumerate}

Note that in the $i$th application of \ref{bbb2}, we have $X'=V_1\cup \ldots\cup V_{i-1}$ and $C''=C_1\cup\ldots \cup C_{i-1}$, so that $|X'|\leq 1500\ell\leq 3\mu n$ and $|C'|\leq 694\ell\leq 3\mu n$, as required.

Let $\hat{C}=\hat C_0\cup C_1\cup \ldots\cup C_{3h}$. We claim $\hat{C}$ has the property required. Indeed, suppose $C'\subset C$ is a set of $\alpha n$
colours. Using~\ref{Hmatch2}, let $M$ be a perfect matching between $[3h]$ and $\hat{C}_0\cup C'$ in $H$, and label $\hat C_0\cup C'=\{c_1,\ldots,c_{3h}\}$ so that, for each $i\in [3h]$, $c_i$ is matched to $i$ in $M$.

By \ref{propabsorb2}, for each $i\in [3h]$, there is an $x_i,y_i$-path, $P_i$ say, with length 695 which is $(C_i\cup\{c\})$-rainbow with internal vertices in $V_i$. Then, $P_1,\ldots, P_{3h}$ are the paths required.
\end{proof}

The following corollary of Lemma~\ref{absorbBmacro} will be convenient to apply.

\begin{corollary}\label{cor-absorbBmacro} Let $1/n\ll \eta \ll \mu\ll \eps$ and $1/k\ll 1$. Let $K_{2n+1}$ be 2-factorized. Suppose that $V_0$ is an $\eps$-random subset of $V(K_{2n+1})$ and $C_0$ is an $\eps$-random subset of $C(K_{2n+1})$, which is independent of $V_0$. Suppose that $\ell=\mu n/k$ and  $695|k$. With high probability, the following holds.

Suppose that $\{x_1,\ldots,x_{\ell},y_1,\ldots,y_\ell\}\subset V(K_{2n+1})$, $\alpha\leq \eta$ and $C\subset C(K_{2n+1})$ is a set of at most $\eta n$ colours. Then, there is a set $\hat{C}\subseteq C_0$ of $k\ell-\alpha n$ colours such that, for every set $C'\subset C$ of $\alpha n$ colours, there is a set of vertex disjoint $x_i,y_i$-paths which are collectively $(\hat{C}\cup C')$-rainbow, have length k, and internal vertices in $V_0$.
\end{corollary}
\begin{proof} Let $V_1,V_2\subset V_0$ be disjoint and $(\eps/2)$-random and let $C_0'\subset C_0$ be $(\eps/2)$-random. Let $\ell'=\mu n/695=\ell \cdot k/695$. We can assume that $3|\ell'$, as discussed in Section~\ref{sec:not}.
By Lemma~\ref{absorbBmacro} (applied with $\eps'=\eps/2$, $\mu=\mu$, $\eta=\eta$, $C_0'=C_0'$ and $V_0'=V_1$) and Lemma~\ref{Lemma_Chernoff}, with high probability we have the following properties.

\stepcounter{propcounter}
\begin{enumerate}[label = {\bfseries \Alph{propcounter}\arabic{enumi}}]
\item
Let $\{x_1,\ldots,x_{\ell'},y_1,\ldots,y_{\ell'}\}\subset V(K_{2n+1})$, $\alpha'\leq \eta$ and let $C'\subset C(K_{2n+1})$ be a set of at most $\eta n$ vertices. Then, there is a set $\hat{C}$ of $695\ell'-\alpha' n$ colours such that, for every set $C''\subset C$ of $\alpha' n$ colours, there is a set of vertex disjoint $x_i,y_i$-paths which are collectively $(\hat{C}\cup C'')$-rainbow, have length 695, and internal vertices in $V_1$.\label{argh1}
\item $|V_2|\geq \mu n/4$.\label{argh2}
\end{enumerate}
We will show that the property in the lemma holds. Suppose therefore that $\{x_1,\ldots,x_{\ell},y_1,\ldots,y_\ell\}\subset V(K_{2n+1})$, $\alpha\leq \eta$ and $C\subset C(K_{2n+1})$ is a set of at most $\eta n$ colours. Let $k'=k/695\in \mathbb{N}$ and, let $\{z_{i,j}:i\in [\ell],j\in [k']\}$ be a set of vertices in $V_2$, using \ref{argh2}. Apply \ref{argh1} to the pairs $(x_i,z_{i,1})$, $(z_{i,j},z_{i,j+1})$ and $(z_{i,k'},y_i)$, $i\in [\ell]$, $1\leq j<k'$, and take their union, to get the required paths.
\end{proof}

\subsection{Covering small colour sets with paths}
We now show that paths can be found using every colour in an arbitrary small set of colours.
\begin{lemma}\label{colourcoverB} Let $1/n\ll \xi \ll\beta$, let $1/n\ll 1/k\ll 1$ with $k= 3 \mod 4$.
Let $K_{2n+1}$ be 2-factorized. Suppose that $V_0\subset V(K_{2n+1})$ and $C_0\subset C(K_n)$ are $\beta$-random and independent of each other. With high probability, the following holds with $m=5\xi n/k$.

For any set $C\subset C(K_{n+1})\setminus C_0$ of at most $\xi n$ colours, and any set $\{x_1,\ldots,x_m,y_1,\ldots,y_m\}\subset V(K_{2n+1})\setminus V_0$, there is a set of vertex-disjoint paths $x_i,y_i$-paths, $i\in [m]$, each with length $k$ and interior vertices in $V_0$, which are collectively $(C\cup C_0)$-rainbow and use all the colours in $C$.
\end{lemma}
\begin{proof} Let $V_1,V_2\subset V_0$ be disjoint and $(\beta/2)$-random. Let $C_1,C_2\subset C_0$ be disjoint and $(\beta/2)$-random.
By Lemma~\ref{Lemma_Chernoff}, and Lemma~\ref{Lemma_number_of_colours_inside_random_set}, with high probability, the following properties hold.
\stepcounter{propcounter}
\begin{enumerate}[label = {\bfseries \Alph{propcounter}\arabic{enumi}}]
\item $|C_1|\geq \beta n/3$.\label{beep1}
\item $V_1$ is $(\beta^2 n/4)$-replete. \label{beep1prime}
\end{enumerate}
By Lemma~\ref{Lemma_few_connecting_paths} applied with $p=\beta /2$, $q=4\xi$, $V=V_2$ and $C=C_2$, with high probability, we have the following property.
\begin{enumerate}[label = {\bfseries \Alph{propcounter}\arabic{enumi}}]\addtocounter{enumi}{2}
 \item For any $m'\leq 4\xi n$ and set $\{x_1, y_1, \dots, x_{m'}, y_{m'}\}\subset V(K_{2n+1})$
 there is a collection $P_1, \dots, P_{m'}$ of vertex-disjoint paths with length $3$, with internal vertices in $V_2$, where $P_i$ is an $x_i,y_i$-path, for each $i\in [m']$, and $P_1\cup \dots\cup P_{m'}$ is $C_2$-rainbow.\label{beep2}
\end{enumerate}

We will now show that the property in the lemma holds. Let then $C\subset C(K_{n+1})\setminus C_0$ have at most $\xi n$ colours and let $X:=\{x_1,\ldots,x_m,y_1,\ldots,y_m\}\subset V(K_{2n+1})\setminus V_0$. Take $\ell=(k-3)/4$, and note that this is an integer and $\ell m\geq \xi n$. Using \ref{beep1}, take an order $\ell m$ set $C'\subset C\cup C_1$ with $C\subset C'$  and label $C'=\{c_{i,j}:i\in [m],j\in [\ell]\}$.

Using \ref{beep1prime}, greedily find independent edges $s_{i,j}t_{i,j}$, $i\in [m],j\in [\ell]$, with vertices in $V_1$, so that each edge  $s_{i,j}t_{i,j}$ has colour $c_{i,j}$. Note that, when the edge $s_{i,j}t_{i,j}$ is chosen at most $2m\ell\leq 2mk\leq 10\xi n$ vertices in $V_1$ are in already chosen edges. 

Using \ref{beep2}  with $m'=m(\ell+1)$, find vertex-disjoint paths $P_{i,j}$, $i\in [m]$, $0\leq j\leq \ell$, with length 3 and internal vertices in $V_2$ so that these paths are collectively $C_2$-rainbow and the following holds for each $i\in [m]$.
\begin{itemize}
\item $P_{i,0}$ is a $x_{i}s_{i,1}$-path.
\item For each $1\leq j<\ell$, $P_{i,j}$ is a $t_{i,j},s_{i,j+1}$-path.
\item $P_{i,\ell}$ is a $t_{i,\ell},y_i$-path.
\end{itemize}
Then, the paths $P_i=\cup_{0\leq j\leq \ell}P_{i,j}$, $i\in [m]$, use each colour in $C'\subset C\cup C_1$, and hence $C$, and otherwise use colours in $C_2$, have interior vertices in $V_0$ and, for each $i\in [m]$, $P_i$ is a length $3(\ell+1)+\ell=k$ path from $x_i$ to $y_i$. That is, the paths $P_i$, $i\in[m]$, satisfy the condition in the lemma. \end{proof}

\subsection{Almost-covering colour sets with paths}
We now show that paths can be found using almost every colour in a set of mostly-random colours.
\begin{lemma}\label{finishingB}
Let $1\geq p\gg q\gg \gamma,\eta\gg  1/k \gg 1/n$ with $k=1\mod 3$,  and let $m\leq 1.01q n/k$. 
Let  $K_{2n+1}$ be $2$-factorized.
Suppose we have disjoint sets $V, V_0\subseteq V(K_{2n+1})$, and $D,D_0\subseteq C(K_{2n+1})$  with $V$ $p$-random, $D$ $q$-random, and $V_0, D_0$ $\gamma$-random. Suppose further that $V_0$ and $D_0$ are independent of each other. Then, the following holds with high probability.

For any set $C\subset C(K_{n+1})$ with $D\cup D_0\subset C$ of $mk+\eta n$ colours, and any set $X:=\{x_1,\ldots,x_m,y_1,\ldots,y_m\}\subset V(K_{2n+1})\setminus (V\cup V_0)$, there is a set of vertex-disjoint paths $x_i,y_i$-paths with length $k$, $i\in [m]$, which have interior vertices in $V\cup V_0$ and which are collectively $C$-rainbow.
\end{lemma}
\begin{proof} Let $\ell=(k-4)/3$ and note that this is an integer. Let $m'=m+\eta n/6k$. Using that $p\gg q,\eta$, take in $V$ vertex disjoint sets $V',V_1,\ldots,V_{\ell}$, such that $V'$ is $(p/2)$-random and, for each $i\in [\ell]$, $V_i$ is $(m'/2n)$-random. Take an $(\eta\gamma)$-random subset $D'_0\subset D_0$. Take in $D$ vertex disjoint sets $C_1,\ldots,C_{2\ell}$ so that, and, for each $1\leq i\leq 2\ell$, $C_i$ is $(m'/n)$-random.
 Note that this later division is possible as $2\ell\cdot m'/n\leq 2\ell \cdot 1.02q n/k\leq q$.

By Lemma~\ref{Lemma_number_of_colours_inside_random_set}, with high probability the following holds.
\addtocounter{propcounter}{1}
\begin{enumerate}[label = {\bfseries \Alph{propcounter}\arabic{enumi}}]
\item $V'$ is $(p^2 n/4)$-replete.\label{mice1}
\end{enumerate}
By Lemma~\ref{Lemma_few_connecting_paths}, applied with $p'=\gamma\eta$ and $q'=2q/k$ to $D'_0$ and a $(\gamma\eta)$-random subset of $V_0$, with high probability the following holds.
\begin{enumerate}[label = {\bfseries \Alph{propcounter}\arabic{enumi}}]\addtocounter{enumi}{1}
\item For any set $\{x_1, y_1, \dots, x_{2m}, y_{2m}\}\subset V(K_{2n+1})$
 there is a collection $P_1, \dots, P_{2m}$ of vertex-disjoint paths with length $3$, having internal vertices in $V_0$, where $P_i$ is an $x_i,y_i$-path, for each $i\in [2m]$, and $P_1\cup \dots\cup P_{2m}$ is $D'_0$-rainbow.\label{mice1a}
\end{enumerate}
By Lemma~\ref{Lemma_Chernoff}, with high probability the following holds.
\begin{enumerate}[label = {\bfseries \Alph{propcounter}\arabic{enumi}}]\addtocounter{enumi}{3}
\item $|D_0'\cup C_1\cup \ldots\cup C_{2\ell}|\leq 2m'\ell +\eta n/2$. \label{mice1b}
\item For each $j\in [\ell]$, $|V_j|\leq m'+\eta^2 n/k^2$.\label{mice1bb}
\end{enumerate}
By Lemma~\ref{Lemma_nearly_perfect_matching}, applied for each $i\in [2\ell]$ and $j\in [\ell]$ with $p'=m'/n$ and $\beta=\eta^2/k^2$, with high probability the following holds.
\begin{enumerate}[label = {\bfseries \Alph{propcounter}\arabic{enumi}}]\addtocounter{enumi}{5}
\item For each $i\in [2\ell]$ and $j\in [\ell]$ and any vertex set $Y\subset V(K_{2n+1})\setminus V$ with $|Y|\leq m'$ there is a $C_i$-rainbow matching from $Y$ into $V_j$ with at least $|Y|-\eta^2 n/k^2$ edges.\label{mice2}
\end{enumerate}
We will show that the property in the lemma holds. Set then $C\subset C(K_{n+1})$ with $D\cup D_0\subset C$ so that $|C|=mk+\eta n$ and let $X:=\{x_1,\ldots,x_m,y_1,\ldots,y_m\}\subset V(K_{2n+1})\setminus (V_0\cup V)$.

Note that, using \ref{mice1b}, 
\begin{align*}
|C\setminus (D_0'\cup C_1\cup \ldots\cup C_{2\ell})|&\geq mk+\eta n-2m'\ell-\eta n/2=
(3\ell+4)m+\eta n/2-2m'\ell\geq 3\ell (m-m')+\eta n/2+m'\ell\\
&= -3\ell \eta n/6k+\eta n/2+m'\ell 
\geq m'\ell,
\end{align*}
and take disjoint sets $C'_1,\ldots,C_\ell'\subset C\setminus (D_0'\cup C_1\cup \ldots\cup C_{2\ell})$  with size $m'$.
Let $C'=C'_1\cup\ldots \cup C'_\ell$. Greedily, using~\ref{mice1}, take vertex disjoint edges $x_cy_c$, $c\in C'$, with vertices in $V'$ so that the edge $x_cy_c$ has colour $c$. Note that these edges have $2m'\ell \leq qn$ vertices in total, so that this greedy selection is possible.
Let $M$ be the matching $\{x_cy_c:c\in C'\}$.
For each $i\in [\ell]$, let $Z_i=\{x_c:c\in C'_i\}$ and $Y_i=\{y_c:c\in C'_i\}$.

For each $i\in [\ell-1]$, use~\ref{mice2} to find a $C_{i}$-rainbow matching $M_i$  with $m'-\eta^2 n/k^2$ edges from $Z_i$ into $V_i$, and a $C_{i+\ell}$-rainbow matching, $M'_i$ say, with $m'-\eta^2 n/k^2$ edges from $Y_i$ into $V_{i-1}$. Note that, by \ref{mice1bb} these matchings overlap in at least $m'-4\eta^2 n/k^2$ vertices.
Therefore, putting together $M$ with the matchings $M_i$, $M_i'$, $i\in [\ell]$ gives at least $m'-\ell \cdot 4\eta^2 n/k^2\geq m$ vertex disjoint paths with length $3\ell-2$. Furthermore, these paths are collectively rainbow with colours in $(C\setminus D'_0)$. Take $m$ such paths, $Q_i$, $i\in [m]$. Apply \ref{mice1a}, to connect one endpoint of $Q_i$ to $x_i$ and another endpoint to $y_i$ using two paths of length 3 and new vertices and colours in $V_0$ and $D'_0$ respectively to get the paths with length $k=3\ell+4$ as required.
\end{proof}

\subsection{Proof of the finishing lemma in Case B}\label{subsec:finishB}
We can now prove Lemma~\ref{lem:finishB}.

\begin{proof}[Proof of Lemma~\ref{lem:finishB}]
Pick $\xi,\beta,\lambda,\alpha$ so that $1/k\ll \xi\ll\beta\ll\lambda\ll \alpha\ll \mu\ll \eta\ll \eps\ll p\leq 1$. Let $V_1',V_2',V_3'\subset V$ be disjoint sets which are each $(p/3)$-random. Let $W_1,W_2,W_3\subset V_0$ be disjoint sets which are each $(\mu/3)$-random.
Let $D_1,D_2,D_3\subset C_0$ be disjoint and $(\mu/3)$-, $\beta$- and $\alpha$-random respectively. 

Set $m_1=(\eps-5\xi-\lambda)n/k$, $m_2=5\xi n/k$ and $m_3=\lambda n/k$. By Lemma~\ref{Lemma_Chernoff}, with high probability we have that $|D_2|\leq 2\beta n$.
By Lemma~\ref{finishingB}, applied with $p'=p/3$, $q=(1-\eta)\eps$, $\gamma=\mu/3$, $\eta'=\xi$, $n=n$, $k=k$, $m=m_1$, $V'=V_1'$, $V_0=W_1$, $D=C$, $D_0=D_1$, with high probability we have the following.
\stepcounter{propcounter}
\begin{enumerate}[label = {\bfseries \Alph{propcounter}\arabic{enumi}}]
\item For any set $\bar{C}\subset C(K_{n+1})$ with $C\cup D_1\subset \bar{C}$ of $m_1k+\xi n=(\eps-4\xi-\lambda)n$ colours, and any collection of vertices $\{x_1,\ldots,x_{m_1},y_1,\ldots,y_{m_1}\}\subset V(K_{2n+1})\setminus (V_1'\cup W_1)$, there is a set of vertex-disjoint $x_i,y_i$-paths, $i\in [m_1]$, each with length $k$, which have interior vertices in $V_1'\cup W_1$ and which are collectively $\bar{C}$-rainbow.
\label{mouseb11}
\end{enumerate}
 By Lemma~\ref{colourcoverB}, applied with $m=m_2$, $V_0\subset W_2$ a $\beta$-random subset and $C_0=D_2$, with high probability, we have the following.
\begin{enumerate}[label = {\bfseries \Alph{propcounter}\arabic{enumi}}]\addtocounter{enumi}{1}
\item For any set $\bar{C}\subset C(K_{n+1})\setminus D_2$ of at most $\xi n$ colours, and any set $\{x_1,\ldots,x_{m_2},y_1,\ldots,y_{m_2}\}\subset V(K_{2n+1})\setminus W_2$, there is a set of vertex-disjoint $x_i,y_i$-paths, $i\in [{m_2}]$, each with length $k$ and interior vertices in $W_2$, which are collectively $(\bar{C}\cup D_2)$-rainbow and use all the colours in $\bar{C}$.\label{mouseb2}
\end{enumerate}
By Corollary~\ref{cor-absorbBmacro}, applied with $\eta'=2\beta$, $\mu'=\lambda$, $\eps'=\alpha$, $n=n$, $k=k$, $\ell=m_3$, $V_0\subset W_3$ an $\alpha$-random subset and $C_0=D_3$, with high probability we have the following.
\begin{enumerate}[label = {\bfseries \Alph{propcounter}\arabic{enumi}}]\addtocounter{enumi}{2}
\item For any $\{x_1,\ldots,x_{m_3},y_1,\ldots,y_{m_3}\}$, $\bar{\beta}\leq 2\beta$ and $\bar{C}\subset C(K_{2n+1})$ with $|\bar{C}|\leq 2\beta n$, there is a set $D_3'\subset D_3$ of $m_3k-\bar{\beta}n$ colours
such that, for every set $C'\subset \bar{C}$ of $\bar{\beta} n$ colours, there is a set of vertex-disjoint $x_i,y_i$-paths, $i\in [m_3]$, which are collectively $(D_3'\cup C')$-rainbow, have length $k$, and internal vertices in $W_3$.\label{mouseb3}
\end{enumerate}
Let $m=\eps n/k$, so that $m=m_1+m_2+m_3$. We will show that the property in the lemma holds.

Suppose then that $x_1,\ldots,x_m,y_1,\ldots,y_m$ are distinct vertices in $V(K_{2n+1})\setminus (V\cup V_0)$ and $D\subset C(K_{2n+1})$ so that $|D|=\eps n$ and $C\cup C_0\subset D$. Let $V_1=V_1'\cup W_1$, $V_2=V_2'\cup W_2$ and $V_3=V_3'\cup W_3$.
Let $[m]=I_1\cup I_2\cup I_3$ be a partition with $|I_i|=m_i$ for each $i\in [3]$.
By \ref{mouseb3} (applied with $\bar C=D_2$ and $\bar{\beta}=(|D_2|/n)-4\xi$), there is a set $D_3'\subset D_3$ of $m_3k-|D_2|+4\xi n$ colours with the following property.
\begin{enumerate}[label = {\bfseries \Alph{propcounter}\arabic{enumi}}]\addtocounter{enumi}{3}
\item For every set $C'\subset D_2$ of $|D_2|-4\xi n$ colours, there is a set of vertex disjoint $x_i,y_i$-paths, $i\in I_3$, which are collectively $(D_3'\cup C')$-rainbow, have length $k$, and internal vertices in $W_3$.\label{mouseb4}
\end{enumerate}

Let $C_1=D\setminus (D_2\cup D_3')$, $C_2=D_2$ and $C_3=D_3'$. We now reason analogously to the discussion after \ref{propp1}--\ref{propp3}.
Note that $C\cup D_1\subset C_1$ and 
\begin{equation}
    |C_1|=|D|-|D_2|-(m_3k-|D_2|+4\xi n)=mk-m_3k-4\xi n=(m_1+m_2)k-4\xi n=m_1k+\xi n.
\label{coldrain}
\end{equation}
Therefore, using \ref{mouseb11}, we can find $m_1$ paths $\{P_1, \dots, P_{m_1}\}$ such that $P_i$ is an $x_i,y_i$-path with length $k$, for each $i\in I_1$, so that these paths are vertex disjoint with internal vertices in $V_1=V_1'\cup W_1$ and are collectively $C_1$-rainbow. 

Let $C'=C_1\setminus (\cup_{i\in X_1}C(P_i))$, so that, by \eqref{coldrain}, $|C'|=\xi n$.
Using \ref{mouseb2},  we can find $m_2$ paths $\{P_1, \dots, P_{m_2}\}$ such that $P_i$ is an $x_i,y_i$-path with length $k$, for each $i\in I_2$, so that these paths are vertex disjoint with internal vertices in $V_2=V_2'\cup W_2$ and are collectively $(C_2\cup C')$-rainbow and use every colour in $C'$. 

Let $C''=C_2\setminus (\cup_{i\in X_2}C(P_i))$, and note that $|C''|=|C_2\cup C'|-m_2k=|D_2|+\xi n-m_2k=|D_2|-4\xi n$.
Using \ref{mouseb4} and that $C_3=D_3'$, we can find $m_3$ paths $\{P_1, \dots, P_{m_3}\}$ such that $P_i$ is an $x_i,y_i$-path with length $k$, for each $i\in I_3$, so that the paths are vertex disjoint with internal vertices in $V_3=V_3'\cup W_3$ which are collectively $(C_3\cup C'')$-rainbow.

Then, for each $i\in [\ell]$, the path $P_i$ is an $x_i,y_i$-path with length $k$, so that all the paths are vertex disjoint with internal vertices in $V_1\cup V_2\cup V_3\subset V_0\cup V$ and which are collectively $D$-rainbow, as required.
\end{proof}


\section{Randomized tree embedding}\label{sec:almost}

In this section, we prove Theorem~\ref{nearembedagain}. We start by formalising what we mean by a random rainbow embedding of a tree.
\begin{definition}
For a probability space $\Omega$, tree $T$ and a coloured graph $G$, a randomized rainbow  embedding of $T$ into $G$ is
a triple $\phi=(V_\phi, C_\phi, T_\phi)$ consisting of a random set of vertices $V_\phi:\Omega\to V(G)$, a random set of colours $C_\phi:\Omega\to C(G)$, and a random subgraph $T_\phi:\Omega\to E(G)$ such that:
\begin{itemize}
\item $V(T_\phi)\subseteq V_\phi$ and $C(T_\phi)\subseteq C_\phi$ always hold.
\item With high probability, $T_\phi$ is a rainbow copy of $T$.
\end{itemize}
\end{definition}

We will embed a tree bit by bit, starting with a randomized rainbow embedding of a small tree and then extending it gradually. For this, we need a concept of one randomized embedding \emph{extending} another.
\begin{definition}\label{Definition_embedding_extension}
Let $G$ be a graph and $T_1\subseteq T_2$ two nested trees. Let $\phi_1=(V_{\phi_1}, C_{\phi_1}, S_{\phi_1})$ and $\phi_2=(V_{\phi_2}, C_{\phi_2}, S_{\phi_2})$ be randomized rainbow embeddings of $T_1$ and $T_2$ respectively into $G$. We say that $\phi_2$ extends $\phi_1$ if
\begin{itemize}
\item $T_{\phi_1}\subseteq T_{\phi_2}$, $C_{\phi_1}\subseteq C_{\phi_2}$, and $V_{\phi_1}\subseteq V_{\phi_2}$ always hold.
\item $V(T_{\phi_2})\setminus V(T_{\phi_1})\subset V_{\phi_2}\setminus V_{\phi_1}$ and $C(T_{\phi_2})\setminus C(T_{\phi_1})\subset C_{\phi_2}\setminus C_{\phi_1}$ always hold.
\end{itemize}
\end{definition}
The above definition implicitly assumes that the two randomized embeddings are defined on the same probability space, which is the case in our lemmas, except for Lemma~\ref{Lemma_extending_with_large_stars}. When $\phi_1$ and $\phi_2$ are defined on different probability spaces $\Omega_{\phi_1}$ and $\Omega_{\phi_2}$ respectively, we use the following definition. We say that an extension of $\phi_1$ is a  measure preserving transformation $f:\Omega_{\phi_2}\to \Omega_{\phi_1}$ (i.e.\ $\P(f^{-1}(A))=\P(A)$ for any $A\subseteq \Omega_{\phi_1}$). We say that $\phi_2$ extends $\phi_1$ if ``for every $\omega\in \Omega_{\phi_2}$ we have $T_{\phi_1}( f(\omega))\subseteq T_{\phi_2}$, $C_{\phi_1}( f(\omega))\subseteq C_{\phi_2}$,  $V_{\phi_1}( f(\omega))\subseteq V_{\phi_2}$, and $V(T_{\phi_2})\setminus V(T_{\phi_1}( f(\omega)))$ is 
contained in $V_{\phi_2}\setminus V_{\phi_1}( f(\omega))$ and $C(T_{\phi_2})\setminus C(T_{\phi_1}(f(\omega)))$ is contained in 
$C_{\phi_2}\setminus C_{\phi_1}( f(\omega))$''. Such a measure preserving transformation ensures that $\phi_1'=(T_{\phi_1}\circ f, C_{\phi_1}\circ f, V_{\phi_1}\circ f)$ is a randomized embedding of $T_1$ defined on $\Omega_{\phi_2}$ which is equivalent to $\phi_1$ (i.e.\ the probability of any outcomes  $\phi_1$ and $\phi_1'$ are the same). It also ensures that $\phi_2$ is an extension of $\phi_1'$ as in Definition~\ref{Definition_embedding_extension}.

The following lemma extends randomized embeddings of trees by adding a large star forest. Recall that a 
$p$-random subset of some finite set is formed by choosing every element of it independently 
with probability $p$.
\begin{lemma}[Extending with a large star forest]\label{Lemma_extending_with_large_stars}
Let $p\gg \beta\gg \gamma \gg d^{-1}, n^{-1}$ and $\log^{-1} n\gg d^{-1}$.
Let $T_1\subseteq T_2$ be forests such that $T_2$ is formed by adding stars with $\geq d$ leaves to vertices of $T_1$. Let  $K_{2n+1}$ be $2$-factorized and suppose that $|T_2|=(1-p)n$.
Let $\phi_1=(V_{\phi_1}, C_{\phi_1}, T_{\phi_1})$ be a randomized rainbow embedding of $T_1$ into $K_{2n+1}$ where $V_{\phi_1}$ and $C_{\phi_1}$ are both $\gamma$-random.

Then, $\phi_1$ can be extended into a randomized rainbow embedding $\phi_2=(V_{\phi_2}, C_{\phi_2}, T_{\phi_2})$ of $T_2$ so that $V_{\phi_2}$ and $C_{\phi_2}$ are $(1-p+\beta)$-random sets (with $V_{\phi_2}$ and $C_{\phi_2}$ allowed to depend on each other).

\end{lemma}
\begin{proof}
Choose $\alpha$ such that $\beta\gg \alpha \gg \gamma $. Let $\theta =|T_1|/n$ and $F=T_2\setminus T_1$, noting that $F$ is a star forest with degrees $\geq d$ and $e(F)=(1-p-\theta)n$. Let $I\subseteq V(T_1)$ be the vertices to which stars are added to get $T_2$ from $T_1$.
Let $\Omega$ be the probability space for $\phi_1$.
We call $\omega\in \Omega$ \emph{successful} if $T_{\phi_1}(\omega)$ is a copy of $T_1$ and if we have $|V_{\phi_1}(\omega)|, |C_{\phi_1}(\omega)|\leq 4\gamma n$.
Using Chernoff's bound and the fact that  $\phi_1$ is a randomized embedding of $T_1$ we have that with high probability a random $\omega \in \Omega$ is succesful.

For every successful $\omega$, let $J^{\omega}$ be the copy of $I$ in $T_{\phi_1}(\omega)$.
We can apply Lemma~\ref{Lemma_randomized_star_forest} with $F=F$, $J=J^{\omega}$,
$V=V(K_{2n+1})\setminus V_{\phi_1}(\omega)$, $C=C(K_{2n+1})\setminus C_{\phi_1}(\omega)$, $p'=p+\theta, \alpha=\alpha, \gamma'=4\gamma, d=d$ and $n=n$. This gives a probability space $\Omega^\omega$, and a randomized subgraph $F^{\omega}$, and randomized sets
$U^\omega\subseteq V(K_{2n+1})\setminus (V_{\phi_1}(\omega)\cup V(F^{\omega})),$ $D^{\omega}\subseteq  C(K_{2n+1})\setminus (C_{\phi_1}(\omega)\cup C(F^{\omega}))$ (so $F^{\omega}$, $U^\omega$, $D^{\omega}$ are functions from $\Omega^{\omega}$ to the families of subgraphs/subsets of vertices/sets of colours of $K_{2n+1}$ respectively). From Lemma~\ref{Lemma_randomized_star_forest} we know that, for each $\omega$,  $F^{\omega}$ is with high probability a copy of $F$, and that $U^{\omega}$ and $D^{\omega}$ are $(1-\alpha)(p+\theta)$-random subsets of $V(K_{2n+1})\setminus V_{\phi_1}(\omega)$ and $C(K_{2n+1})\setminus C_{\phi_1}(\omega)$ respectively. Setting $T^{\omega}=F^{\omega}\cup T_{\phi_1}(\omega)$ gives a subgraph which is a copy of $T_2$ with high probability (for successful $\omega$).
For every unsuccessful $\omega$, set $T^{\omega}=T_{\phi_1}(\omega)$, and choose  $U^\omega\subseteq V(K_{2n+1})\setminus V_{\phi_1}(\omega),$ $D^{\omega}\subseteq  C(K_{2n+1})\setminus C_{\phi_1}(\omega)$ to be independent $(1-\alpha)(p+\theta)$-random  subsets (in this case letting $\Omega^{\omega}$ be an arbitrary probability space on which such $U^{\omega}, D^{\omega}$ are defined).

Let $\Omega_2=\{(\omega, \omega'): \omega\in \Omega, \omega'\in \Omega^{\omega}\}$ and set $\P((\omega, \omega'))=\P(\omega)\P(\omega')$. Notice that $\Omega_2$ is a probability space.
Let $T_{\phi_2}$ be a random subgraph formed by choosing $(\omega, \omega')\in \Omega_2$, and setting $T_{\phi_2}= T^\omega(\omega')$. Similarly define $U=U^{\omega}(\omega')$ and $D=D^{\omega}(\omega')$. Notice that $V\setminus V_{\phi_1}$ is $(1-\gamma)$-random and that $U|V_{\phi_1}$ is a $(1-\alpha)(p+\theta)$-random subset of $V\setminus V_{\phi_1}$. By Lemma~\ref{Lemma_mixture_of_p_random_sets}, $U$  is a $(1-\alpha)(1-\gamma)(p+\theta)$-random subset of $V(K_{2n+1})$. Similarly, $D$ is a $(1-\alpha)(1-\gamma)(p+\theta)$-random  subset of $C(K_{2n+1})$. Since $\beta\gg \alpha,\gamma$, we have $(1-\alpha)(1-\gamma)(p+\theta)\geq p-\beta$, and therefore can choose $(p-\beta)$-random subsets $U'$ and $D'$ such that $U'\subseteq  U$ and  $D'\subseteq D$.
Now $V_{\phi_2}:=V(K_{2n+1})\setminus U'$ and $C_{\phi_2}=C(K_{2n+1})\setminus D'$ are $(1-p+\beta)$-random sets of vertices/colours.
We also have that with high probability $T_{\phi_2}$ is a copy of $T_2$ since
$$\P(\text{$T_{\phi_2}(\omega, \omega')$ is not a copy of $T_2$})\leq \P(\text{$\omega$ unsuccessful})+\P(\text{$\omega$ successful and $T_{\omega}$ not a copy of $T_2$}).$$
Thus the required extension of $\phi_1$ is  $\phi_{2}=(T_{\phi_2}, V_{\phi_2}, C_{\phi_2})$ together with the measure preserving transformation $f:\Omega_2\to \Omega$ with $f:(\omega, \omega')\to\omega$.
\end{proof}

The following lemma extends randomized embeddings of trees by adding connecting paths.
\begin{lemma}[Extending with connecting paths]\label{Lemma_extending_with_connecting_paths}
Let $p\gg q\gg n^{-1}$.
Let $T_1\subseteq T_2$ be forests such that $T_2$ is formed by adding $qn$ paths of length $3$ connecting different components of $T_1$. Let  $K_{2n+1}$ be $2$-factorized.
Let $\phi_1=(V_{\phi_1}, C_{\phi_1}, T_{\phi_1})$ be a randomized rainbow embedding of $T_1$ into $K_{2n+1}$ and  let $U\subseteq V(K_{2n+1})\setminus V_{\phi_1}$, $D\subseteq C(K_{2n+1})\setminus C_{\phi_1}$ be $p$-random, independent subsets. Let $V_{\phi_2}=V_{\phi_1}\cup U$ and $C_{\phi_2}=C_{\phi_1}\cup D$.

Then, $\phi_1$ can be extended into a randomized rainbow embedding $\phi_2=(V_{\phi_2}, C_{\phi_2}, T_{\phi_2})$ of $T_2$.
\end{lemma}
\begin{proof}
Let $\Omega$ be the probability space for $\phi_1$.
We say that $\omega\in \Omega$ is \emph{successful} if $T_{\phi_1}(\omega)$ is a copy of $T_1$ and the conclusion of Lemma~\ref{Lemma_few_connecting_paths} holds with $V=U$, $C=D$, $p=p$, $q=q$ and $n=n$. As $\phi_1$ is a randomized embedding of $T_1$, and by Lemma~\ref{Lemma_few_connecting_paths}, we have that, with high probability, $\omega$ is successful.

For each successful $\omega$, let $x_1^{\omega}, y_1^{\omega}, \dots, x_{qn}^{\omega}, y_{qn}^{\omega}$ be the vertices of $T_{\phi_1}(\omega)$ which need to be joined by paths of length $3$ to get a copy of $T_2$. From the conclusion of Lemma~\ref{Lemma_few_connecting_paths}, for each successful $\omega$, we can find paths $P_i^\omega$, $i\in [qn]$, so that the paths are vertex disjoint and collectively $D$-rainbow, and each path $P_i^\omega$ is an $x_i,y_i$-path with length 3 and internal vertices in $U$.
Letting $T_{\phi_2}=T_{\phi_1}(\omega)\cup P_1^{\omega}\cup \dots\cup P_{qn}^{\omega}$, $(V_{\phi_2}, C_{\phi_2}, T_{\phi_2})$ gives the required randomized embedding of $T_2$.
\end{proof}

The following lemma extends randomized embeddings of trees by adding a sequence of matchings of leaves.
\begin{lemma}[Extending with matchings]\label{Lemma_extending_with_matchings}
Let $\ell^{-1},p, q\gg  n^{-1}$.
Let $T_1\subseteq T_{\ell}$ be forests such that $T_{\ell}$ is formed by adding a sequence of $\ell$ matchings of leaves to $T_1$. Let  $K_{2n+1}$ be $2$-factorized.  Suppose $|T_{\ell}|-|T_1|\leq pn$.
Let $\phi_1=(V_{\phi_1}, C_{\phi_1}, T_{\phi_1})$ be a  randomized rainbow embedding of $T_1$ into $K_{2n+1}$.  Let $U_{ind}\subseteq V(K_{2n+1})\setminus V_{\phi_1}$, $D_{ind}\subseteq C(K_{2n+1})\setminus C_{\phi_1}$ be $q$-random, independent subsets.  Let $U_{dep}\subseteq V(K_{2n+1})\setminus (V_{\phi_1}\cup U_{ind})$ be $p/2$-random, and $D_{dep}\subseteq C(K_{2n+1})\setminus(C_{\phi_1}\cup D_{ind})$  $p$-random (possibly depending on each other).

Then $\phi_1$ can be extended into a randomized rainbow embedding $\phi_{\ell}=(V_{\phi_{\ell}}, C_{\phi_{\ell}}, T_{\phi_{\ell}})$ of $T_{\ell}$ into $V_{\phi_{\ell}}=V_{\phi_1}\cup U_{ind}\cup U_{dep}$ and $C_{\phi_{\ell}}=C_{\phi_1}\cup D_{ind}\cup D_{dep}$.
\end{lemma}
\begin{proof}
Without loss of generality, suppose that  $|T_{\ell}|-|T_1|= pn$.
Define forests $T_1, \dots, T_{\ell}$, and $p_1, \dots, p_{\ell}\in [0,1]$ such that each $T_{i+1}$ is constructed from $T_i$ by adding a matching of $p_in$ leaves.
Randomly partition $U_{dep}=\{U_{dep}^1, \dots, U_{dep}^{\ell}\}$ and $D_{dep}=\{D_{dep}^1, \dots, D_{dep}^{\ell}\}$ so that each set $U_{dep}^i\subseteq V(K_{2n+1})$ is $p_i/2$-random and each set $D_{dep}^i\subseteq C(K_{2n+1})$ is $p_i$-random.
Randomly partition   $U_{ind}=\{U_{ind}^1, \dots, U_{ind}^{\ell}\}, D_{ind}=\{D_{ind}^1, \dots, D_{ind}^{\ell}\}$ so that each set $U_{ind}^i$ and $D_{ind}^i$ is $(q/\ell)$-random.

Let $\Omega$ be the probability space for $\phi_1$.
We say that $\omega\in \Omega$ is \emph{successful} if $T_{\phi_1}(\omega)$ is a copy of $T_1$ and, for each $i\in [\ell]$, the conclusion of Lemma~\ref{Lemma_sat_matching_random_embedding} holds for $U_{dep}^{i}$, $U_{ind}^{i}$, $D_{dep}^{i}$, $D_{ind}^{i}$, $p=p_i$, $\gamma = q/\ell$ and    $n=n$.
As $\phi_1$ is a randomized embedding of $T_1$, and by Lemma~\ref{Lemma_sat_matching_random_embedding}, we have that, with high probability, $\omega$ is successful. Note that, for this, we take a union bound over $\ell$ events, using that the conclusion of Lemma~\ref{Lemma_sat_matching_random_embedding} holds
with probability $1-o(n^{-1})$ in each application.

For each successful $\omega$, define a sequence of trees $T_0^{\omega},T_1^{\omega}, \dots, T_{\ell}^{\omega}$. Fix $T_0^{\omega}=T_{\phi_1}(\omega)$ and recursively construct $T_{i+1}^{\omega}$ from $T_i^{\omega}$ by adding an appropriate $(D_{dep}^i\cup D_{ind}^i)$-rainbow size $p_in$ matching into $U_{dep}^i\cup U_{ind}^i$ to make $T_i^\omega$ into a copy of $T_{i+1}$ (which exists as the conclusion of Lemma~\ref{Lemma_sat_matching_random_embedding} holds because $\omega$ is successful).
Letting $T_{\phi_{\ell}}(\omega)= T_{\ell}^{\omega}$ gives the required randomized embedding of $T_{\ell}$.
\end{proof}

The following lemma gives a randomized embedding of any small tree.
\begin{lemma} [Small trees with a replete subset]\label{Lemma_embedding_small_tree}
Let $q\gg \gamma\geq \nu \gg \xi \gg n^{-1}$.
Let $T$ be a tree with $|T|= \gamma n$ containing a set $U\subset V(T)$ with $|U|= \nu n$. Let $K_{2n+1}$ be $2$-factorized with $V_\phi, C_\phi$ $q$-random sets of vertices and colours respectively.
Then, there is a randomized rainbow embedding $\phi=(T_\phi, V_\phi, C_\phi)$ along with independent $q/2$-random sets $V_0\subseteq V_{\phi}\setminus V(T_{\phi})$,  $C_0\subseteq C_{\phi}\setminus C(T_{\phi})$    with $(V_0, U_\phi)$ $\xi n$-replete with high probability (where $U_\phi$ is the copy of $U$ in $T_\phi$).
\end{lemma}
\begin{proof}
Choose $\alpha$ such that $q\gg \alpha\gg \gamma, \nu$.
Inside $V_{\phi}$ choose disjoint sets $V_0, V_1, V_2$ which are $q/2$-random, $\alpha$-random, and $100\nu/q$-random respectively.
Inside $C_{\phi}$ choose disjoint sets  $C_0, C_1, C_2$ which are $q/2$-random, $q/4$-random, and $q/4$-random respectively. Notice that the total number of edges of any particular color is $2n+1$
and the probability that any such edge is going from $V_2$ to $V_1  \cup V_2$ is 
$2(100\nu/q)(100\nu/q+\alpha)$. Therefore by Lemmas~\ref{Lemma_high_degree_into_random_set} and \ref{Lemma_Azuma} the following hold with high probability.
\begin{enumerate}[label = {\bfseries \Alph{propcounter}\arabic{enumi}}]
    \item \label{eq_repletetree_1} Every vertex $v$ has $|N_{C_1}(v)\cap V_1|\geq \alpha q n/4\geq 3|T|$ and $|N_{C_2}(v)\cap V_2|\geq (100\nu/q)\cdot qn/4\geq 4|U|$.
 \item \label{eq_repletetree_3} Every colour has at most $400\alpha (\nu/q) n\leq 0.1|U|$ edges from $V_2$ to $V_1\cup V_2$.
\end{enumerate}
Using \ref{eq_repletetree_1}, we can greedily find a rainbow copy $T_\phi$ of $T$ in $V_1\cup V_2\cup C_1\cup C_2$ with $U$ embedded into $V_2$. Indeed,  embed one vertex at a time, always embedding a vertex which has one neighbour into preceding vertices. Moreover,  put a vertex from $T\setminus U$ into $V_1$ using an edge of colour from $C_1$ and put a vertex from $U$ into $V_2$ using an edge of colour from $C_2$.
Let $U_\phi$ be the copy of $U$ in $T_\phi$. By~\ref{eq_repletetree_3}, and as $|U_\phi|=\nu n$, $(U_\phi, V(K_{2n+1})\setminus (V_1\cup V_2))$ is  $0.9\nu n$-replete. Indeed, for every color there are at least $|U_\phi|$ edges of this color incident to $U_\phi$ and at most $0.1|U_\phi|$ of them going to $V_1 \cup V_2$.

Consider a set $V_0'\subseteq V(K_{2n+1})$ which is $\big(\frac{1}{1-\alpha-100\nu/q}\big)\cdot (q/2)$-random and independent of $V_1, V_2, C_0, C_1, C_2$. Notice that $V_0'\setminus (V_1\cup V_2)$ is $q/2$-random. Thus the joint distributions of the families of random sets $\{V_0'\setminus (V_1\cup V_2),  V_1, V_2, C_0, C_1, C_2\}$ and $\{V_0,  V_1, V_2, C_0, C_1, C_2\}$ are exactly the same. By Lemma~\ref{Lemma_inheritence_of_lower_boundedness_random} applied with $A=U_{\phi}, B=V(K_{2n+1})\setminus (V_1\cup V_2), V=V_0'$, the pair $(U_{\phi}, V_0'\setminus (V_1\cup V_2))$ is $\xi$-replete with high probability. Since $(U_{\phi}, V_0'\setminus (V_1\cup V_2))$ has the same distribution as $(U_{\phi}, V_0)$, the latter pair is also  $\xi$-replete with high probability.
\end{proof}

The following is the main result of this section. It finds a rainbow embedding of every nearly-spanning tree. Given the preceding lemmas, the proof is quite simple. First, we decompose the tree into star forests, matchings of leaves, and connecting paths using  Lemma~\ref{Lemma_tree_splitting}. Then, we embed each of these parts using the preceding lemmas in this section.
\begin{proof}[Proof of Theorem~\ref{nearembedagain}]
Choose $\beta$ and $d$ such that $1\geq  \eps\gg  \eta\gg \beta \gg  \mu\gg  d^{-1}\gg  \xi$ with $k^{-1}\gg d^{-1}$.
By Lemma~\ref{Lemma_tree_splitting}, there are forests $T_1^{\mathrm{small}}\subseteq T_2^{\mathrm{stars}}\subseteq T_3^{\mathrm{match}}\subseteq T_4^{\mathrm{paths}}\subseteq T_5^{\mathrm{match}}=T'$ satisfying \ref{rush1} -- \ref{rush5} with  $d=d, n=(1-\eps)n$ and $U=U$. Fix $U'=U\cap T_1^{\mathrm{small}}$.

Fix $p'=1-|T_2^{\mathrm{stars}}|/n=(|T'|-|T_2^{\mathrm{stars}}|)/n+\eps$. We claim that $p'\geq (p+\eps)/2$. To see this, let $T''$ be the tree with $|T''|=pn$ formed by deleting leaves of vertices with $\geq k$ leaves in $T'$. Let $A\subset V(T_1^{\mathrm{small}})$ be the vertices to which leaves need added to get $T_2^{\mathrm{stars}}$ from $T_1^{\mathrm{small}}$, and, for each $v\in A$, let $d_v\geq d$ be the number of such leaves that are added to $v$. Then we have that $\sum_{v\in A} d_v=|T_2^{\mathrm{stars}}|-|T_1^{\mathrm{small}}|$.
Let $d'_v$ be the number of these leaves added to $v$ which are not themselves leaves in $T'$. 
Note that, for each $v\in A$, at least $d_v'$ neighbours of $v$ lie in $V(T')\setminus V(T_2^{\mathrm{stars}})$. Thus, we have
$\sum_{v\in A}d_v'\leq |T'|-|T_2^{\mathrm{stars}}|$. Moreover, for each $v\in A$ with $d_v-d'_v \geq k$ we deleted  $d_v-d'_v$ leaves attached to $v$ when we formed $T''$ from $T'$, otherwise we deleted nothing. In both cases we deleted at least $d_v-d'_v-k$ leaves for each $v \in A$.
Therefore, using that $|A| \leq|T_1^{\mathrm{small}}|\leq n/d$, we have
$$|T'|-|T''| \geq \sum_{v\in A}(d_v-d_v'-k)=|T_2^{\mathrm{stars}}|-|T_1^{\mathrm{small}}|-\sum_{v\in A}(d_v'+k) \geq |T_2^{\mathrm{stars}}|-(k+1)n/d-\sum_{v\in A}d_v'\,.$$
Thus,
$$
pn= |T''| \leq |T'|-|T_2^{\mathrm{stars}}|+(k+1)n/d+\sum_{v\in A}d_v'\leq 2(|T'|-|T_2^{\mathrm{stars}}|)+\eps n= 2p'n-\eps n,
$$
implying, $p'\geq (p+\eps)/2$.

For each $i\in \{2, \dots, 5\}$ (and the label $*$ meaning ``$\mathrm{small},\, \mathrm{stars},\,  \mathrm{match}$ or $\mathrm{paths}$" appropriately), let $p_i=(|T_i^*|-|T_{i-1}^*|)/n$. Set $p_6=\epsilon-\beta.$
These constants represent the proportion of colours which are used for embedding each of the subtrees (with $p_6$ representing the proportion of colours left over for the set $C$ in the statement of the lemma).
Notice that $p'=p_3+p_4+p_5+p_6+\beta$.
For each $i\in [6]$, choose disjoint, independent $2\mu$-random sets $V^{i}_{ind}, C_{ind}^i$. Do the following.

\begin{itemize}
    \item Apply Lemma~\ref{Lemma_embedding_small_tree} to $T_1^{\mathrm{small}}$, $U=U'$, $V_\phi=V_{1}^{ind}$, $C_\phi=C_{1}^{ind}$, $q=2\mu$, $\gamma=|T_1^{small}|/n$, $\nu=|U'|/n$ and $\xi=\xi$.
This gives a  randomized rainbow embedding $\phi_1=(T_{\phi_1}, V_{1}^{ind}, C_{1}^{ind})$ of $T_1^{small}$ and  $\mu$-random, independent sets $V_0\subseteq V_{1}^{ind}\setminus V(T_{\phi_1})$, $C_0\subseteq C_{1}^{ind}\setminus C(T_{\phi_1})$ with $(U'_{\phi_1},V_0)$ $\xi n$-replete with high probability (where $U'_{\phi_1}$ is the copy of $U'$ in $T_{\phi_1}$).
\end{itemize}
Notice that $\phi_1'=(V_1^{ind}\cup \dots \cup V_6^{ind}, C_1^{ind}\cup \dots\cup C_6^{ind}, T_{\phi_1} )$ is also a  randomized embedding  of $T_1^{\mathrm{small}}$. 
Then, do the following.
\begin{itemize}
\item Apply  Lemma~\ref{Lemma_extending_with_large_stars} with $T_1=T_1^{\mathrm{small}}$, $T_2=T_2^{\mathrm{stars}}$, $\phi_1=\phi_1'$, $p=p', \beta=\beta, \gamma=12\mu, d=d, n=n$. This gives a randomized embedding $\phi_{2}=(V_{\phi_2}, C_{\phi_2}, T_{\phi_2})$ of $T_2^{\mathrm{stars}}$ extending $\phi_{1}'$, with $V_{dep}=V(K_{2n+1})\setminus V_{\phi_2}$, $C_{dep}=C(K_{2n+1})\setminus C_{\phi_2}$ $(p'-\beta)$-random sets of vertices/colours.
\end{itemize} 
Notice that $\phi_2'=(V_{\phi_2}\setminus (V_3^{ind}\cup \dots \cup V_6^{ind}), C_{\phi_2}\setminus (C_3^{ind}\cup \dots \cup C_6^{ind}), T_{\phi_2})$ is also a randomized embedding  of $T_2^{\mathrm{stars}}$. Indeed, recall that by our definition of extension we embed vertices of $T_2^{\mathrm{stars}}\setminus T_1^{small}$ outside of the set $V_1^{ind}\cup \dots \cup V_6^{ind}$ using the edges whose colors are also outside of $C_1^{ind}\cup \dots\cup C_6^{ind}$. Moreover, by our constriction, $V(T_{\phi_1})$ is disjoint from
$V_2^{ind}\cup \dots \cup V_6^{ind}$ and $C(T_{\phi_1})$ is disjoint from
$C_2^{ind}\cup \dots\cup C_6^{ind}$.

Using that $p_3+p_4+p_5+p_6 = p'-\beta$ and Lemma~\ref{Lemma_mixture_of_p_random_sets}, randomly partition $C_{dep}= C_{dep}^3\cup C_{dep}^4\cup C_{dep}^5\cup C_{dep}^6$  with the sets $C_{dep}^i$ $p_i$-random subsets of $C(K_{2n+1})$.

Using that $p'-\beta= (p_3+p_4+p_5+p_6)/2 + (p'-\beta)/2\geq p_3/2+p_4/2+p_5/2+ (p+\eps)/6$, inside $V_{dep}$ we can choose disjoint sets $V_{dep}^3, V_{dep}^4, V_{dep}^5, V_{dep}^6$ which are $\rho$-random for $\rho=p_3/2, p_4/2, p_5/2, (p+\eps)/6$ respectively. We remark that some of the random sets ($V_{ind}^2, C_{ind}^2, V_{ind}^6, C_{ind}^6, V_{dep}^4,$ and $C_{dep}^4$) will not actually used in the embedding. They are only there because it is notationally convenient to allocate equal amounts of vertices/colours for the different parts of the embedding.
\begin{itemize}
\item Apply Lemma~\ref{Lemma_extending_with_matchings} with $T_{\ell}=T_3^{\mathrm{match}}, T_1=T_{2}^{\mathrm{stars}}$, $V_{dep}=V_{dep}^3$, $C_{dep}=C_{dep}^3$, $V_{ind}=V_{ind}^3$, $C_{ind}=C_{ind}^3$, $\phi_1=\phi_2'$, $\ell=d, p=p_3, q=2\mu, n=n$. This gives we get a  randomized embedding $\phi_3=(V_{\phi_3}, C_{\phi_3}, T_{\phi_3})$  of $T_3^{\mathrm{match}}$ into $V^3_{ind}\cup V^3_{dep}\cup V_{\phi_2'}$, $V^3_{dep}\cup  V^3_{ind}\cup C_{\phi_2'}$ extending  $\phi_2$.
\item Apply Lemma~\ref{Lemma_extending_with_connecting_paths} with $T_2=T_4^{\mathrm{paths}}, T_1=T_{3}^{\mathrm{match}}$, $V_{ind}=V_{ind}^4$, $C_{ind}=C_{ind}^4$, $\phi_1=\phi_3$, $p=2\mu, q=(2k)^{-1}, n=n$. This gives a  randomized embedding $\phi_4=(V_{\phi_4}, C_{\phi_4}, T_{\phi_4})$  of $T_4^{\mathrm{paths}}$ into $V^4_{ind}\cup V_{\phi_3}, C^4_{ind}\cup C_{\phi_3}$ extending  $\phi_3$.
\item By Lemma~\ref{Lemma_extending_with_matchings} with $T_{\ell}=T_5^{\mathrm{match}}, T_1=T_{4}^{\mathrm{paths}}$, $V_{dep}=V_{dep}^5$, $C_{dep}=C_{dep}^5$, $V_{ind}=V_{ind}^5$, $C_{ind}=C_{ind}^5$, $\phi_1=\phi_4$, $\ell=d, p=p_5, q=2\mu, n=n$, we get a  randomized embedding $\phi_5=(V_{\phi_5}, C_{\phi_5}, T_{\phi_5})$  of $T_5^{\mathrm{match}}$ into $V^5_{ind}\cup V^5_{dep}\cup V_{\phi_4}$, $C^5_{dep}\cup  C^5_{ind}\cup V_{\phi_4}$ extending  $\phi_4$.
\end{itemize}
Now the lemma holds with $\hat T'=T_{\phi_5}$, $V_0=V_0$, $C_0=C_0$ (recall that these sets are from the application of Lemma~\ref{Lemma_embedding_small_tree} above),  $V=V_{dep}^6$, and $C\subseteq C_{dep}^6$ a $(1-\eta)\eps$-random subset. 
\end{proof}


\section{The embedding in Case C}\label{sec:lastC}\label{sec:caseC}
For our embedding in Case C, we distinguish between those trees with one very high degree vertex, and those trees without. The first case is covered by Theorem~\ref{Theorem_one_large_vertex}. In this section, we prove Lemma~\ref{Lemma_small_tree}, which covers the second case and thus completes the proof of Theorem~\ref{Theorem_case_C}.

For trees with no very high degree vertex, we start with the following lemma which embeds a small tree in the $ND$-colouring in a controlled fashion. In particular, we wish to embed a prescribed small portion of the tree into an interval in the cyclic ordering of $K_{2n+1}$ so that the distance between any two consecutive vertices in the image of the embedding is small.

\begin{lemma}[Embedding a small tree into prescribed intervals]\label{Lemma_small_tree}
Let $n\geq 10^5$.
Let $K_{2n+1}$ be $ND$-coloured, and let $T$ be a tree with at most $n/100$ vertices whose vertices are partitioned as  $V(T)=V_0\cup V_1\cup V_2$ with $|V_1|, |V_2|\leq 2n/\log^4 n$.
Let $I_0, I_1, I_2$ be disjoint intervals in $V(K_{2n+1})$ with $|I_0|\geq 7|V_0|$, $|I_1|\geq 8|V_1|\log^3 n $ and $|I_2|\geq 8|V_2|\log^3 n $.

Then, there is a rainbow copy $S$ of $T$ in $K_{2n+1}$ such that, for each $0\leq i\leq 2$, the image of $V_i$ is contained in $I_i$, and, furthermore, any consecutive pair of vertices of $S$ in $I_1$, or in $I_2$, are within distance $16\log^3 n$ of each other.
\end{lemma}
\begin{proof}
Let $k\in \{\lfloor 2\log n\rfloor,\lfloor 2\log n\rfloor-1\}$ be odd.
Choose   consecutive subintervals $I_1^{1}, \dots, I_1^{|V_1|}$ in $I_1$ of length $k^3$, and choose consecutive subintervals $I_2^{1}, \dots, I_2^{|V_2|}$   in $I_2$ of length $k^3$.
Partition $C(K_{2n+1})$ as $C_0\cup C_1^{1}\cup \dots\cup C_1^{k}\cup C_2^{1}\cup \dots\cup C_2^{k}$ so that, for each $m\in [n]$ and $i\in [k]$,
\[
m\in \left\{\begin{array}{ll}
C_0 & \text{ if }m\text{ is odd},\\
C_1^i & \text{ if }m\text{ is even and }m\equiv i\mod{2k+1},\\
C_2^i & \text{ if }m\text{ is even and }m\equiv i+k\mod{2k+1}.
\end{array}
\right.
\]
Note that, for every $p\in [2]$, $i\in [k]$, $j\in [|V_p|]$, and $v\notin I_p^j$, since $C_p^i$ contains only even colours and $k$ is odd, $v$ has at least $\lfloor |I_p^j|/2(2k+1)\rfloor\geq k$ colour-$C_p^i$ neighbours in $I_p^j$ (we call a vertex $u$ a colour-$C_p^i$ neighbour of $v$ if the colour of the edge $vu$ belongs to $C_p^i$).

Order $V(T)$ arbitrarily as $v_1, \dots, v_{m}$ such that $\{v_1, \dots, v_i\}$ forms a subtree of $T$ for each $i\in [m]$.
We embed the vertices $v_1, \dots, v_{m}$ one by one, in $m$ steps, so that, in step $i$, we embed $v_i$ to some $u_i\in V(K_{2n+1})$. Since $T[v_1, \dots, v_i]$ is a tree, for each $i$, there is a unique $f(i)<i$ for which $v_{f(i)}v_i\in E(T)$. We will maintain that the edges $u_iu_{f(i)}$ have different colours for each $i$.
At step $i$ (the step in which $u_1, \dots, u_{i-1}$ have already been chosen and we are choosing $u_i$) we say that a vertex $x$ is \emph{free} if $x\not\in\{u_1, \dots, u_{i-1}\}$. We say that an interval $I_s^t$ is \emph{free}, if it contains no vertices from $\{u_1, \dots, u_{i-1}\}$. We say that a colour is \emph{free} if it does not occur on $\{u_{f(j)}u_j:2\leq j\leq i-1\}$. The procedure to choose $u_i$ at step $i$ is as follows.
\begin{itemize}
\item If $v_i\in V_0$, then $u_i$ is embedded to any free vertex in $I_0$ so that, if $i\geq 2$, $u_{f(i)}u_i$ is any free colour in $C_0$.
\item For $p\in [2]$, if $v_i\in V_p$, then $u_i$ is embedded into any free interval $I_p^i$ so that, if $i\geq 2$, $u_{f(i)}u_i$ uses any free colour in $C_p^j$, for $j$ as small as possible.
\end{itemize}
The lemma follows if we can embed all the vertices in this way. Indeed, as we have exactly the right number of intervals to embed one vertex per interval, consecutive vertices in $I_1$ (or $I_2$) are in consecutive intervals $I_1^i$, and hence at most $16\log^3n$ apart. Since $C_0$ contains all the odd colours,
vertex $u_{f(i)}$ has $\geq \lfloor |I_0|/2\rfloor-1 > 3|V_0|$ colour-$C_0$ neighbours in $I_0$. At each step at most $|V_0|$ of the vertices in $I_0$ are occupied, and at most $|V_0|$ colours are used. Since the $ND$-colouring is locally $2$-bounded,
this forbids at most $3|V_0|$ vertices. Therefore, there is always room to embed each $v_i \in V_0$ into $I_0$.
To finish the proof of the lemma, it is sufficient to show that, throughout the process, for each $p\in [2]$, there will always be sets $C_p^j$ which are never used.

\begin{claim}
Let $p\in [2]$, $i\in [m]$ and $s\in [k]$.
At step $i$, if there are $\geq |V_p|/2^{s}$ free intervals $I_p^j$, then no colour from $C_p^s$ has been used up to step $i$.
\end{claim}
\begin{proof}
The proof is by induction on $i$. The initial case $i=0$ is trivial.
Suppose it is true for steps $1,\ldots,i-1$ and suppose that there are still at least $|V_p|/2^s$ free intervals $I_p^j$. Now, by the induction hypothesis for $p'=p$, $i'=i-1$ and $s'=s-1$, colours from $C_p^{s-1}$ were only used when there were fewer than $|V_p|/2^{s-1}$ free intervals $I_p^j$, and hence at most $|V_p|/2^{s-1}$ vertices were embedded using colours from $C_p^{s-1}$. In each free interval, $u_{f(i)}$ has at least $k$ colour-$C_p^{s-1}$ neighbours.
and hence at least $k|V_p|/2^s$ colour-$C_p^{s-1}$ neighbours in the free intervals in total. As the $ND$-colouring is locally 2-bounded, this gives at least
$k|V_p|/2^{s+1}> |V_p|/2^{s-1}$ colour-$C_p^{s-1}$ neighbours adjacent to $u_{f(i)}$ by edges of different colours. Therefore $u_{f(i)}$ has some colour $C_p^{s-1}$-neighbour $u_i$ in a free interval so that $u_{f(i)}u_i$ is a free colour. Thus, no colour from $C_p^s$ is used in step $i$.
\end{proof}
Taking $s=k$, we see that colours from $C_1^k$ and $C_2^k$ never get used while there is at least one free interval. As there are exactly as many intervals as vertices we seek to embed, the colours from $C_1^k$ and $C_2^k$ are never used, and hence the procedure successfully embeds $T$.
\end{proof}

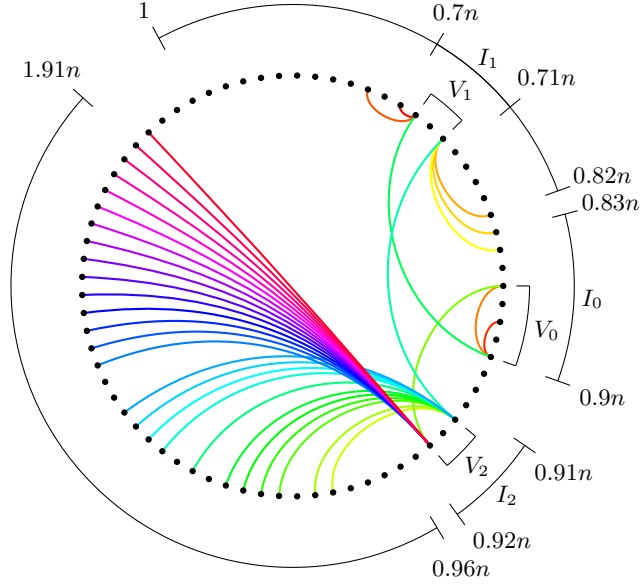
\begin{figure}[h]
\begin{center}
{
\begin{tikzpicture}[scale=0.7,define rgb/.code={\definecolor{mycolor}{rgb}{#1}},
                    rgb color/.style={define rgb={#1},mycolor}]]

\draw [thick,rgb color={1,0,0}, rotate = {4.9315068493*11}] (0:4) to [in={4.9315068493*1+180},out=180] ({4.9315068493 *1}:4);
\draw [thick,rgb color={1,0.166666666666667,0}, rotate = {4.9315068493*-4}] (0:4) to [in={4.9315068493*2+180},out=180] ({4.9315068493 *2}:4);
\draw [thick,rgb color={1,0.333333333333333,0}, rotate = {4.9315068493*11}] (0:4) to [in={4.9315068493*3+180},out=180] ({4.9315068493 *3}:4);
\draw [thick,rgb color={1,0.5,0}, rotate = {4.9315068493*-4}] (0:4) to [in={4.9315068493*4+180},out=180] ({4.9315068493 *4}:4);
\draw [thick,rgb color={1,0.666666666666667,0}, rotate = {4.9315068493*4}] (0:4) to [in={4.9315068493*5+180},out=180] ({4.9315068493 *5}:4);
\draw [thick,rgb color={1,0.833333333333333,0}, rotate = {4.9315068493*3}] (0:4) to [in={4.9315068493*6+180},out=180] ({4.9315068493 *6}:4);
\draw [thick,rgb color={1,1,0}, rotate = {4.9315068493*2}] (0:4) to [in={4.9315068493*7+180},out=180] ({4.9315068493 *7}:4);
\draw [thick,rgb color={0.833333333333333,1,0}, rotate = {4.9315068493*-16}] (0:4) to [in={4.9315068493*8+180},out=180] ({4.9315068493 *8}:4);
\draw [thick,rgb color={0.666666666666667,1,0}, rotate = {4.9315068493*-17}] (0:4) to [in={4.9315068493*9+180},out=180] ({4.9315068493 *9}:4);
\draw [thick,rgb color={0.5,1,0}, rotate = {4.9315068493*-10}] (0:4) to [in={4.9315068493*10+180},out=180] ({4.9315068493 *10}:4);
\draw [thick,rgb color={0.333333333333333,1,0}, rotate = {4.9315068493*-19}] (0:4) to [in={4.9315068493*11+180},out=180] ({4.9315068493 *11}:4);
\draw [thick,rgb color={0.166666666666667,1,0}, rotate = {4.9315068493*-20}] (0:4) to [in={4.9315068493*12+180},out=180] ({4.9315068493 *12}:4);
\draw [thick,rgb color={0,1,0}, rotate = {4.9315068493*-21}] (0:4) to [in={4.9315068493*13+180},out=180] ({4.9315068493 *13}:4);
\draw [thick,rgb color={0,1,0.166666666666667}, rotate = {4.9315068493*-22}] (0:4) to [in={4.9315068493*14+180},out=180] ({4.9315068493 *14}:4);
\draw [thick,rgb color={0,1,0.333333333333333}, rotate = {4.9315068493*-4}] (0:4) to [in={4.9315068493*15+180},out=180] ({4.9315068493 *15}:4);
\draw [thick,rgb color={0,1,0.5}, rotate = {4.9315068493*-24}] (0:4) to [in={4.9315068493*16+180},out=180] ({4.9315068493 *16}:4);
\draw [thick,rgb color={0,1,0.666666666666667}, rotate = {4.9315068493*-8}] (0:4) to [in={4.9315068493*17+180},out=180] ({4.9315068493 *17}:4);
\draw [thick,rgb color={0,1,0.833333333333333}, rotate = {4.9315068493*-26}] (0:4) to [in={4.9315068493*18+180},out=180] ({4.9315068493 *18}:4);
\draw [thick,rgb color={0,1,1}, rotate = {4.9315068493*-27}] (0:4) to [in={4.9315068493*19+180},out=180] ({4.9315068493 *19}:4);
\draw [thick,rgb color={0,0.833333333333333,1}, rotate = {4.9315068493*-28}] (0:4) to [in={4.9315068493*20+180},out=180] ({4.9315068493 *20}:4);
\draw [thick,rgb color={0,0.666666666666667,1}, rotate = {4.9315068493*-29}] (0:4) to [in={4.9315068493*21+180},out=180] ({4.9315068493 *21}:4);
\draw [thick,rgb color={0,0.5,1}, rotate = {4.9315068493*-32}] (0:4) to [in={4.9315068493*22+180},out=180] ({4.9315068493 *22}:4);
\draw [thick,rgb color={0,0.333333333333333,1}, rotate = {4.9315068493*-33}] (0:4) to [in={4.9315068493*23+180},out=180] ({4.9315068493 *23}:4);
\draw [thick,rgb color={0,0.166666666666667,1}, rotate = {4.9315068493*-34}] (0:4) to [in={4.9315068493*24+180},out=180] ({4.9315068493 *24}:4);
\draw [thick,rgb color={0,0,1}, rotate = {4.9315068493*-35}] (0:4) to [in={4.9315068493*25+180},out=180] ({4.9315068493 *25}:4);
\draw [thick,rgb color={0.166666666666667,0,1}, rotate = {4.9315068493*-36}] (0:4) to [in={4.9315068493*26+180},out=180] ({4.9315068493 *26}:4);
\draw [thick,rgb color={0.333333333333333,0,1}, rotate = {4.9315068493*-37}] (0:4) to [in={4.9315068493*27+180},out=180] ({4.9315068493 *27}:4);
\draw [thick,rgb color={0.5,0,1}, rotate = {4.9315068493*-38}] (0:4) to [in={4.9315068493*28+180},out=180] ({4.9315068493 *28}:4);
\draw [thick,rgb color={0.666666666666667,0,1}, rotate = {4.9315068493*-39}] (0:4) to [in={4.9315068493*29+180},out=180] ({4.9315068493 *29}:4);
\draw [thick,rgb color={0.833333333333333,0,1}, rotate = {4.9315068493*-40}] (0:4) to [in={4.9315068493*30+180},out=180] ({4.9315068493 *30}:4);
\draw [thick,rgb color={1,0,1}, rotate = {4.9315068493*-41}] (0:4) to [in={4.9315068493*31+180},out=180] ({4.9315068493 *31}:4);
\draw [thick,rgb color={1,0,0.833333333333333}, rotate = {4.9315068493*-42}] (0:4) to [in={4.9315068493*32+180},out=180] ({4.9315068493 *32}:4);
\draw [thick,rgb color={1,0,0.666666666666667}, rotate = {4.9315068493*-43}] (0:4) to [in={4.9315068493*33+180},out=180] ({4.9315068493 *33}:4);
\draw [thick,rgb color={1,0,0.5}, rotate = {4.9315068493*-44}] (0:4) to [in={4.9315068493*34+180},out=180] ({4.9315068493 *34}:4);
\draw [thick,rgb color={1,0,0.333333333333333}, rotate = {4.9315068493*-45}] (0:4) to [in={4.9315068493*35+180},out=180] ({4.9315068493 *35}:4);
\draw [thick,rgb color={1,0,0.166666666666667}, rotate = {4.9315068493*-46}] (0:4) to [in={4.9315068493*36+180},out=180] ({4.9315068493 *36}:4);

\foreach \x in {0,...,72}
{
\draw [fill=black] (4.9315068493*\x:4) circle [radius=0.05cm];
}


\def\diamin{4.25}
\def\diamout{4.5}
\def\diammid{4.5}
\def\Trad{4.9}
\def\Idiamin{5.1}
\def\Iradout{5.6}
\def\Idiammid{5.35}
\def\Irad{5.7}
\def\Ilimit{6.25}
\def\sizer{\small}

\draw [rotate=4.9315068493*-4]  (0:\diamin) -- (0:\diamout);
\draw [rotate=4.9315068493*0] (0:\diamin) -- (0:\diamout);
\draw [rotate=4.9315068493*-4]  (4.5,0) arc (0:4.9315068493*4:4.5);
\draw (4.9315068493*-2:\Trad) node {\sizer $V_0$};

\draw [rotate=4.9315068493*-4,black]  (0:\Idiamin) -- (0:\Iradout);
\draw [rotate=4.9315068493*-4,black] (0:\Ilimit) node {\sizer $0.9n$};
\draw [rotate=4.9315068493*3,black] (0:\Idiamin) -- (0:\Iradout);
\draw [rotate=4.9315068493*3,black] (0:\Ilimit) node {\sizer $0.83n$};
\draw [rotate=4.9315068493*-4,black]  (\Idiammid,0) arc (0:4.9315068493*7:\Idiammid);
\draw [black] (4.9315068493*-0.5:\Irad) node {\sizer $I_0$};

\draw [rotate=4.9315068493*9]  (0:\diamin) -- (0:\diamout);
\draw [rotate=4.9315068493*11] (0:\diamin) -- (0:\diamout);
\draw [rotate=4.9315068493*9]  (\diammid,0) arc (0:4.9315068493*2:\diammid);
\draw (4.9315068493*10:\Trad) node {\sizer $V_1$};
\draw [rotate=4.9315068493*8,black]  (0:\Idiamin) -- (0:\Iradout);
\draw [rotate=4.9315068493*8,black] (0:\Ilimit) node {\sizer $0.71n$};
\draw [rotate=4.9315068493*12,black] (0:\Idiamin) -- (0:\Iradout);
\draw [rotate=4.9315068493*12,black] (0:{\Ilimit-0.2}) node {\sizer $0.7n$};
\draw [rotate=4.9315068493*8,black]  (\Idiammid,0) arc (0:4.9315068493*4:\Idiammid);
\draw [black] (4.9315068493*10:\Irad) node {\sizer $I_1$};

\draw [rotate=4.9315068493*4,black]  (0:\Idiamin) -- (0:\Iradout);
\draw [rotate=4.9315068493*4,black] (0:\Ilimit) node {\sizer $0.82n$};
\draw [rotate=4.9315068493*24,black] (0:\Idiamin) -- (0:\Iradout);
\draw [rotate=4.9315068493*24,black] (0:{\Ilimit-0.3}) node {\sizer $1$};
\draw [rotate=4.9315068493*4,black]  (\Idiammid,0) arc (0:4.9315068493*20:\Idiammid);

\draw [rotate=4.9315068493*-10]  (0:\diamin) -- (0:\diamout);
\draw [rotate=4.9315068493*-8] (0:\diamin) -- (0:\diamout);
\draw [rotate=4.9315068493*-10]  (\diammid,0) arc (0:4.9315068493*2:\diammid);
\draw (4.9315068493*-9:\Trad) node {\sizer $V_2$};
\draw [rotate=4.9315068493*-11,black]  (0:\Idiamin) -- (0:\Iradout);
\draw [rotate=4.9315068493*-11,black] ($(0:\Ilimit)+(0,0.4)$) node {\sizer $0.92n$};
\draw [rotate=4.9315068493*-7,black] (0:\Idiamin) -- (0:\Iradout);
\draw [rotate=4.9315068493*-7,black] (0:\Ilimit) node {\sizer $0.91n$};
\draw [rotate=4.9315068493*-11,black]  (\Idiammid,0) arc (0:4.9315068493*4:\Idiammid);
\draw [black] (4.9315068493*-9:\Irad) node {\sizer $I_2$};

\draw [rotate=4.9315068493*-45,black]  (0:\Idiamin) -- (0:\Iradout);
\draw [rotate=4.9315068493*-45,black] (0:{\Ilimit-0.1}) node {\sizer $1.91n$};
\draw [rotate=4.9315068493*-12,black] (0:\Idiamin) -- (0:\Iradout);
\draw [rotate=4.9315068493*-12,black] (0:\Ilimit) node {\sizer $0.96n$};
\draw [rotate=4.9315068493*-45,black]  (\Idiammid,0) arc (0:4.9315068493*33:\Idiammid);

\end{tikzpicture}
}
\end{center}

\vspace{-0.6cm}

\caption{Embedding the tree in Case C.}\label{Figure_many_vertex}
\end{figure}

Now we are ready to prove Case C. Letting $T'$ be a tree $T$ in Case C with the neighbouring leaves of vertices with $\geq \log^4 n$ leaves removed, we embed $T'$ using the above lemma.
Then, much like in the proof Theorem~\ref{Theorem_one_large_vertex}, leaves are added to high degree vertices one at a time. In order to make sure we use all the colours, these leaves are located in particular intervals of the ND-colouring. In the proof, for technical reasons, we have three different kinds of high degree vertices (called type 1, type 2 and type 3 vertices), with vertices of a particular type having its neighbours chosen by a particular rule. (We note that we may have no type 2 vertices.) We embed the vertices of different types into 3 disjoint intervals.

\begin{proof}[Proof of Theorem~\ref{Theorem_case_C}]
See Figure~\ref{Figure_many_vertex} for an illustration of this proof. If $d_i\geq 2n/3$ for some $i$, then the result follows from Theorem~\ref{Theorem_one_large_vertex}.
Assume then that $v_1, \dots, v_{\ell}$ are ordered so that $2n/3\geq d_1\geq \dots\geq d_{\ell}$. Note that $d_1+\dots+d_{\ell}= n+1-|T'|\geq 99n/100+1$. Choose the smallest $m$ with $0.05n< d_1+ \dots+ d_m$. From minimality, we have $d_1+\dots+ d_m\leq d_1+0.05n$. Let $V_1=\{v_1, \dots ,v_m\}$, $V_2=\{v_{m+1}, \dots, v_{\ell}\}$, and $V_0=V(T')\setminus (V_1\cup V_2)$.
Let $I_0=[0.83n, 0.9n]$, $I_1=[0.7n, 0.71n]$, and $I_2=[0.91n, 0.92n]$. Then, it is easy to check that $|I_0| \geq 7|V_0|$, $|I_1|\geq 8|V_1|\log^3 n$ and $|I_2|\geq 8|V_2|\log^3 n$.

By Lemma~\ref{Lemma_small_tree}, there is a rainbow embedding of $T'$ in $I_0\cup I_1\cup I_2$ with $V_i$ embedded to $I_i$, for each $0 \leq i\leq 2$, and consecutive embedded vertices in $I_1$, and in $I_2$, within distance $16\log^3 n$ of each other. Say the copy of $T'$ from this is $S'$.

For odd  $m$, relabel vertices in $V_1$ so that, if $u_i$ is the image of $v_i$ under the embedding, then, for each $i\in [m]$, the vertices in the image of $V_1$ appear in $I_1$ in the following order:
\[
u_1,u_3,u_5,\ldots,u_{m},u_{m-1},u_{m-3},\ldots,u_2.
\]
For even $m$, relabel vertices in $V_1$ so that, if $u_i$ is the image of $v_i$ under the embedding, then, for each $i\in [m]$, the vertices in the image of $V_1$ appear in $I_1$ in the following order:
\[
u_1,u_3,u_5,\ldots,u_{m-1},u_{m},u_{m-2},\ldots,u_2.
\]
Relabel the vertices in $V_2$ so that, if $u_i$ is the image of $v_i$ under the embedding, then for each $i\in [\ell]\setminus [m]$, the vertices in the image of $V_2$ appear in $I_2$ in the order $u_{m+1},\ldots,u_\ell$. Finally, relabel the integers $d_i$, $i\in [\ell]$, to match this relabelling of the vertices $v_i$, $i\in [\ell]$. We uses this new labelling to embed neighbours of $u_i, i \leq \ell$ into three disjoint intervals (that are also disjoint from $V(S')$): $[1,u_1)$ if $i\leq m$ is odd, $(u_2,0.82n]$ if $i\leq m$ is even and $[0.96n,1.92n]$ if $m < i \leq \ell$.

For each $1\leq i\leq m$, if $i$ is odd we say that $u_i$ is \emph{type 1}, and if $i$ is even we say that $u_i$ is \emph{type 2}. For each $m<i\leq \ell$, we say that $u_i$ is \emph{type 3}. Note that, if $m=1$, then there are no type 2 vertices.
Furthermore, note that, if $m=1$, then $V(S')\subset \{u_1\}\cup I_0\cup I_2\subset \{u_1\}\cup [0.83n,0.92n]$, while, if $m>1$, then $V(S')\subset [u_1,u_2]\cup I_0\cup I_2\subset [u_1,u_2]\cup [0.83n,0.92n]$.

Let $C(K_{2n+1})\setminus C(S')=C_1\cup \ldots \cup C_\ell$ so that, for each $i\in [\ell]$, $|C_i|=d_i$, and, if $i<j$, then all the colours in $C_i$ are smaller than all the colours in $C_j$. Note that this is possible as $|C(K_{2n+1})\setminus C(S')|=n-|E(T')|=d_1+\ldots+d_\ell$.
Futhermore, note that, if $i,j\in [\ell]$ with $i\leq j-2$, then $\max(C_i)\leq \min (C_j)-d_{i+1}\leq \min(C_j)-\log^4 n$.

For each $i\in [\ell]$, do the following.
\begin{itemize}
\item If $u_i$ is a type 1 vertex, embed the vertices in $U_i:=\{u_i-c:c\in C_{i}\}$ as neighbours of $u_i$.
\item If $u_i$ is a type 2 or 3 vertex, embed the vertices in $U_i:=\{u_i+c:c\in C_{i}\}$ as neighbours of $u_i$.
\end{itemize}

Notice that, from the definition of the ND-colouring, we use edges with colour in $C_i$ to attach the neighbours of $u_i$, for each $i\in [\ell]$. Therefore, by design, the graph formed is rainbow. We need to show that the sets $U_i$, $i\in [\ell]$, are disjoint from each each other and from $V(S')$, so that we have a copy of $T$.

\begin{claim} If $m=1$, then $\max(C_m)\leq 0.68n$, while, if $m>1$, then $\max(C_m)\leq 0.11n$.
\end{claim}
\begin{proof}
If $m=1$, then, as at most $0.01n$ colours are used to embed $T'$ and $d_1\leq 2n/3$, we have $\max(C_m)\leq 0.67n+0.01n=0.68n$.
Suppose then that $m>1$, and hence $d_1<0.05n$. By the minimality in the choice of $m$, we have that $d_1+\ldots+d_m\leq  0.05n+0.05n=0.1n$. Therefore, $\max(C_m)\leq 0.1n+0.01n= 0.11n$.
\end{proof}

From the claim, as the vertices $u_1,\ldots,u_m$ are in $[0.7n,0.71n]$ and $\max(C_i) \leq \max (C_m)$, for each $i\in [m]$, for each type 1 vertex $u_{i}$, we have $U_{i}\subset [u_{i}-\max (C_i),u_i-\min(C_i)]\subset [0.02n,0.71n] \subset [1,0.71n]$. For each type 2 vertex $u_{i}$, we have $U_{i}\subset [u_{i}+\min (C_i),u_{i}+\max(C_i)]\subset [0.7n,0.82n]$.

Now, for each type 3 vertex $u_{i}$, we have, as $i>m$, that $\min(C_i)\geq 0.05n$. Therefore, as $\max(C_i)\leq n$, we have $U_{i}\subset [u_{i}+\min (C_i),u_{i}+\max(C_i)]\subset [0.96n,1.92n]$. Furthermore, by the ordering of $u_{m+1},\ldots,u_\ell$, for each $m+1\leq i<j\leq \ell$, we have $u_i<u_j$ and $\max(C_i)<\min(C_j)$ so that
$u_{i}+\max (C_i)< u_{j}+\min (C_j)$. Thus, the sets $U_i$, $m+1\leq i\leq \ell$ are disjoint sets in $[0.96n,1.92n]$.

\begin{claim}
If $u_{i}$ is a type 1 vertex with $i\leq m-2$, then $\min(U_i)< \max (U_{i+2})$.
If $u_{i}$ is a type 2 vertex with $i\leq m-2$, then $\max(U_i) < \min (U_{i+2})$.
\end{claim}
\begin{proof}
If $u_i$ is a type 1 vertex with $i\leq m-2$, then $u_{i+2}$ is also a type 1 vertex, which appears consecutively with $u_i$ in the interval $I_1$ in the embedding of $T'$, and hence $u_{i+2}-u_i\leq 16 \log^3 n<\log^4 n$. Therefore,
\[
\min(U_i)=u_{i}-\max (C_i)\geq u_i-\min(C_{i+2})+\log^4 n > u_{i+2}-\min(C_{i+2})=\max(U_{i+2}).
\]
On the other hand, if $u_i$ is a type 2 vertex with $i\leq m-2$, then $u_{i+2}$ is also a type 2 vertex, which appears consecutively with $u_i$ in the interval $I_1$ in the embedding of $T'$, and hence $u_{i}-u_{i+2}\leq 16 \log^3 n <\log^4 n$. Therefore,
\[
\max(U_i)=u_{i}+\max (C_i)\leq u_i+\min(C_{i+2})-\log^4 n < u_{i+2}+\min(C_{i+2})=\min(U_{i+2}).\qedhere
\]
\end{proof}

By the claim, the sets $U_i$, $i\leq m$ and $i$ is odd, are disjoint and lie in the interval $[1,\max(U_1)]\subset [1,u_1)$. If $m\geq 2$, then, by the claim again, the sets $U_i$, $i\leq m$ and $i$ is even, are disjoint and lie in the interval $[\min(U_2),0.82n]\subset (u_2,0.82n]$.

In summary, we have shown the following.
\begin{itemize}
\item The sets $U_i$, $i\leq m$ and $i$ is odd are disjoint sets in $[1,u_1)$.
\item If $m\geq 2 $, then the sets $U_i$, $i\leq m$ and $i$ is even, are disjoint sets in $(u_2,0.82n]$.
\item The sets $U_i$, $m<i\leq \ell$, are disjoint sets in $[0.96n,1.92n]$.
\item If $m=1$, then $V(S')\subset \{u_1\}\cup [0.83n,0.92n]$, while, if $m\geq 2$, then $V(S')\subset [u_1,u_2]\cup [0.83n,0.92n]$.
\end{itemize}
Thus, both when $m=1$ and when $m\geq 2$, we have that $U_i$, $i\in [\ell]$, are disjoint sets in $[2n+1]\setminus V(S')$, as required.
\end{proof}

\section{Concluding remarks}\label{sec:conc}
Many powerful techniques have been developed to decompose graphs into bounded degree graphs~\cite{bottcher2016approximate, messuti2016packing, ferber2017packing, kim2016blow,joos2016optimal}.
On the other hand, all these techniques encounter some barrier when dealing with trees with arbitrarily large degrees.
Having overcome this ``bounded degree barrier'' for Ringel's Conjecture, we hope that our ideas might be useful for other problems as well. Here we mention two such questions.

The closest relative of Ringel's conjecture is the following conjecture on graceful labellings, which is also mentioned in the introduction. This is a very natural problem to apply our techniques to.
\begin{conjecture}[K\"otzig-Ringel-Rosa, \cite{rosa1966certain}]\label{Conjecture_Graceful}
The vertices of every $n$ vertex tree $T$ can be labelled by the numbers $1, \dots, n$, such that  the differences $|u-v|$,  $uv\in E(T)$, are distinct.
\end{conjecture}

\noindent
This conjecture was proved for many isolated classes of trees like caterpillars, trees with $\leq 4$ leaves, firecrackers,  diameter $\leq 5$ trees,   symmetrical trees, trees with $\leq 35$ vertices, and olive trees (see Chapter 2 of \cite{gallian2009dynamic} and the references therein).
Conjecture~\ref{Conjecture_Graceful} is also known to hold asymptotically for trees of maximum degree at most $n/\log n$ \cite{adamaszek2016almost} but solving it for general trees, even asymptotically, is already wide open.

Another very interesting related problem is the G\'yarf\'as Tree Packing Conjecture. This also concerns decomposing a complete graph into specified trees, but the trees are allowed to be different from each other.
\begin{conjecture}[Gy\'arf\'as, \cite{gyarfas1978packing}]\label{Conjecture_Gyarfas}
Let $T_1, \dots, T_{n}$ be trees with $|T_i|=i$ for each $i\in [n]$. The edges of $K_n$ can be decomposed into $n$ trees which are isomorphic to  $T_1, \dots, T_{n}$ respectively.
\end{conjecture}
\noindent
This conjecture has been proved for bounded degree trees by Joos, Kim, K{\"u}hn and Osthus~\cite{joos2016optimal} but in general it is wide open.
It would be interesting to see if any of our techniques can be used here to make further progress.

\vspace{0.25cm}
\noindent
{\bf Acknowledgements.} Parts of this work were carried out when the first two authors visited the Institute for Mathematical Research (FIM) of ETH Zurich. We would like to thank FIM for its hospitality and for creating a stimulating research environment.

\bibliographystyle{abbrv}
\bibliography{proof-ringel-conjecture}

\end{document}